%% file: main.tex
\documentclass[graybox, envcountsame, envcountsect, leqno]{svmult}

\usepackage{type1cm}        
\usepackage{makeidx}         
\usepackage{graphicx}        
\usepackage{multicol}        
\usepackage[bottom]{footmisc}

\usepackage{newtxtext}       %
\usepackage[varvw]{newtxmath}       

\makeindex             

\input{head.tex}

\begin{document}

\title*{An abelian ambient category for behaviors in algebraic systems theory}
\author{Sebastian Posur}
\institute{Sebastian Posur \at University of Münster,
Fachbereich Mathematik und Informatik,
Einsteinstraße 62,
48149 Münster \email{sebastian.posur@uni-muenster.de}
}
%
%
\maketitle

\abstract{We describe an abelian category $\mathbf{ab}(M)$ in which the solution sets of finitely many linear equations over an arbitrary ring $R$ with values in an arbitrary left $R$-module $M$ reside as objects.
Such solution sets are also called behaviors in algebraic systems theory.
We both characterize $\mathbf{ab}(M)$ by a universal property and give a construction of $\mathbf{ab}(M)$ as a Serre quotient of the free abelian category generated by $R$.
We discuss features of $\mathbf{ab}(M)$ relevant in the context of algebraic systems theory: if $R$ is left coherent and $M$ is an fp-injective fp-cogenerator, then $\mathbf{ab}(M)$ is antiequivalent to the category of finitely presented left $R$-modules. This provides an alternative point of view to the important module-behavior duality in algebraic systems theory. We also obtain a dual statement: if $R$ is right coherent and $M$ is fp-faithfully flat, then $\mathbf{ab}(M)$ is equivalent to the category of finitely presented right $R$-modules. As an example application, we discuss delay-differential systems with constant coefficients and a polynomial signal space. Moreover, we propose definitions of controllability and observability in our setup.}



\section{Introduction}
\input{introduction}

\section{Notations and conventions}
\input{notation}

\section{Behaviors with latent variables}
\input{behaviors}

\section{The functor of trajectories}
\input{trajectories}

\section{The category of finitely presented functors on finitely presented modules}\label{section:category_rmodmod}
\input{fpfunctors}

\section{An abelian ambient category for behaviors}
\input{ambientcat}

\section{An example: delay-differential systems}\label{section:example}
\input{example}

\section{Controllability and observability}\label{section:controllability}
\input{controllability}

\section*{Appendix}
\stepcounter{section}
\addcontentsline{toc}{section}{Appendix}
\input{appendix}

\begin{acknowledgement}
The author would like to thank Alban Quadrat for his personal introduction to algebraic systems theory and for posing the central question of this paper.
The author would also like to thank Alex Martsinkovsky for his personal introduction to the theory of finitely presented functors. Without both of these introductions, this paper could obviously not have been written by the author.
\end{acknowledgement}

\input{main.bbl}
\end{document}

%% file: head.tex
\usepackage{tikz}
\usetikzlibrary{
automata,shapes,arrows,matrix,backgrounds,positioning,plotmarks,calc,patterns,matrix,decorations.pathreplacing,decorations.pathmorphing,decorations.text,decorations.markings
}

\spnewtheorem{convention}[theorem]{Convention}{\itshape}{\rmfamily}
\spnewtheorem{notation}[theorem]{Notation}{\itshape}{\rmfamily}
\spnewtheorem{construction}[theorem]{Construction}{\bfseries}{\rmfamily}
\spnewtheorem{para}[theorem]{\nocaption}{\bfseries}{\rmfamily}
\spnewtheorem{ex}[theorem]{Example}{\bfseries}{\rmfamily}

\spnewtheorem{lemmadef}[theorem]{Lemma and Definition}{\bfseries}{\itshape}

\DeclareMathOperator{\End}{\mathrm{End}}
\DeclareMathOperator{\Hom}{\mathrm{Hom}}
\DeclareMathOperator{\Ext}{\mathrm{Ext}}
\DeclareMathOperator{\kernel}{\mathrm{ker}}
\DeclareMathOperator{\cokernel}{\mathrm{cok}}
\DeclareMathOperator{\image}{\mathrm{im}}

\newcommand{\id}{\mathrm{id}}
\newcommand{\op}{\mathrm{op}}

\newcommand{\Modl}{\text{-}\mathbf{Mod}}
\newcommand{\Modr}{\mathbf{Mod}\text{-}}
\newcommand{\modl}{\text{-}\mathbf{mod}}
\newcommand{\modr}{\mathbf{mod}\text{-}}
\newcommand{\Ab}{\mathbf{Ab}}

\newcommand{\Beh}[2]{\mathbf{ab}({#2})}

\newcommand{\yoneda}{\mathrm{Yoneda}}
\newcommand{\defect}{\mathrm{Defect}}
\newcommand{\codefect}{\mathrm{Covdefect}}
\newcommand{\ev}[1]{\mathrm{Eval}_{#1}}

\newcommand{\asBeh}[2]{{#1}[{#2}]}
\newcommand{\asBehmor}[2]{{#1}_{#2}}
\newcommand{\ringmap}{\sigma}

\newcommand{\Nzero}{\mathbb{Z}_{\geq 0}}
\newcommand{\N}{\mathbb{N}}
\newcommand{\Z}{\mathbb{Z}}

\newcommand{\R}{\mathbb{R}}

\newcommand{\K}{\mathbb{K}}

\newcommand{\Quo}{\mathbf{Quo}}
\newcommand{\Sub}{\mathbf{Sub}}
\newcommand{\AC}{\mathbf{A}}
\newcommand{\BC}{\mathbf{B}}

\newcommand{\CC}{\mathbf{C}}
\newcommand{\DC}{\mathbf{D}}

\newcommand{\ff}{\mathcal{F}}

\newcommand{\M}{M}

\newcommand{\pmatcol}[2]{ \begin{pmatrix}{#1} \\ {#2} \end{pmatrix} }

%% file: introduction.tex
This paper deals with the following question:
What is a ``good''  notion of a homomorphism between solution sets of finitely many linear equations over a ring $R$ with values in an $R$-module $M$, i.e., sets of the form
\begin{equation}\label{equation:solution_set}
\{ x \in M^{n \times 1} \mid Ax = 0\}    
\end{equation}
for $A \in R^{m \times n}$ a matrix with entries in $R$, and $m,n \in \Nzero$.
We can slightly reframe this question: What is a ``good'' category in which such solution sets reside as objects, i.e., a ``good'' ambient category for such solution sets? We emphasize that we ask this question for an arbitrary module $M$ over an arbitrary ring $R$, in particular, we do not assume $R$ to be commutative.
Besides its fundamental character which makes this question interesting in its own right, it arises naturally in algebraic systems theory.
For example, in the case of dynamical systems, $R$ is usually a ring of differential/difference operators, $M$ is some $R$-module of functions (e.g., a module
of smooth functions or distributions), and consequently, the objects of interest are solution sets of linear differential/difference equations. See, e.g., \cite{Robertz15} for an introduction to the so-called algebraic analysis approach to linear systems.

Following the terminology coined by Willems in his behavioral approach to systems and control theory \cite{WilParadigms}, we refer to solutions sets as in \eqref{equation:solution_set} by the term \emph{behaviors} and call the module $M$ a \emph{signal space}.
In this paper, we will even adhere to this terminology in the context of arbitrary modules over arbitrary rings, since the context of algebraic systems theory is the main motivating factor of this paper.

Malgrange \cite{MalgrangeIsomorphism} observed that we can identify a behavior as in \eqref{equation:solution_set} with the set of $R$-module homomorphisms
\begin{equation}\label{equation:solution_set_as_hom}
\Hom( \cokernel( A ), M )
\end{equation}
in a natural way (see \ref{lemma:malgrange_iso} for details).
Here, $\cokernel( A )$ denotes the finitely presented $R$-module that arises as the cokernel of the morphism $R^{1 \times m} \rightarrow R^{1 \times n}$ between free $R$-modules defined by the matrix $A$.
The functoriality of $\Hom$ in both of its components makes the following two structural aspects of behaviors immediately evident:
\begin{itemize}
    \item $\Hom( \cokernel( A ), M )$ is an $\End(M)$-module, where $\End(M)$ denotes the ring of $R$-module endomorphisms of $M$.
    \item A morphism between finitely presented $R$-modules $\cokernel( A ) \rightarrow \cokernel( B )$ gives rise to a map between behaviors
    \begin{equation}\label{equation:hom_between_homsets}
    \Hom( \cokernel( B ), M ) \rightarrow \Hom( \cokernel( A ), M )
    \end{equation}
    in a contravariant way, where $B$ denotes a matrix over $R$.
\end{itemize}
Oberst uses this first structural aspect as the decisive feature of behaviors\footnote{Oberst formulates his ideas in \cite{Ob} in the context where $R$ is a noetherian commutative ring.}. Accordingly, he defines homomorphisms of behaviors as $\End(M)$-module homomorphisms and consequently regards the category $\End(M)\Modl$ of all $\End(M)$-modules as an appropriate ambient category for behaviors \cite[Definition and Corollary 15]{Ob}.

The map in \eqref{equation:hom_between_homsets} is an $\End(M)$-module homomorphism by the functoriality of $\Hom$. It follows that we may regard $\Hom(-,M)$ as a functor of type
\begin{equation}\label{equation:hom_M_with_new_type}
(R\modl)^{\op} \rightarrow \End(M)\Modl
\end{equation}
where $R\modl$ denotes the category of finitely presented $R$-modules.
Oberst proves that this functor is fully faithful (and hence induces an equivalence with its essential image) if the signal space $M$ is a large injective cogenerator. Here, large means that every $N \in R\modl$ can be embedded into a \emph{finite} direct sum of $M$.

A remarkable example of a large injective cogenerator is given by smooth real functions $\mathcal{C}^{\infty}(\R^n,\R)$ regarded as a module over the polynomial ring $\R[\partial_1, \dots, \partial_n]$ whose $n \in \Nzero$ indeterminates act via partial derivations. 
Its injectivity was proven by Malgrange in \cite{MalgrangePhD}, Oberst proved its largeness in \cite{Ob}.
See also \cite{ObFr},\cite{Zer06} for further examples relevant to algebraic systems theory.

The functor in \eqref{equation:hom_M_with_new_type} has moreover motivated the development of computer algebra tools for performing sophisticated computations with $R$-modules, e.g., 
the computation of torsion submodules of modules over Ore algebras \cite{CQR05}, or the computation of so-called grade filtrations \cite{QGrade}, \cite{barhabil}.
In fact, an algorithmic treatment of $R\modl$ (for various rings) has been realized by several computer algebra packages \cite{OreModules07}, \cite{homalg-project}, \cite{CAP-project} with the goal of performing effective computations which are helpful in an algorithmic study of structural properties of linear systems.

Computationally relevant aspects of behaviors can already be dealt with when $M$ is an injective cogenerator that is not necessarily large. In \cite{WoodKey}, Wood explains\footnote{Wood works in the context where $R$ is a left and right noetherian domain.}: 
\begin{itemize}
    \item The cogenerator property ensures that containment of one behavior within another can be decided by mere computations within $(R\modl)^{\op}$ \cite[Lemma 3.5]{WoodKey}. We provide a purely categorical explanation of this fact in \ref{remark:first_key_problem}.
    \item The injectivity ensures that we can solve the elimination (of quantifiers) problem by mere computations within $(R\modl)^{\op}$. Here, the elimination problem poses the question if all sets of the form
    \begin{equation}\label{equation:solution_set_pp}
    \{ x \in M^{k\times 1} \mid \exists x'~ A \pmatcol{x}{x'} = 0 \}    
    \end{equation}
    for $0 \leq k \leq n$ can be rewritten without the existential quantifier, i.e., is of the form as in \eqref{equation:solution_set} \cite[Lemma 4.2]{WoodKey}.
\end{itemize}
The property of $M$ being ``large'', on the other hand, does not seem to play a crucial role in an algorithmic treatment of behaviors. From a theoretical point of view, it merely ensures that the functor in \eqref{equation:hom_M_with_new_type} is full.

When $M$ is not an injective cogenerator, the functor in \eqref{equation:hom_M_with_new_type} need neither be faithful nor right exact, which limits the ability to draw conclusions about behaviors from mere computations within $(R\modl)^{\op}$, see also \cite{WoodKey}.

In this paper we propose an ambient category for behaviors, denoted by $\Beh{R}{M}$, for an arbitrary module $M$ over an arbitrary ring $R$.
We characterize $\Beh{R}{M}$ as the category which is universal (initial) among all categories which satisfy the following three specifications:
\begin{enumerate}
    \item The category $\Beh{R}{M}$ is abelian.
    \item We have a faithful and exact functor
    \[
    \Beh{R}{M} \rightarrow \Ab
    \]
    into the category of abelian groups. 
    \item There is a designated object $\asBeh{\ff}{M} \in \Beh{R}{M}$ with designated endomorphisms \[\asBeh{\ff}{M} \xrightarrow{r}  \asBeh{\ff}{M}\] for all $r \in R$. These endomorphisms map to 
    \[
    M \xrightarrow{x \mapsto rx} M
    \]
    via the functor $\Beh{R}{M} \rightarrow \Ab$ of the second specification.
\end{enumerate}
See \ref{theorem:up_of_beh} for a precise formulation of the universal property of $\Beh{R}{M}$.

This is how we should think about $\Beh{R}{M}$ and the three specifications:
We call the objects in $\Beh{R}{M}$ \emph{abstract behaviors}.
We think of the functor in the second specification as ``taking the underlying abelian group'' of an abstract behavior.
We call $\asBeh{\ff}{M}$ the \emph{abstract signal space} since its underlying abelian group is given by the signal space due to the third specification.
Since $\Beh{R}{M}$ has finite direct sums and using the third specification, we see that any matrix $A \in R^{m \times n}$ gives rise to an endomorphism
\begin{equation}\label{equation:map_abstract_signal_space}
\asBeh{\ff}{M}^{\oplus n} \xrightarrow{A} \asBeh{\ff}{M}^{\oplus m} 
\end{equation}
whose underlying abelian group map is given by
\begin{equation}\label{equation:underlying_map}
M^{n \times 1} \rightarrow
{M}^{m \times 1}: x \mapsto Ax.
\end{equation}
The kernel of the map in \eqref{equation:underlying_map} is exactly the behavior in \eqref{equation:solution_set}. It follows that we should think about the kernel of the morphism in \eqref{equation:map_abstract_signal_space}, which exists by the first specification, as the abstract behavior whose underlying abelian group is the behavior of \eqref{equation:solution_set}.
Since abelian categories also possess cokernels and images, we furthermore have abstract behaviors within $\Beh{R}{M}$ whose underlying abelian groups are given as in \eqref{equation:solution_set_pp} or by quotients of two such abelian groups.

This paper provides a construction of $\Beh{R}{M}$ in Subsection \ref{subsection:construction_of_behM} and a proof that this construction satisfies the universal property in \ref{theorem:up_of_beh}.
Moreover, we prove the following equivalence of categories
\[
\Beh{R}{\M} \simeq (R\modl)^{\op}
\]
in the case where $R$ is a left coherent ring and $\M \in R\Modl$ is fp-injective and an fp-cogenerator (Definition \ref{definition:fpcogenerator}). In particular, we do not assume $M$ to be a large cogenerator. In this sense, we recover the famous module-behavior duality of algebraic systems theory.
We note that coherent rings occur quite naturally in algebraic analysis, e.g., the Hardy algebra is coherent \cite{HardyAlg} and has been used to study stabilization problems in control theory (see, e.g., \cite{Zames81}).
Moreover, we emphasize again that $R$ does not need to be commutative and $R\modl$ is not necessarily equivalent\footnote{A typical situation in which $R\modl$ is equivalent to $\modr R$ is given whenever we have a ring isomorphism $R \cong R^{\op}$, which is also called an \emph{involution}. Such involutions are an important computational tool in algebraic analysis, see, e.g., \cite[Remark 3.11]{Robertz15}.} to $\modr R$.

Moreover, we obtain a dual picture. If $R$ is a right coherent ring and $\M \in R\Modl$ is an fp-faithfully flat module (Definition \ref{definition:fp_faithfully_flat}), then we have an equivalence
\[
\Beh{R}{\M} \simeq \modr R
\]
which we prove in \ref{corollary:equivalence_beh_modR}.
These results also appear in an exposition by Prest \cite[Proposition 7.1, Proposition 7.2]{prest2012categories} with a particular emphasis on the so-called definable subcategories corresponding to $\Beh{R}{\M}$ in both cases.

We compare our setup with the one by Oberst in \ref{example:comparision_with_oberst}: first, the category $\End(M)\Modl$ also satisfies the three specifications given above. Consequently, the universal property of $\Beh{R}{\M}$ provides a faithful comparison functor
\[
\Beh{R}{\M} \rightarrow \End(M)\Modl
\]
which commutes with taking the underlying abelian groups, but which is not full in general. Hence, our setup admits fewer morphisms than Oberst's setup. This complies with the observation that $M$ being ``large'' should not play an essential role in our theory.

In Section \ref{section:example}, we give an example application: we study delay-differential systems with constant coefficients and a polynomial signal space. In this example, the signal space is not injective, however, we nevertheless manage to describe the corresponding ambient abelian category for behaviors.

In Section \ref{section:controllability} we discuss the notions \emph{controllability} and \emph{observability} in our proposed setup .

The construction of $\Beh{R}{M}$ builds both on the theory of finitely presented functors and Serre quotients.
The theory of finitely presented functors was introduced by Maurice Auslander in \cite{A}.
A functor of type $\AC \rightarrow \Ab$ for $\AC$ an additive category is called finitely presented if it arises as the cokernel of a natural transformation between representable functors. 
In this paper, we are mainly interested in finitely presented functors of type $R\modl \rightarrow \Ab$, and we denote their category by $R\modl\modl$.
The relevance of $R\modl\modl$ lies in the fact that it can be seen as the abelian category freely generated by $R$ (Theorem \ref{theorem:up_Rmodmod}).
The category $R\modl\modl$ has been studied and used by many authors, see, e.g., \cite{AusFun}, \cite{Beligiannis2000},\cite{HerzogContravariant}, \cite{HID14}.  We recall several features of $R\modl\modl$ in Section \ref{section:category_rmodmod}.

We construct $\Beh{R}{M}$ as a Serre quotient of $R\modl\modl$ by modding out those finitely presented functors whose evaluation at $M$ yields zero. 
The resulting category has already been introduced by Prest in \cite{PrestSurvey}, where he argues that this category should even be considered as an appropriate replacement of $M$ itself.
Serre quotients of $R\modl\modl$ were also studied as an alternative language for dealing with the model theory of modules, see \cite{PrestPSL} for details.

The main contribution of this paper is that it highlights the synergies between the theory of finitely presented functors and algebraic systems theory. To the best of the author's knowledge, such a link between both theories has not been studied before in the literature.

%% file: notation.tex
\begin{convention}
All functors between (pre)additive categories are supposed to be additive.
\end{convention}

\begin{notation}
Throughout our paper, $R$ denotes an arbitrary ring. Whenever we need special properties for $R$, we will make them explicit.
\end{notation}

\begin{notation}[Matrices as module homomorphisms]
For a ring $R$ and $n \in \Nzero$, we denote by $R^{1 \times n}$ the \emph{left} free $R$-module of rows with $n$ entries. Likewise, we denote by $R^{n \times 1}$ the \emph{right} free $R$-module of columns with $n$ entries. 
A matrix $A \in R^{m \times n}$ gives rise to both a morphism between free left modules
\[
R^{1 \times m} \xrightarrow{A} R^{1 \times n}: x \mapsto x\cdot A
\]
and between free right modules
\[
R^{n \times 1} \xrightarrow{A} R^{m \times 1}: x \mapsto A\cdot x.
\]
By abuse of notation, we will refer to both of these morphisms by $A$.
\end{notation}

\begin{notation}[Categories]
All categories are printed in bold. Moreover, we make use of the following categories:
\begin{itemize}
    \item $\Ab$: the category of abelian groups.
\end{itemize}
For $R$ a ring, we have:
\begin{itemize}
    \item $R\Modl$: the category of all left $R$-modules.
    \item $\Modr R$: the category of all right $R$-modules.
    \item $R\modl$: the category of finitely presented left $R$-modules.
    \item $\modr R$: the category of finitely presented right $R$-modules.
\end{itemize}
We write ${_R}M$ if we wish to highlight that $M \in R\Modl$ is a left module over $R$.
Likewise, we write $M{_R}$ if we wish to highlight that $M \in \Modr R$ is a right module over $R$.
\end{notation}

\begin{notation}[Morphisms in a category]
Let $\CC$ be a category, and let $A,B \in \CC$ be two of its objects. We denote the set of homomorphisms from $A$ to $B$ by $\Hom_{\CC}(A,B)$, or simply $\Hom(A,B)$ if the category is clear from the context.
We also say that a morphism $\alpha \in \Hom_{\CC}(A,B)$ is of \emph{type} $A \rightarrow B$.
We regard certain morphisms of a particular type as being \textbf{canonical}. 
For example, if $\alpha: A \rightarrow B$ is a morphism in an abelian category, then we regard its kernel embedding as the canonical morphism of type $\kernel( \alpha ) \rightarrow A$, and its cokernel projection as the canonical morphism of type $B \rightarrow \cokernel( \alpha )$.
We will mark several morphisms (in particular functors) as being canonical within this paper. And whenever we refer to a particular morphism by simply providing its type, we always mean to refer to the canonical morphism of that type.
Furthermore, we use the abbreviations mono/epi/iso for monomorphism/epimorphism/isomorphism.
\end{notation}

\begin{notation}[Functors]
Let $R$ be a ring and $M \in R\Modl$. We denote its hom functor by
\[
\Hom( M, - ): R\Modl \rightarrow \Ab
\]
and its tensor functor (over $R$) by
\[
(- \otimes M): \Modr R \rightarrow \Ab.
\]
For $N \in \Modr R$, we denote its tensor functor (over $R$) by
\[
(N \otimes -): R\Modl \rightarrow \Ab.
\]
The forgetful functor
\[
\ff: R\Modl \longrightarrow \Ab
\]
that maps an $R$-module $M$ to its underlying abelian group is of special importance. We use the symbol $\ff$ exclusively for the forgetful functor in this paper, in particular, we never use this symbol as a variable denoting any other functor but the forgetful one.
\end{notation}

%% file: behaviors.tex
In this section, we introduce behaviors with latent variables.

\begin{definition}[pp formula]
Let $R$ be a ring, $m,n \in \Nzero$ and $k \in \{0, \dots, n\}$.
Suppose given two matrices $B \in R^{m \times k}$, $B' \in R^{m \times (n-k)}$.
We call a formula with $k$ \textbf{free variables} 
\[
x = (x_1, \dots, x_k)^{\mathrm{tr}}
\]
and $n - k$ \textbf{bound variables} 
\[
x' = (x_{k+1}, \dots, x_{n})^{\mathrm{tr}}
\]
of the following form
\begin{equation}\label{equation:ppformula}
\exists x'~ A \pmatcol{x}{x'} = 0
\end{equation}
the \textbf{pp formula\footnote{pp stands for positive primitive} over $R$}
defined by $B$ and $B'$, where $A := (B|B') \in R^{m \times n}$ is the concatenated matrix of $B$ and $B'$.
We also write
\[
{_B}\phi_{B'}
\]
in order to refer to such a pp formula.
If all variables are free, i.e., no variables are bound, i.e., $k = n$, then we speak of a \textbf{quantifier-free pp formula} since it is simply given by the following system of linear equations:
\[
A x = 0.
\]
We also write
\[
{_A}\phi
\]
in order to refer to such a quantifier-free pp formula.
\end{definition}

\begin{para}
Suppose given a pp formula ${_B}\phi_{B'}$.
Since the matrix $A = (B|B')$ acts from the left on the column $\pmatcol{x}{x'}$, a natural context for evaluating $\phi$ is provided by the elements of a left $R$-module $M$. For a tuple $\omega \in M^{k \times 1}$, we set $\phi( \omega )$ as true if
\[
\exists \omega' \in M^{(n-k) \times 1}: \big(A \pmatcol{\omega}{\omega'} = 0\big)
\]
holds, and as false otherwise.
\end{para}

\begin{definition}\label{def:behavior}
Let $M$ be an $R$-module.
A set of the form
\[
\mathcal{B}( \phi, M ) \coloneqq \{ \omega \in M^{k \times 1} \mid \phi( \omega ) \}
\]
where $\phi(x_1, \dots, x_k)$ is a pp formula is called a \textbf{behavior (over $M$)}. 
The bound variables of $\phi$ are also refered to as the \textbf{latent variables} of the behavior $\mathcal{B}( \phi, M )$.
We call an element in $\mathcal{B}( \phi, M )$ a \textbf{trajectory}.
In this context, the module $M$ is also called the \textbf{signal space}.
\end{definition}

\begin{remark}
By using the names \emph{behavior}, \emph{latent variables}, \emph{trajectory}, and \emph{signal space} for the abstract mathematical objects in \ref{def:behavior}, we make use of the terminology from Willems's approach to systems theory \cite{WilParadigms}.
We remark that Willems also refers to the triple $(M,k,\mathcal{B}( \phi, M ))$ as a \textbf{system}.
We also remark that Willems thinks of the behavior as a subobject of $M^{k \times 1}$, i.e., from the categorical point of view, it comes equipped with its embedding $\mathcal{B}( \phi, M ) \hookrightarrow M^{k \times 1}$.
This is why the mere object $\mathcal{B}( \phi, M )$ is sometimes refered to as an \textbf{abstract behavior} \cite[Page 418]{LomadzeBeh}, i.e., a behavior without its embedding.
We will formally introduce the term \emph{abstract behaviors} for the objects in our abelian ambient category for behaviors in \ref{remark:abstract_names}.
\end{remark}

We give several examples from which the concrete terminology introduced in \ref{def:behavior} originates.

\begin{ex}[Linear ordinary differential equations with constant coefficients]\label{ex:linear_odes}
Let $R := \R[\partial]$ denote the ring of polynomials in one indeterminate $\partial$ with real coefficients.
Then the set of smooth functions $M := \mathcal{C}^{\infty}(\R, \R)$ from $\R$ to $\R$ can be endowed with the structure of an $R$-module by letting $\partial$ act on $f \in M$ via the derivative
\[
\partial f(t):= \frac{df(t)}{dt}.
\]
Now, a linear ordinary differential equation with constant coefficients can be encoded by a quantifier-free pp formula ${_A}\phi$ of the form
\[
A \cdot x = 0
\]
where $A \in R^{m \times n}$.
Its set of solutions within $M$ is given by $\mathcal{B}( {_A}\phi, M )$.
In this example, elements of $\mathcal{B}( {_A}\phi, M )$ can be directly interpreted as trajectories, i.e., as paths of some point moving in time within $\R^n$. Moreover, the subset $\mathcal{B}( {_A}\phi, M ) \subset M^n$ can be interpreted as defining all admissible paths that can be taken within our system. In that sense, it describes the system's behavior.
\end{ex}

\begin{ex}[Linear ordinary difference equations with constant coefficients]\label{ex:linear_odiffes}
Let $\K$ be an arbitrary field and let $R := \K[\sigma]$ denote the ring of polynomials in one indeterminate $\sigma$ with coefficients in $\K$.
Let $T \in \{ \N, \Z \}$. Then $M := \K^T$, i.e., the set of functions of type $T \rightarrow \K$, can be endowed with the structure of an $R$-module by letting $\sigma$ act on $f \in M$ via the unit shift
\[
\sigma f(t) := f(t + 1).
\]
Now, a linear ordinary difference equation can be encoded by a quantifier-free pp formula ${_A}\phi$ of the form
\[
A \cdot x = 0
\]
where $A \in R^{m \times n}$.
Again, elements in the solution set $\mathcal{B}( {_A}\phi, M ) \subset M^n$ can directly be intepreted as (discrete) trajectories.
\end{ex}

\begin{ex}\label{ex:pdes}
Both \ref{ex:linear_odes} and  \ref{ex:linear_odiffes} can be generalized to the case of several variables:
\begin{itemize}
    \item The commutative polynomial ring $R := \R[\partial_1, \dots, \partial_n]$ in $n \in \Nzero$ indeterminates acts on $M := \mathcal{C}^{\infty}(\R^n,\R)$ via partial derivations, i.e.,
    \[
    \partial_i(f) := \frac{\partial f}{\partial x_i}
    \]
    for $i = 1, \dots n$, where $f \in M$ is a function with arguments $x_1, \dots,x_n$.
    This leads to linear partial differential equations with constant coefficients.
    \item $R := \K[\sigma_1, \dots, \sigma_n]$ acts on $M := \K^{T^n}$ via unit shifts, i.e.,
    \[
    \sigma_i(f) := f(x_1, \dots, x_{i-1}, x_i + 1, x_{i+1}, \dots, x_n)
    \]
    for $i = 1, \dots n$, where $f \in M$ is a function with arguments $x_1, \dots,x_n$, $\K$ is an arbitrary field, and $T \in \{ \N, \Z\}$.
    This leads to linear partial difference equations with constant coefficients.
\end{itemize}
\end{ex}

\begin{para}
A behavior $\mathcal{B}( \phi, M )$ is always an abelian subgroup of $M^{k \times 1}$.
If $R$ is commutative, then $\mathcal{B}( \phi, M )$ is even an $R$-submodule of $M^{k \times 1}$, however, for non-commutative rings, this is not true in general.
\end{para}

\begin{para}\label{para:question_mor_of_beh}
The following natural question arises: what is an appropriate notion of a morphism between two behaviors over $M$? Treating behaviors merely as abelian groups means forgetting too much of their defining context. It is a main goal of this paper to describe an ambient category of behaviors over $M$.
\end{para}

%% file: trajectories.tex
In this section, we discuss consequences of the functoriality of behaviors in $M$.

\begin{construction}[Behaviors are functorial in $M$]
Let $\phi(x_1, \dots, x_k)$ be a pp formula over $R$. We construct a functor of the form
\[
\mathcal{B}( \phi, - ): R\Modl \longrightarrow \Ab.
\]
This functor maps an object $M \in R\Modl$ to its behavior $\mathcal{B}( \phi, M )$. 
A morphism $\alpha: M \rightarrow N$ in $R\Modl$ is mapped to the function
\begin{align*}
\mathcal{B}( \phi, M ) &\rightarrow \mathcal{B}( \phi, N ) \\
(\omega_1, \dots, \omega_k) &\mapsto (\alpha(\omega_1), \dots, \alpha(\omega_k)).
\end{align*}
This function is well-defined and additive since $\alpha$ is $R$-linear. We call the functor $\mathcal{B}( \phi, - )$ the \textbf{functor of trajectories} (defined by $\phi$).
\end{construction}

\begin{para}
The name \emph{functor of trajectories} is chosen in analogy to the well-known \emph{functor of points} in algebraic geometry, which maps a commutative ring $K$ to the $K$-valued solutions of a system of polynomial equations.
\end{para}

\begin{para}
The evident notion of a morphism between two functors of trajectories $\mathcal{B}( \phi, - )$ and $\mathcal{B}( \phi', - )$ is given by the notion of a natural transformation between functors.
\end{para}

\begin{ex}
The forgetful functor
\[
\ff: R\Modl \longrightarrow \Ab
\]
that maps an $R$-module $M$ to its underlying abelian group is a functor of trajectories, defined by the following pp formula with a single free variable:
\[
0\cdot x = 0.
\]
We have the following natural isomorphisms for the forgetful functor:
\[
\ff \simeq \Hom( {_R}R, - ) \simeq (R_{R} \otimes -).
\]
\end{ex}

\begin{notation}
Since the forgetful functor will play such a prominent role in this paper, we are going to use the symbol $\ff$ exclusively as its notation.
\end{notation}

\begin{ex}[Natural endomorphisms of $\ff$]\label{example:end_of_ff}
An element $r \in R$ defines the natural endomorphism $\ff \rightarrow \ff$ whose component at $M \in R\Modl$ is given by\begin{align*}
\ff(M) &\rightarrow \ff(M) \\
m &\mapsto rm.
\end{align*}
By abuse of notation, we denote this natural transformation by $\ff \xrightarrow{r} \ff$. It follows from Yoneda's lemma that all endomorphisms of $\ff$ are of this form\footnote{Yoneda's lemma gives us the bijection $\End( \Hom( {_R}R, - ) ) \simeq \End( _{R}R )$.}.
More generally, a matrix $A \in R^{m \times n}$ with $m,n \in \Nzero$ gives rise to a natural transformation $\ff^{\oplus n} \rightarrow \ff^{\oplus m}$ whose component at the module $M$ is given by the morphism of abelian groups
\begin{align*}
\ff(M)^{\oplus n} \rightarrow \ff(M)^{\oplus m}: x \mapsto A \cdot x.
\end{align*}
By abuse of notation, we denote this natural transformation again by $\ff^{\oplus n} \xrightarrow{A} \ff^{\oplus m}$.
Yoneda's lemma implies that all natural transformations of type $\ff^{\oplus n} \rightarrow \ff^{\oplus m}$ are of this form.
\end{ex}

\begin{para}\label{para:beh_canonical_inclusion}
We have a \emph{canonical} inclusion of a given functor of trajectories into the $k$-fold finite direct sum of the forgetful functor, where $k$ is the number of free variables of the defining pp formula $\phi$:
\begin{align*}
\mathcal{B}( \phi, - ) \hookrightarrow \ff^{\oplus k}.
\end{align*}
If $\phi = {_A}\phi$ is quantifier-free and defined by the matrix $A \in R^{m \times n}$, then we even have an exact sequence of functors
\begin{equation}\label{equation:ex_seq_of_beh}
   \begin{tikzpicture}[label/.style={postaction={
        decorate,
        decoration={markings, mark=at position .5 with \node #1;}},
        mylabel/.style={thick, draw=none, align=center, minimum width=0.5cm, minimum height=0.5cm,fill=white}}]
        \coordinate (r) at (2,0);
        \node (A) {$0$};
        \node (B) at ($(A)+(r)$) {$\mathcal{B}( {_A}\phi, - )$};
        \node (C) at ($(B) + (r)$) {$\ff^{\oplus n}$};
        \node (D) at ($(C) + (r)$) {$\ff^{\oplus m}$.};
        \draw[->,thick] (A) to  (B);
        \draw[->,thick] (B) to (C);
        \draw[->,thick] (C) to node[above]{$A$} (D);
  \end{tikzpicture}
\end{equation}
Here, exactness means that this sequence of natural transformations yields an exact sequence of abelian groups for each evaluation at an $M \in R\Modl$.
\end{para}

\begin{lemma}[Malgrange isomorphism \cite{MalgrangeIsomorphism}]\label{lemma:malgrange_iso} Suppose given a quantifier-free pp formula ${_A}\phi$ defined by the matrix $A \in R^{m \times n}$.
Then we have a natural isomorphism of functors
\[
\Hom(\cokernel(A),-) \simeq \mathcal{B}( {_A}\phi, - ),
\]
i.e., the functor of trajectories defined by ${_A}\phi$ is representable by $\cokernel(A)$, the cokernel of the module homomorphism $R^{1 \times m} \xrightarrow{A} R^{1 \times n}$.
\end{lemma}
\begin{proof}
We have the following short exact sequence in $R\Modl$:
\begin{center}
       \begin{tikzpicture}[label/.style={postaction={
        decorate,
        decoration={markings, mark=at position .5 with \node #1;}},
        mylabel/.style={thick, draw=none, align=center, minimum width=0.5cm, minimum height=0.5cm,fill=white}}]
        \coordinate (r) at (3,0);
        \node (A) {$R^{1 \times m}$};
        \node (B) at ($(A)+(r)$) {$R^{1 \times n}$};
        \node (C) at ($(B) + (r)$) {$\cokernel( A )$};
        \node (D) at ($(C) + (r)$) {$0$.};
        \draw[->,thick] (A) to node[above]{$A$} (B);
        \draw[->,thick] (B) to (C);
        \draw[->,thick] (C) to (D);
        \end{tikzpicture}
\end{center}
Since hom functors a left exact, we obtain a short exact sequence of functors
\begin{center}
   \begin{tikzpicture}[label/.style={postaction={
        decorate,
        decoration={markings, mark=at position .5 with \node #1;}},
        mylabel/.style={thick, draw=none, align=center, minimum width=0.5cm, minimum height=0.5cm,fill=white}}]
        \coordinate (r) at (3.5,0);
        \node (A) {$0$};
        \node (B) at ($(A)+0.65*(r)$) {$\Hom(\cokernel(A),-)$};
        \node (C) at ($(B)+0.9*(r)$) {$\Hom(R^{1 \times n},-)$};
        \node (D) at ($(C) + 1.25*(r)$) {$\Hom(R^{1 \times m},-)$};
        \draw[->,thick] (A) to  (B);
        \draw[->,thick] (B) to (C);
        \draw[->,thick] (C) to node[above]{$\Hom(A,-)$} (D);
  \end{tikzpicture}
\end{center}
We can identify this exact sequence of functors with the one in \eqref{equation:ex_seq_of_beh} as follows:
\begin{center}
   \begin{tikzpicture}[label/.style={postaction={
        decorate,
        decoration={markings, mark=at position .5 with \node #1;}},
        mylabel/.style={thick, draw=none, align=center, minimum width=0.5cm, minimum height=0.5cm,fill=white}}]
        \coordinate (r) at (3.5,0);
        \coordinate (d) at (0,-1.5);
        \node (A) {$0$};
        \node (B) at ($(A)+0.65*(r)$) {$\Hom(\cokernel(A),-)$};
        \node (C) at ($(B)+0.9*(r)$) {$\Hom(R^{1 \times n},-)$};
        \node (D) at ($(C) + 1.25*(r)$) {$\Hom(R^{1 \times m},-)$};

        \node (A2) at ($(A)+(d)$) {$0$};
        \node (B2) at ($(A2)+0.65*(r)$) {$\mathcal{B}( {_A}\phi, - )$};
        \node (C2) at ($(B2)+0.9*(r)$) {$\ff^{\oplus n}$};
        \node (D2) at ($(C2) + 1.25*(r)$) {$\ff^{\oplus m}$};
        
        \draw[->,thick] (A) to  (B);
        \draw[->,thick] (B) to (C);
        \draw[->,thick] (C) to node[above]{$\Hom(A,-)$} (D);

        \draw[->,thick] (A2) to  (B2);
        \draw[->,thick] (B2) to (C2);
        \draw[->,thick] (C2) to node[above]{$A$} (D2);

        \draw[->,thick] (C) to node[right]{$\simeq$} (C2);
        \draw[->,thick] (D) to node[right]{$\simeq$} (D2);
        \draw[->,thick] (B) to node[right]{$\simeq$} (B2);
\end{tikzpicture}
\end{center}
The two vertical isomorphisms on the right hand side are given by the universal property of free modules.
The vertical isomorphism on the left hand side is induced by the functoriality of taking kernels.
\end{proof}

\begin{definition}
Let $A \in R^{m \times n}$.
Any module that represents $\mathcal{B}( {_A}\phi, - )$ is called a \textbf{system module} for $\mathcal{B}( {_A}\phi, - )$.
\end{definition}

\begin{para}
By Yoneda's lemma and \ref{lemma:malgrange_iso}, all system modules of a given functor of trajectories $\mathcal{B}( {_A}\phi, - )$ are isomorphic to the finitely presented module $\cokernel(A)$.
\end{para}

\begin{construction}\label{construction:presentation_of_pp_formula}
If $\phi$ is not quantifier-free, $\mathcal{B}( \phi, - )$ is not representable in general.
However, $\mathcal{B}( \phi, - )$ can be \emph{presented} as the cokernel of a natural transformation between representable functors. We give the construction of this presentation.

Let $\phi := {_B}\phi_{B'}$, where $B \in R^{m \times k}$, $B' \in R^{m \times (n-k)}$. We set $A := (B|B') \in R^{m \times n}$.
Then we obtain the following commutative diagram of functors with exact rows:
\begin{center}
   \begin{tikzpicture}[label/.style={postaction={
        decorate,
        decoration={markings, mark=at position .5 with \node #1;}},
        mylabel/.style={thick, draw=none, align=center, minimum width=0.5cm, minimum height=0.5cm,fill=white}}]
        \coordinate (r) at (2.5,0);
        \coordinate (d) at (0,-1.5);
        \node (A) {$0$};
        \node (B) at ($(A)+(r)$) {$\ff^{\oplus n-k}$};
        \node (C) at ($(B)+(r)$) {$\ff^{\oplus n}$};
        \node (D) at ($(C) + (r)$) {$\ff^{\oplus k}$};
        \node (E) at ($(D) + (r)$) {$0$};

        \node (A2) at ($(A)+(d)$) {$0$};
        \node (B2) at ($(A2)+(r)$) {$\mathcal{B}( {_{B'}}\phi, - )$};
        \node (C2) at ($(B2)+(r)$) {$\mathcal{B}( {_{A}}\phi, - )$};
        \node (D2) at ($(C2) + (r)$) {$\mathcal{B}( {_{B}}\phi{_{B'}}, - )$};
        \node (E2) at ($(D2) + (r)$) {$0$};
        
        \draw[->,thick] (A) to (B);
        \draw[->,thick] (B) to node[above]{$x' \mapsto \pmatcol{0}{x'}$} (C);
        \draw[->,thick] (C) to node[above]{$\pmatcol{x}{x'} \mapsto x$}(D);
        \draw[->,thick] (D) to (E);

        \draw[->,thick] (A2) to  (B2);
        \draw[->,thick] (B2) to (C2);
        \draw[->,thick] (C2) to (D2);
        \draw[->,thick] (D2) to (E2);

        \draw[<-left hook,thick] (C) to (C2);
        \draw[<-left hook,thick] (D) to (D2);
        \draw[<-left hook,thick] (B) to (B2);
\end{tikzpicture}
\end{center}
Here, the vertical arrows denote the canonical inclusions.
The verification of commutativity and exactness can be done componentwise.
When we apply the Malgrange isomorphism (Lemma \ref{lemma:malgrange_iso}) to the lower short exact sequence, we obtain our desired short exact sequence
\begin{center}
   \begin{tikzpicture}[label/.style={postaction={
        decorate,
        decoration={markings, mark=at position .5 with \node #1;}},
        mylabel/.style={thick, draw=none, align=center, minimum width=0.5cm, minimum height=0.5cm,fill=white}}]
        \coordinate (r) at (4,0);
        \coordinate (d) at (0,-1.5);

        \node (A2) {$0$};
        \node (B2) at ($(A2)+0.5*(r)$) {$\Hom( \cokernel( B' ), - )$};
        \node (C2) at ($(B2)+(r)$) {$\Hom( \cokernel( A ), - )$};
        \node (D2) at ($(C2) + 0.75*(r)$) {$\mathcal{B}( \phi, - )$};
        \node (E2) at ($(D2) + 0.35*(r)$) {$0$};
        
        \draw[->,thick] (A2) to  (B2);
        \draw[->,thick] (B2) to node[above]{$\Hom(\alpha,-)$} (C2);
        \draw[->,thick] (C2) to (D2);
        \draw[->,thick] (D2) to (E2);
\end{tikzpicture}
\end{center}
Here, $\alpha: \cokernel( B' ) \rightarrow \cokernel( A )$ is the unique morphism of finitely presented left $R$-modules that fits into the following commutative diagram:
\begin{equation}\label{equation:presentation_mor_of_beh}
   \begin{tikzpicture}[mylabel/.style={fill=white},baseline=(Rn)]
        \coordinate (r) at (3.5,0);
        \coordinate (d) at (0,-1.5);

        \node (Rml) {$R^{1 \times m}$};
        \node (Rmr) at ($(Rml)+(r)$) {$R^{1 \times m}$};
        \node (Rn) at ($(Rml)+(d)$) {$R^{1 \times n}$};
        \node (Rnk) at ($(Rn) + (r)$) {$R^{1 \times (n-k)}$};
        \node (cokA) at ($(Rn) + (d)$) {$\cokernel(A)$};
        \node (cokC) at ($(cokA) + (r)$) {$\cokernel(B')$};
        
        \draw[->,thick] (Rml) to node[above]{$\id$} (Rmr);
        \draw[->,thick] (Rn) to node[above]{$(x,x') \mapsto x'$} (Rnk);
        \draw[->,thick] (cokA) to node[above]{$\alpha$}(cokC);

        \draw[->,thick] (Rml) to node[left]{$A$} (Rn);
        \draw[->,thick] (Rmr) to node[left]{$B'$} (Rnk);
        
        \draw[->,thick] (Rn) to (cokA);
        \draw[->,thick] (Rnk) to (cokC);

\end{tikzpicture}
\end{equation}
\end{construction}

Motivated by the presentation constructed in \ref{construction:presentation_of_pp_formula}, we now describe a formal context for finitely presented functors, which was introduced by Auslander \cite{A}. 

\begin{definition}\label{definition:fp_functor}
Let $\AC$ be an additive category.
A functor of the form
\[
\mathcal{G}: \AC \longrightarrow \Ab
\]
is called \textbf{finitely presented} if there exist objects $M, N \in \AC$, a morphism $M \xrightarrow{\alpha} N$, and an exact sequence of functors 
\begin{equation}\label{equation:presentation_given_by_alpha}
   \begin{tikzpicture}[mylabel/.style={fill=white},baseline=(A)]
        \coordinate (r) at (4.5,0);
        \node (A) {$\Hom( N, - )$};
        \node (B) at ($(A)+(r)$) {$\Hom( M, - )$};
        \node (C) at ($(B) + 0.5*(r)$) {$\mathcal{G}$};
        \node (D) at ($(C) + 0.25*(r)$) {$0$.};
        \draw[->,thick] (A) to node[above]{$\Hom( \alpha, - )$} (B);
        \draw[->,thick] (B) to (C);
        \draw[->,thick] (C) to (D);
  \end{tikzpicture}
\end{equation}
We refer to the exact sequence in \eqref{equation:presentation_given_by_alpha} as a \textbf{presentation} of $\mathcal{G}$ given by the morphism $\alpha$.
We denote by $\AC\modl$ the category whose objects are given by finitely presented functors on $\AC$ and whose morphisms are given by natural transformations.
\end{definition}

\begin{remark}
{A priori, the natural transformations between two functors form a class. However, for the natural transformations between finitely presented functors, it easily follows from Yoneda's lemma that their natural transformations actually form a set. Thus, $\AC\modl$ is a (locally small) category.}
\end{remark}

\begin{para}\label{para:eval_at_a_module}
We are mainly interested in \ref{definition:fp_functor} in the case where $\AC$ is the additive category $R\modl$ of finitely presented $R$-modules, i.e., we are interested in the category
\[
R\modl\modl.
\]
First, we clarify its relationship to the category $R\Modl\modl$, i.e., finitely presented functors on the whole module category $R\Modl$.
We have a full and faithful inclusion of categories
\begin{align*}
R\modl\modl &\hookrightarrow R\Modl\modl
\end{align*}
that can be explained in two ways (see, e.g., \cite[Section 10]{AJI}):
\begin{enumerate}
    \item A presentation of a finitely presented functor $\mathcal{G} \in R\modl\modl$ given by a morphism $\alpha \in R\modl$, i.e., an exact sequence of functors of the form
    \begin{center}
   \begin{tikzpicture}[label/.style={postaction={
        decorate,
        decoration={markings, mark=at position .5 with \node #1;}},
        mylabel/.style={thick, draw=none, align=center, minimum width=0.5cm, minimum height=0.5cm,fill=white}}]
        \coordinate (r) at (5,0);
        \node (A) {$\Hom_{R\modl}( N, - )$};
        \node (B) at ($(A)+(r)$) {$\Hom_{R\modl}( M, - )$};
        \node (C) at ($(B) + 0.5*(r)$) {$\mathcal{G}$};
        \node (D) at ($(C) + 0.25*(r)$) {$0$,};
        \draw[->,thick] (A) to node[above]{$\Hom_{R\modl}( \alpha, - )$} (B);
        \draw[->,thick] (B) to (C);
        \draw[->,thick] (C) to (D);
  \end{tikzpicture}
\end{center}
    gives rise to a presentation of functors of the form
    \begin{center}
   \begin{tikzpicture}[label/.style={postaction={
        decorate,
        decoration={markings, mark=at position .5 with \node #1;}},
        mylabel/.style={thick, draw=none, align=center, minimum width=0.5cm, minimum height=0.5cm,fill=white}}]
        \coordinate (r) at (5,0);
        \node (A) {$\Hom_{R\Modl}( N, - )$};
        \node (B) at ($(A)+(r)$) {$\Hom_{R\Modl}( M, - )$};
        \node (C) at ($(B) + 0.5*(r)$) {$\mathcal{G}'$};
        \node (D) at ($(C) + 0.25*(r)$) {$0$};
        \draw[->,thick] (A) to node[above]{$\Hom_{R\Modl}( \alpha, - )$} (B);
        \draw[->,thick] (B) to (C);
        \draw[->,thick] (C) to (D);
  \end{tikzpicture}
\end{center}
    given by $\alpha$ interpreted as a morphism in the whole module category $R\Modl$.
    Then $\mathcal{G}' \in R\Modl\modl$ is the image of $\mathcal{G}$ under the full and faithful inclusion.
    \item A finitely presented functor $\mathcal{G} \in R\modl\modl$ is mapped to its extension by filtered colimits, i.e., to a functor of type $R\Modl \rightarrow \Ab$ with the following two properties which determine it uniquely up to natural isomorphism:
    \begin{enumerate}
        \item it coincides with $\mathcal{G}$ on $R\modl$,
        \item it commutes with filtered colimits.
    \end{enumerate}
\end{enumerate}
\end{para}

\begin{convention}
It follows from our discussion in \ref{para:eval_at_a_module} that we can evaluate a functor $\mathcal{G} \in R\modl\modl$ at an arbitrary module $M \in R\Modl$ by means of the second interpretation. We will tacitly make use of this second point of view of the category $R\modl\modl$ throughout the rest of our paper, i.e., we will regard $R\modl\modl$ as a full subcategory of $R\Modl\modl$ whenever it is convenient to do so.
\end{convention}

The following lemma summarizes \ref{construction:presentation_of_pp_formula} in the terminology of finitely presented functors.

\begin{lemma}[Malgrange isomorphism for a behavior with latent variables]\label{lemma:malgrange_beh}
Let ${_B}\phi_{B'}$ be the pp formula given by $B \in R^{m \times k}$, $B' \in R^{m \times (n-k)}$.
Then the functor of trajectories $\mathcal{B}( {_B}\phi_{B'}, - )$ lies in $R\modl\modl$. More precisely, it is finitely presented with a presentation given by an epimorphism $\alpha$ between finitely presented modules of type
\[
\cokernel( (B|B') ) \rightarrow \cokernel( B' ),
\]
in other words, we have a natural isomorphism
\[
\cokernel( \Hom( \alpha, - ) ) \simeq \mathcal{B}( {_B}\phi_{B'}, - ).
\]
\end{lemma}

\begin{remark}
If $\phi$ is quantifier-free, then $B'$ in \ref{lemma:malgrange_beh} is the empty matrix of type $R^{m \times 0}$ and $\cokernel( B' )$ is the zero module. In that case, the lemma specializes to the usual Malgrange isomorphism in \ref{lemma:malgrange_iso}.
\end{remark}

%% file: fpfunctors.tex
In this section, we discuss the universal property of $R\modl\modl$ and some of its direct consequences.

\begin{para}
We denote the category of all additive functors of type $R\modl \rightarrow \Ab$ by $R\modl\Modl$. Morphisms in this category are the natural transformations.
Since $R\modl$ is skeletally small, i.e., it is equivalent to a small category, $R\modl\Modl$ is locally small.
Moreover, $R\modl\Modl$ is an abelian category, kernels and cokernels of natural transformations can be computed componentwise.
By definition, the category $R\modl\modl$ is a full subcategory of $R\modl\Modl$.
We also warn the reader that $R\modl\Modl$ should not be confused with $R\Modl\modl$, the category which we used in \ref{para:eval_at_a_module}.
\end{para}

\begin{theorem}
The category $R\modl\modl$ is abelian.
Moreover, the canonical inclusion functor
\[
R\modl\modl \hookrightarrow R\modl\Modl
\]
is exact.
In particular, the kernel/cokernel/image of a natural transformation in $R\modl\modl$ can be computed componentwise.
\end{theorem}
\begin{proof}
See, e.g., \cite[Proposition 10.2.4]{PrestPSL}.
\end{proof}

\begin{para}
Any ring $R$ can be regarded as a category with a single object $\bullet$, whose endomorphisms are given by the elements in $R$, and whose (post)composition is given by ring multiplication. Diagrammatically, this means that composition in that category works as follows:
\[
(\bullet \xrightarrow{r} \bullet \xrightarrow{s} \bullet) = (\bullet \xrightarrow{sr} \bullet)
\]
for all $r,s \in R$. By abuse of notation, we denote this category again by $R$.
\end{para}

\begin{para}\label{para:R_to_Rmodmod}
We have a \emph{canonical} functor
\begin{align*}
R &\rightarrow R\modl\modl \\ 
(\bullet \xrightarrow{r} \bullet) &\mapsto (\ff \xrightarrow{r} \ff).
\end{align*}
which is full and faithful by \ref{example:end_of_ff}.
\end{para}

\begin{theorem}[Universal property of $R\modl\modl$]\label{theorem:up_Rmodmod}
Let $R$ be a ring.
The category $R\modl\modl$ is the \emph{free abelian category} generated by $R$, or more precisely, generated by the forgetful functor $\ff$ and its endomorphisms. Concretely: if $\AC$ is an abelian category and $R \xrightarrow{M} \AC$ a functor, then there is a unique (up to natural isomorphism) exact functor $\ev{M}: R\modl\modl \rightarrow \AC$ such that the following diagram commutes (up to natural isomorphism):
\begin{equation}\label{equation:up_rmodmod}
   \begin{tikzpicture}[mylabel/.style={fill=white}, baseline=($(R)+0.5*(d)$)]
    \coordinate (r) at (4,0);
    \coordinate (d) at (0,-2);
    \node (R) {$R$};
    \node (Free) at ($(R)+(r)$) {$R\modl\modl$};
    \node (A) at ($(Free) + (d)$) {$\AC$};
    \draw[->,thick] (R) to (Free);
    \draw[->,thick] (R) to node[below]{$M$} (A);
    \draw[->,thick,dashed] (Free) to node[right]{$\ev{M}$} (A);
    \end{tikzpicture}
\end{equation}
\end{theorem}
\begin{proof}
Peter Freyd was the first who proved the existence of free abelian categories in \cite[Theorem 4.1]{FreydRep}.
A proof that directly applies to $R\modl\modl$ can for example be found in \cite[Universal Property 2.10]{KrauseLocFP}.
\end{proof}

\begin{para}\label{para:modules_as_functors}
The notation $\ev{M}$ in \ref{theorem:up_Rmodmod} suggests that we are evaluating the functor $\mathcal{G} \in R\modl\modl$ at $M: R \rightarrow \AC$.
We explain this notation in the case $\AC = \Ab$.
First, we note that there is a equivalence of categories
\begin{equation}\label{equation:equiv_modules_functors}
R\Modl \simeq \Hom( R, \Ab )
\end{equation}
where $\Hom( R, \Ab )$ is the category of all additive functors $R \rightarrow \Ab$ with natural transformations as morphisms.
This equivalence is given by sending an $R$-module $M$ to the functor $R \rightarrow \Ab: (\bullet \xrightarrow{r} \bullet) \mapsto (\ff(M) \xrightarrow{r} \ff(M))$. By abuse of notation, we also write $M: R \rightarrow \Ab$ for this functor.

Second, as we explained in \ref{para:eval_at_a_module}, we can think about $\mathcal{G} \in R\modl\modl$ as a functor of type $R\Modl \rightarrow \Ab$. In particular, we can evaluate it at any $M \in R\Modl$, and evaluation at $M$ yields an exact functor
\begin{equation}\label{equation:eval_functor}
R\modl\modl \rightarrow \Ab: \mathcal{G} \mapsto \mathcal{G}(M)
\end{equation}
since exactness in a functor category can be tested componentwise.
By the universal property of $R\modl\modl$, this evaluation functor is uniquely determined by its restriction to $R$ along the canonical functor of type $R \rightarrow R\modl\modl$. But this restriction is exactly given by the functor $M: R \rightarrow \Ab$ corresponding to the module $M$ under the equivalence in \eqref{equation:equiv_modules_functors}. In other words, $\ev{M}$ equals the evaluation functor in \eqref{equation:eval_functor}.

We endorse this point of view by introducing the following notation: for every $\mathcal{G} \in R\modl\modl$ and every $M: R \rightarrow \AC$, we set
\[
\mathcal{G}(M) := \ev{M}(\mathcal{G}).
\]
\end{para}

\begin{ex}
Let $\phi$ be a pp formula over $R$ and let $\mathcal{B}( \phi, - )$ be its corresponding functor of trajectories.
By \ref{lemma:malgrange_beh}, $\mathcal{B}( \phi, - )$ can be regarded as an object in $R\modl\modl$.
Let $M \in R\Modl$. Then, we can regard $M$ as a functor $M: R \rightarrow \Ab$, and as such, it triggers the universal property of $R\modl\modl$ and yields an exact functor $\ev{M}: R\modl\modl \rightarrow \Ab$.
Our explanation in \ref{para:modules_as_functors} shows that 
\[
\ev{M}(\mathcal{B}( \phi, - )) \cong \mathcal{B}( \phi, M ).
\]
\end{ex}

\begin{para}\label{para:compute_ev}
We explain how to compute $\ev{M}$ (up to natural isomorphism) in the case where $M: R \rightarrow \AC$ is an arbitrary additive functor into an abelian category $\AC$.
We proceed step by step in our explanation.
\begin{itemize}
\item For $r \in R$, we have 
\[
\ev{M}(\ff \xrightarrow{r} \ff) \cong \left( M( \bullet ) \xrightarrow{M( r ) } M( \bullet )\right)
\]
directly from the commutativity of the diagram in \eqref{equation:up_rmodmod}. 
\item For a matrix $A = (a_{ij})_{ij} \in R^{m \times n}$, we have
\[
\ev{M}(\ff^{\oplus n} \xrightarrow{A} \ff^{\oplus m}) \cong \left( M( \bullet )^{\oplus n} \xrightarrow{ (M( a_{ij} ))_{ij} } M( \bullet )^{\oplus m})\right)
\]
since $\ev{M}$ is an additive functor.
\item For $N \in R\modl$ presented by the matrix $A \in R^{m \times n}$, we have
\[
\Hom(N,M) = 
\ev{M}( \Hom(N,-) ) \cong \kernel( \ev{M}(\ff^{\oplus n} \xrightarrow{A} \ff^{\oplus m}) )
\]
since $\ev{M}$ respects kernels and since we have an exact sequence
\begin{center}
       \begin{tikzpicture}[mylabel/.style={fill=white}]
        \coordinate (r) at (3,0);
        \node (A) {$0$};
        \node (B) at ($(A)+(r)$) {$\Hom(N,-)$};
        \node (C) at ($(B) + (r)$) {$\ff^{\oplus n}$};
        \node (D) at ($(C) + (r)$) {$\ff^{\oplus m}$.};
        \draw[->,thick] (A) to (B);
        \draw[->,thick] (B) to (C);
        \draw[->,thick] (C) to node[above]{$A$}(D);
        \end{tikzpicture}
    \end{center}
\item Let $\alpha: N \rightarrow N' \in R\modl$. By lifting $\alpha$ to presentations, we obtain a commutative diagram in $R\modl$ with exact rows of the form
\begin{equation}\label{equation:module_presentation}
       \begin{tikzpicture}[mylabel/.style={fill=white}, baseline = ($(A) + 0.5*(d)$)]
        \coordinate (r) at (3,0);
        \coordinate (d) at (0,-2);
        \node (A) {$R^{1 \times m}$};
        \node (B) at ($(A)+(r)$) {$R^{1 \times n}$};
        \node (C) at ($(B) + (r)$) {$N$};
        \node (D) at ($(C) + (r)$) {$0$};

        \node (A2) at ($(A) + (d)$){$R^{1 \times m'}$};
        \node (B2) at ($(A2)+(r)$) {$R^{1 \times n'}$};
        \node (C2) at ($(B2) + (r)$) {$N'$};
        \node (D2) at ($(C2) + (r)$) {$0$};
        
        \draw[->,thick] (A) to node[above]{$A$} (B);
        \draw[->,thick] (B) to (C);
        \draw[->,thick] (C) to (D);

        \draw[->,thick] (A2) to node[above]{$A'$} (B2);
        \draw[->,thick] (B2) to (C2);
        \draw[->,thick] (C2) to (D2);

        \draw[->,thick] (A) to node[left]{$B'$} (A2);
        \draw[->,thick] (B) to node[left]{$B$}(B2);
        \draw[->,thick] (C) to node[left]{$\alpha$}(C2);
        \end{tikzpicture}
    \end{equation}
    where $A \in R^{m \times n}$, $A' \in R^{m' \times n'}$, $B \in R^{n \times n'}$, $B' \in R^{m \times m'}$.
    Now, let 
    \[
    \mathcal{G} := \cokernel( \Hom( \alpha, - ) ) \in R\modl\modl.
    \]
    Then \eqref{equation:module_presentation} gives rise the following commutative diagram in $R\modl\modl$ with exact rows and columns:
    \begin{equation}\label{equation:res_of_G_in_rmodmod}
    \begin{tikzpicture}[mylabel/.style={fill=white}, baseline = (A)]
        \coordinate (r) at (3,0);
        \coordinate (d) at (0,-2);
        \node (A) {$\ff^{\oplus m}$};
        \node (B) at ($(A)+(r)$) {$\ff^{\oplus n}$};
        \node (C) at ($(B) + (r)$) {$\Hom(N,-)$};
        \node (D) at ($(C) + (r)$) {$0$};

        \node (A2) at ($(A) + (d)$){$\ff^{\oplus m'}$};
        \node (B2) at ($(A2)+(r)$) {$\ff^{\oplus n'}$};
        \node (C2) at ($(B2) + (r)$) {$\Hom(N',-)$};
        \node (D2) at ($(C2) + (r)$) {$0$};

        \node (G) at ($(C) - 0.75*(d)$) {$\mathcal{G}$};

        \node (Gz) at ($(G) - 0.75*(d)$) {$0$};

        \draw[->,thick] (C) to (G);
        \draw[->,thick] (G) to (Gz);
        
        \draw[<-,thick] (A) to node[above]{$A$} (B);
        \draw[<-,thick] (B) to (C);
        \draw[<-,thick] (C) to (D);

        \draw[<-,thick] (A2) to node[above]{$A'$} (B2);
        \draw[<-,thick] (B2) to (C2);
        \draw[<-,thick] (C2) to (D2);

        \draw[<-,thick] (A) to node[right]{$B'$} (A2);
        \draw[<-,thick] (B) to node[right]{$B$}(B2);
        \draw[<-,thick] (C) to node[right]{$\Hom(\alpha,-)$}(C2);
        \end{tikzpicture}
    \end{equation}
    Applying the exact functor $\ev{M}$ to the diagram in \eqref{equation:res_of_G_in_rmodmod} yields the following commutative diagram in $\AC$ with exact rows and columns:
    \begin{equation}\label{equation:how_to_compute_eval}
    \begin{tikzpicture}[mylabel/.style={fill=white}, baseline = (G)]
        \coordinate (r) at (4,0);
        \coordinate (d) at (0,-2.25);
        \node (A) {$M(\bullet)^{\oplus m}$};
        \node (B) at ($(A)+(r)$) {$M(\bullet)^{\oplus n}$};
        \node (C) at ($(B) + 0.85*(r)$) {$\Hom(N,M)$};
        \node (D) at ($(C) + 0.5*(r)$) {$0$};

        \node (A2) at ($(A) + (d)$){$M(\bullet)^{\oplus m'}$};
        \node (B2) at ($(A2)+(r)$) {$M(\bullet)^{\oplus n'}$};
        \node (C2) at ($(B2) + 0.85*(r)$) {$\Hom(N',M)$};
        \node (D2) at ($(C2) + 0.5*(r)$) {$0$};

        \node (G) at ($(C) - 0.75*(d)$) {$\mathcal{G}(M)$};

        \node (Gz) at ($(G) - 0.75*(d)$) {$0$};

        \draw[->,thick] (C) to (G);
        \draw[->,thick] (G) to (Gz);
        
        \draw[<-,thick] (A) to node[above]{$\ev{M}(A)$} (B);
        \draw[<-,thick] (B) to (C);
        \draw[<-,thick] (C) to (D);

        \draw[<-,thick] (A2) to node[above]{$\ev{M}(A')$} (B2);
        \draw[<-,thick] (B2) to (C2);
        \draw[<-,thick] (C2) to (D2);

        \draw[<-,thick] (A) to node[right]{$\ev{M}(B')$} (A2);
        \draw[<-,thick] (B) to node[right]{$\ev{M}(B)$}(B2);
        \draw[<-,thick] (C) to node[right]{$\Hom(\alpha,M)$}(C2);
        \end{tikzpicture}
    \end{equation}
    From this last diagram \eqref{equation:how_to_compute_eval}, we can now read off how to compute $\mathcal{G}(M)$.
\end{itemize}
\end{para}

\begin{ex}[Contravariant defect]\label{example:defect}
We have a functor
\[
M: R \rightarrow (R\Modl)^{\op}
\]
that maps the unique object $\bullet$ of $R$ (regarded as a category) to the free left $R$-module $R^{1 \times 1}$ of rank $1$, and $r \in R$ to the morphism of left modules induced by right multiplication with $r$. Note that this assignment is contravariant, which is why we take the opposite category of $R\Modl$ in our definition of $M$.

Let $\mathcal{G} \in R\modl\modl$ have a presentation given by $\alpha: N \rightarrow N' \in R\modl$.
By \ref{para:compute_ev}, we can compute 
\[
\ev{M}(\mathcal{G}) \cong \kernel( \alpha ) 
\]
where we take the kernel in $R\Modl$.
Thus, the functor
\[
\ev{M}: R\modl\modl \rightarrow (R\Modl)^{\op}
\]
is given by taking the so-called \textbf{(contravariant) defect} of $\mathcal{G}$ \cite{Russell16}, which we also denote by
\[
\defect(\mathcal{G}) := \ev{M}( \mathcal{G} ).
\]
It will play a crucial role for the module-behavior duality in \ref{corollary:beh_rmodop_equivalence}.
\end{ex}

\begin{ex}[Contravariant defect of a pp formula]
Let $\phi := {_B}\phi_{B'}$ be the pp formula given by $B \in R^{m \times k}$, $B' \in R^{m \times (n-k)}$, and set $A := (B|B')$. We want to compute
\[
\defect( \mathcal{B}( \phi, - ) ).
\]
By the extended version of the Malgrange isomorphism (Lemma \ref{lemma:malgrange_beh}), we have
\[
\cokernel( \Hom( \alpha, - ) ) \simeq \mathcal{B}( \phi, - )
\]
for $\alpha$ the morphism in the diagram \eqref{equation:presentation_mor_of_beh}. 
By \ref{example:defect} we need to compute $\kernel( \alpha )$.
For this, we take a look at the following extension of the diagram in \eqref{equation:presentation_mor_of_beh}:
\begin{center}
   \begin{tikzpicture}[mylabel/.style={fill=white}]
        \coordinate (r) at (2.5,0);
        \coordinate (d) at (0,-1.5);

        \node (Rml) {$R^{1 \times m}$};
        \node (Rmr) at ($(Rml)+(r)$) {$R^{1 \times m}$};
        \node (Rn) at ($(Rml)+(d)$) {$R^{1 \times n}$};
        \node (Rnk) at ($(Rn) + (r)$) {$R^{1 \times (n-k)}$};
        \node (cokA) at ($(Rn) + (d)$) {$\cokernel(A)$};
        \node (cokC) at ($(cokA) + (r)$) {$\cokernel(B')$};

        \node (OL) at ($(Rml) - (r)$) {$0$};
        \node (Rk) at ($(Rn) - (r)$) {$R^{1 \times k}$};
        \node (Rk2) at ($(Rk) + (d)$) {$R^{1 \times k}$};
        \node (z25) at ($(Rmr) + (r)$) {$0$};
        \node (z35) at ($(Rnk) + (r)$) {$0$};
        \node (z21) at ($(OL) - (r)$) {$0$};
        \node (z31) at ($(Rk) - (r)$) {$0$};
        \node (kerBp) at ($(Rmr) - (d)$) {$\kernel( B')$};
        \node (kerA) at ($(Rml) - (d)$) {$\kernel( A )$};
        \node (z12) at ($(OL) - (d)$) {$0$};
        
        \draw[->,thick] (Rml) to node[above]{$\id$} (Rmr);
        \draw[->,thick] (Rn) to (Rnk);
        \draw[->,thick] (cokA) to node[above]{$\alpha$}(cokC);

        \draw[->,thick] (Rml) to node[left]{$A$} (Rn);
        \draw[->,thick] (Rmr) to node[left]{$B'$} (Rnk);
        
        \draw[->,thick] (Rn) to (cokA);
        \draw[->,thick] (Rnk) to (cokC);

        \draw[->,thick] (z12) to (kerA);
        \draw[->,thick] (kerA) to (kerBp);
        \draw[->,thick] (Rk2) to (cokA);
        \draw[->,thick] (z21) to (OL);
        \draw[->,thick] (OL) to (Rml);
        \draw[->,thick] (Rmr) to (z25);

        \draw[->,thick] (z31) to (Rk);
        \draw[->,thick] (Rk) to (Rn);
        \draw[->,thick] (Rnk) to (z35);

        \draw[->,thick] (z12) to (OL);
        \draw[->,thick] (kerA) to (Rml);
        \draw[->,thick] (kerBp) to (Rmr);
        \draw[->,thick] (OL) to (Rk);
        \draw[->,thick] (Rk) to (Rk2);

        \draw[->,thick,rounded corners,dashed] ($(kerBp.east)$) -| node[mylabel,yshift=-2em]{$\partial$} ($(Rmr.east) + 0.3*(d) + 0.3*(r)$) -| ($(Rk.west) - 0.1*(r)$) |-  ($(Rk2.west)$);
\end{tikzpicture}
\end{center}
The second and third row of this diagram are short exact sequences.
Thus, we can apply the snake lemma, from which we get the morphism $\partial$ together with a long exact sequence
\begin{center}
   \begin{tikzpicture}[mylabel/.style={fill=white}]
        \coordinate (r) at (2,0);;

        \node (A) {$0$};
        \node (B) at ($(A)+0.75*(r)$) {$\kernel(A)$};
        \node (C) at ($(B)+(r)$) {$\kernel(B')$};
        \node (D) at ($(C)+(r)$) {$R^{1 \times k}$};
        \node (E) at ($(D)+(r)$) {$\cokernel(A)$};
        \node (F) at ($(E)+(r)$) {$\cokernel(B')$};
        \node (G) at ($(F)+0.75*(r)$) {$0$.};
        
        \draw[->,thick] (A) to (B);
        \draw[->,thick] (B) to (C);
        \draw[->,thick] (C) to node[above]{$\partial$}(D);
        \draw[->,thick] (D) to (E);
        \draw[->,thick] (E) to node[above]{$\alpha$} (F);
        \draw[->,thick] (F) to (G);
\end{tikzpicture}
\end{center}
From this long exact sequence, we can deduce
\[
\kernel( \alpha ) \cong \frac{R^{1 \times k}}{\image( \partial )}.
\]
By chasing elements, we can explicitly compute $\image( \partial )$ as
\[
\image( \partial ) = \{ y \cdot B \mid y \in \kernel(B') \} \subseteq R^{1 \times k}.
\]
Thus,
\[
\defect( \mathcal{B}( {_B}\phi_{B'}, - ) ) \cong \frac{R^{1 \times k}}{\{ y \cdot B \mid y \in \kernel(B') \}}.
\]
\end{ex}

\begin{remark}\label{remark:left_rop_modules}
In order to prepare the duality theorem for $R\modl\modl$, we remark that working with right $R$-modules is equivalent to working with left $R^{\op}$-modules, where $R^{\op}$ denotes the opposite ring of $R$. In the context of finitely presented modules, this means that we have an equivalence
\[
\modr R \simeq R^{\op}\modl.
\]
\end{remark}

\begin{para}\label{para:can_modrmodl}
Similar to \ref{para:R_to_Rmodmod}, we have a \emph{canonical} functor of type
\[
R \rightarrow ((\modr R)\modl)^{\op}
\]
that maps the unique object $\bullet$ of $R$ to the forgetful functor $\ff$ of type $\modr R \rightarrow \Ab$.
It can be obtained as follows:
first, we take the canonical functor of type $R^{\op} \rightarrow R^{\op}\modl\modl$ from \ref{para:R_to_Rmodmod}, where we use $R^{\op}$ instead of $R$.
Second, we reinterpret this functor by means of \ref{remark:left_rop_modules} and obtain a functor of type $R^{\op} \rightarrow (\modr R)\modl$. Last, we take the opposite functor.
\end{para}

\begin{theorem}[Duality of $R\modl\modl$ and $(\modr R)\modl$]\label{theorem:agj_duality}
Let
\[
M: R \rightarrow ((\modr R)\modl)^{\op}.
\]
denote the canonical functor, see \ref{para:can_modrmodl}.
Then
\[
\ev{M}: R\modl\modl \rightarrow ((\modr R)\modl)^{\op}
\]
is an equivalence of categories.
\end{theorem}
\begin{proof}
Let $N$ denote the canonical functor of type $R \rightarrow R\modl\modl$.
Since $(\modr R)\modl \simeq R^{\op}\modl\modl$ is the free abelian category of $R^{\op}$, an inverse of $\ev{M}$ (up to equivalence) can be obtained by $\ev{N^{\op}}^{\op}: ((\modr R)\modl)^{\op} \rightarrow R\modl\modl$.
The fact that this indeed yields an inverse can be easily checked on the full subcategories spanned by the respective forgetful functors, which, by the universal properties, is all we need to check.
\end{proof}

\begin{remark}
In \cite{DR18} the authors call the duality of $R\modl\modl$ and $(\modr R)\modl$ the \textbf{Auslander-Gruson-Jensen duality} since it was discovered by Auslander \cite{AusDuality} and Gruson and Jensen \cite{GJ80}.
\end{remark}

\begin{para}\label{para:duality_hom_tensor}
The Auslander-Gruson-Jensen duality maps hom functors to tensor functors.
To see this, we first remark that by its definition, the duality maps the forgetful functor 
\[
\Hom(R,-): R\modl \rightarrow \Ab
\]
to the forgetful functor 
\[
(- \otimes R): \modr R \rightarrow \Ab.
\]
Next, let $M \in R\modl$ be presented by the matrix $A \in R^{m \times n}$:
\begin{center}
   \begin{tikzpicture}[label/.style={postaction={
    decorate,
    decoration={markings, mark=at position .5 with \node #1;}},
    mylabel/.style={thick, draw=none, align=center, minimum width=0.5cm, minimum height=0.5cm,fill=white}}]
    \coordinate (r) at (3,0);
    \node (A) {$R^{1 \times m}$};
    \node (B) at ($(A)+(r)$) {$R^{1 \times n}$};
    \node (C) at ($(B) + (r)$) {$M$};
    \node (D) at ($(C) + (r)$) {$0$.};
    \draw[->,thick] (A) to node[above]{$A$} (B);
    \draw[->,thick] (B) to (C);
    \draw[->,thick] (C) to (D);
    \end{tikzpicture}
\end{center}
Now, the Auslander-Gruson-Jensen duality maps the kernel of
\[
\Hom(R^{n \times 1},-) \xrightarrow{\Hom(A,-)} \Hom(R^{m \times 1},-),
\]
which is $\Hom(M,-)$, to the cokernel of
\[
(- \otimes R^{1 \times m}) \xrightarrow{(- \otimes A)} (- \otimes R^{1 \times n})
\]
which is isomorphic to $(- \otimes M)$ since tensor functors are right exact.
\end{para}

\begin{corollary}\label{corollary:copresentations}
For every $\mathcal{G} \in R\modl\modl$ there exists an exact sequence in $R\modl\modl$ of the form
\begin{center}
   \begin{tikzpicture}[mylabel/.style={fill=white}]
    \coordinate (r) at (3,0);
    \node (A) {$0$};
    \node (B) at ($(A)+0.5*(r)$) {$\mathcal{G}$};
    \node (C) at ($(B) + 0.5*(r)$) {$(N \otimes -)$};
    \node (D) at ($(C) + (r)$) {$(N' \otimes -)$};
    \draw[->,thick] (A) to (B);
    \draw[->,thick] (B) to (C);
    \draw[->,thick] (C) to node[above]{$(\alpha \otimes -)$}(D);
    \end{tikzpicture}
\end{center}
for some morphism $\alpha: N \rightarrow N'$ in $\modr R$.
We refer to such an exact sequence as a \textbf{copresentation} of $\mathcal{G}$ given by the morphism $\alpha$.
\end{corollary}
\begin{proof}
By \ref{para:duality_hom_tensor} the Auslander-Gruson-Jensen duality sends presentations to copresentations.
\end{proof}

\begin{ex}[Copresentation of a pp formula]\label{example:copresentation_ppformula}
Let ${_B}\phi_{B'}$ be the pp formula given by $B \in R^{m \times k}$, $B' \in R^{m \times (n-k)}$.
We want to compute a copresentation of $\mathcal{B}( {_B}\phi_{B'}, - )$.
For this we let $\alpha$ be the unique morphism in $\modr R$ that makes the following diagram with exact rows commutative:
\begin{center}
    \begin{tikzpicture}[mylabel/.style={fill=white}]
      \coordinate (r) at (3,0);
      \coordinate (d) at (0,-2);
      
      \node (11) {$0$};
      \node (12) at ($(11)+(r)$) {$R^{k \times 1}$};
      \node (13) at ($(12) + (r)$) {$R^{k \times 1}$};
      \node (14) at ($(13) + (r)$) {$0$};
    
      \node (21) at ($(11) + (d)$){$R^{(n-k) \times 1}$};
      \node (22) at ($(21)+(r)$) {$R^{m \times 1}$};
      \node (23) at ($(22) + (r)$) {$\cokernel( B' )$};
      \node (24) at ($(23) + (r)$) {$0$};
      
      \draw[->,thick] (11) to (12);
      \draw[->,thick] (12) to (13);
      \draw[->,thick] (13) to (14);
    
      \draw[->,thick] (21) to node[above]{$B'$} (22);
      \draw[->,thick] (22) to (23);
      \draw[->,thick] (23) to (24);
    
      \draw[->,thick] (11) to (21);
      \draw[->,thick] (12) to node[left]{$B$} (22);
      \draw[->,thick] (13) to node[left]{$\alpha$} (23);
    \end{tikzpicture}
\end{center}
If we use the canonical identification $\ff^{\oplus k} \simeq (R^{k \times 1} \otimes -)$ in order to view $\mathcal{B}( {_B}\phi_{B'}, - )$ as a subobject of $(R^{k \times 1} \otimes -)$, then it is easy to compute that we have an exact sequence
\begin{center}
  \begin{tikzpicture}[mylabel/.style={fill=white}]
      \coordinate (r) at (3,0);
      
      \node (A) {$0$};
      \node (B) at ($(A)+(r)$) {$\mathcal{B}( {_B}\phi_{B'}, - )$};
      \node (C) at ($(B) + (r)$) {$(R^{k \times 1} \otimes -)$};
      \node (D) at ($(C) + (r)$) {$(\cokernel(B') \otimes -)$};
      
      \draw[->,thick] (A) to (B);
      \draw[->,thick] (B) to (C);
      \draw[->,thick] (C) to node[above]{$\alpha$} (D);
  \end{tikzpicture}
\end{center}
and hence our desired copresentation.
\end{ex}

\begin{corollary}\label{corollary:subobjects_are_beh_in_rmodmod}
Let $k \in \Nzero$. Any subobject $\mathcal{U} \hookrightarrow \ff^{\oplus k}$ in $R\modl\modl$ is equal to $\mathcal{B}(\phi,-) \hookrightarrow \ff^{\oplus k}$ as subobjects for some pp formula $\phi$ with $k$ free variables.
\end{corollary}
\begin{proof}
By \ref{corollary:copresentations} we can embed the quotient object $\ff^{\oplus k}/\mathcal{U}$ into some tensor functor $(N \otimes -)$ for some $N \in \modr R$. This gives us an exact sequence in $R\modl\modl$ of the form
\begin{center}
   \begin{tikzpicture}[mylabel/.style={fill=white}]
    \coordinate (r) at (3,0);
    \node (A) {$0$};
    \node (B) at ($(A)+(r)$) {$\mathcal{U}$};
    \node (C) at ($(B) + (r)$) {$\ff^{\oplus k}$};
    \node (D) at ($(C) + (r)$) {$(N \otimes -)$};
    \draw[->,thick] (A) to (B);
    \draw[->,thick] (B) to (C);
    \draw[->,thick] (C) to node[above]{$\alpha$}(D);
    \end{tikzpicture}
\end{center}
where $\alpha$ is the composition of the canonical projection $\ff^{\oplus k} \twoheadrightarrow \ff^{\oplus k}/\mathcal{U}$ with the chosen embedding $\ff^{\oplus k}/\mathcal{U} \hookrightarrow (N \otimes -)$.
Now, from \ref{example:copresentation_ppformula} we can see that this exact sequence has to be a copresentation of $\mathcal{B}( \phi, - )$ for some pp formula $\phi$ with $k$ free variables.
\end{proof}

\begin{ex}[Covariant defect]\label{example:covariant_defect}
There is a covariant version of the defect\footnote{Called the \textbf{covariant defect} by Alex Sorokin.}. 
It can be obtained by substituting $R^{\op}$ for $R$ in \ref{example:defect} and using the Auslander-Gruson-Jensen duality described in \ref{theorem:agj_duality}.
Explicitly:
We have a functor
\[
M: R \rightarrow \Modr R
\]
that maps the unique object $\bullet$ of $R$ (regarded as a category) to the free right $R$-module $R^{1 \times 1}$ of rank $1$, and $r \in R$ to the morphism of right modules induced by left multiplication with $r$.
Let $\mathcal{G} \in R\modl\modl$ have a copresentation given by $\alpha: N \rightarrow N' \in \modr R$ (see \ref{corollary:copresentations}).
Then, we can compute the covariant defect as follows:
\[
\codefect(\mathcal{G}) := \ev{M}(\mathcal{G}) \cong \kernel( \alpha ).
\]
An alternative way to describe $\codefect(\mathcal{G})$ is given by evaluation at ${_R}R$: we have a functor
\begin{align*}
    R\modl\modl &\rightarrow \Modr R \\
    \mathcal{G} &\mapsto \mathcal{G}({_R}R)
\end{align*}
where we turn $\mathcal{G}({_R}R)$ into a right $R$-module by letting $r \in R$ act via the abelian group homomorphism $\mathcal{G}( {_R}R \xrightarrow{r} {_R}R )$.
Since it is an evaluation functor, it is exact. It follows from the universal property of $R\modl\modl$ that it has to be naturally isomorphic to the covariant defect.
\end{ex}

\begin{ex}[Covariant defect of a pp formula]
Let $\phi $ be a pp formula. We want to compute
\[
\codefect( \mathcal{B}( \phi, - ) ).
\]
By the alternative description given in \ref{example:covariant_defect}, we simply have
\[
\codefect( \mathcal{B}( \phi, - ) ) \cong \{ x \in R^{k \times 1} \mid \phi( x ) \} \in \Modr R.
\]
\end{ex}

\begin{para}
There are many more interesting homological properties of $R\modl\modl$. See \cite{FuHerzog20} for a detailed explanation of the relationship between exactness conditions on functors and their projective and injective dimensions.
\end{para}

\begin{para}
Since the possibility for an algorithmic treatment is vital for algebraic systems theory, we remark that $R\modl\modl$ can be treated constructively whenever the additive closure of $R$ has decidable homotopy equations, see \cite{PosFree} for details. The underlying data structure of an object in $R\modl\modl$ is due to an explicit description of free abelian categories by Adelman in \cite{Adelman}.
\end{para}

%% file: ambientcat.tex
In this section we construct an abelian ambient category for behaviors as a Serre quotient of $R\modl\modl$. Moreover, we prove that this construction satisfies the universal property that we sketched in our introduction.

\subsection{Preliminaries on Serre quotients}

This subsection gives a short introduction to the well-known theory of Serre quotients.
Throughout this subsection, $\AC$ denotes an abelian category.

Serre quotients are the universal solution to the problem of formally setting a subclass $\CC \subseteq \AC$ to zero in such a way that the resulting quotient category is again an abelian category.
Serre quotients in the context of abelian groups were introduced by Serre in \cite{Serre53}.
When Grothendieck introduced abelian categories in his famous Tôhoku paper \cite{GroTohoku}, he also described the process of taking quotients of abelian categories.
A first thorough treatment of Serre quotients can be found in Gabriel's thesis \cite{Gab_thesis}.

\begin{para}
An additive full subcategory $\CC \subseteq \AC$ is called a \textbf{Serre subcategory} if it is closed under passing to subobjects, quotient objects, and extensions.
\end{para}

\begin{para}
The \textbf{kernel} $\kernel(F)$ of an exact functor $F: \AC \rightarrow \BC$ between abelian categories $\AC$ and $\BC$ is defined as the full subcategory of $\AC$ spanned by the objects $\{ A \in \AC \mid F(A) \cong 0 \}$.
The kernel of an exact functor is always a Serre subcategory.
\end{para}

\begin{para}[Implicit description of Serre quotients by their universal property]\label{para:up_of_serre_quotients}
Let $\CC \subseteq \AC$ be a Serre subcategory. 
The \textbf{Serre quotient} of $\AC$ by $\CC$ consists of
\begin{enumerate}
\item an abelian category $\frac{\AC}{\CC}$,
\item an exact functor $\AC \rightarrow \frac{\AC}{\CC}$ (called the \textbf{canonical quotient functor})
\end{enumerate}
which satisfy the following universal property: for every abelian category $\BC$ and every exact functor $F: \AC \rightarrow \BC$ such that $\CC \subseteq \kernel( F )$, there exists an exact functor $\frac{\AC}{\CC} \rightarrow \BC$ (uniquely determined up to natural isomorphism) such that the following diagram commutes (up to natural isomorphism):
\begin{center}
       \begin{tikzpicture}[label/.style={postaction={
        decorate,
        decoration={markings, mark=at position .5 with \node #1;}},
        mylabel/.style={thick, draw=none, align=center, minimum width=0.5cm, minimum height=0.5cm,fill=white}}]
        \coordinate (r) at (4,0);
        \coordinate (d) at (0,-2);
        \node (A) {$\AC$};
        \node (B) at ($(A)+(r)$) {$\frac{\AC}{\CC}$};
        \node (C) at ($(B) + (d)$) {$\BC$};
        \draw[->,thick] (A) to (B);
        \draw[->,thick] (A) to node[below]{$F$} (C);
        \draw[->,thick,dashed] (B) to node[right]{$\exists!$} (C);
        \end{tikzpicture}
\end{center}
\end{para}

\begin{para}
The question whether Serre quotients exist leads to set-theoretic considerations (see, e.g., \cite[Chapter 10.3]{Weibel94}).
If $\AC$ is a small abelian category then the Serre quotient exists for all Serre subcategories.
We give two more scenarios in which Serre quotients do always exist in \ref{para:explicit_serre_quotients} and \ref{para:serre_quots_via_adjunctions}.
\end{para}

\begin{para}[Explicit description of Serre quotients]\label{para:explicit_serre_quotients}
Let $F: \AC \rightarrow \BC$ be an exact functor between abelian categories $\AC$ and $\BC$, and set $\CC := \kernel( F )$.
In this case, it is not hard to see that we can describe $\frac{\AC}{\CC}$ by means of $F$ and $\BC$:
\begin{enumerate}
\item The objects in $\frac{\AC}{\CC}$ are the same objects as in $\AC$.
\item The homomorphism sets $\Hom_{\frac{\AC}{\CC}}(A,A')$ for $A, A' \in \frac{\AC}{\CC}$ are given by a subset of $\Hom_{\BC}( FA, FA' )$, namely by all $f \in \Hom_{\BC}( FA, FA' )$ such that there exists an object $A'' \in \AC$ and two morphisms (in $\AC$)
\[
A \xleftarrow{g} A'' \xrightarrow{h} A'
\] 
such that $F(g)$ is an iso (in $\BC$) and 
\[
(FA \xrightarrow{f} FA') = (FA \xrightarrow{F(g)^{-1}} FA'' \xrightarrow{F(h)} FA').
\]
\item Composition in $\frac{\AC}{\CC}$ is given by composition in $\BC$, the identity of $A \in \frac{\AC}{\CC}$ is given by $\id_{FA}$.
\end{enumerate}
Note that we get a canonical factorization of $F$ into a composition of the canonical quotient functor followed by a canonical faithful exact functor:
\begin{center}
       \begin{tikzpicture}[label/.style={postaction={
        decorate,
        decoration={markings, mark=at position .5 with \node #1;}},
        mylabel/.style={thick, draw=none, align=center, minimum width=0.5cm, minimum height=0.5cm,fill=white}}]
        \coordinate (r) at (2,0);
        \coordinate (d) at (0,-1);
        \node (A) {$\AC$};
        \node (B) at ($(A)+(r)+(d)$) {$\frac{\AC}{\CC}$};
        \node (C) at ($(A) + 2*(r)$) {$\BC$};
        \draw[->,thick] (A) to (B);
        \draw[->,thick] (A) to node[above]{$F$} (C);
        \draw[->,thick] (B) to (C);
        \end{tikzpicture}
\end{center}
We will make use of this factorization within the next subsections.
\end{para}

\begin{para}[Identification of Serre quotients via adjunctions]\label{para:serre_quots_via_adjunctions}
Let $F: \AC \rightarrow \BC$ be an exact functor between abelian categories. Suppose that we have a right adjoint $R: \BC \rightarrow \AC$ to $F$ such that the counit of the adjunction $FR \rightarrow \id_{\BC}$ is a natural isomorphism.
Then the unique functor induced by the universal property of Serre quotients
\begin{center}
\begin{equation}\label{equation:serre_via_adjunction}
       \begin{tikzpicture}[label/.style={postaction={
        decorate,
        decoration={markings, mark=at position .5 with \node #1;}},
        mylabel/.style={thick, draw=none, align=center, minimum width=0.5cm, minimum height=0.5cm,fill=white}},baseline = ($(A) + 0.5*(d)$)]
        \coordinate (r) at (4,0);
        \coordinate (d) at (0,-2);
        \node (A) {$\AC$};
        \node (B) at ($(A)+(r)$) {$\frac{\AC}{\kernel(F)}$};
        \node (C) at ($(B) + (d)$) {$\BC$};
        \draw[->,thick] (A) to (B);
        \draw[->,thick] (A) to node[below]{$F$} (C);
        \draw[->,thick,dashed] (B) to node[right]{$\exists!$} (C);
        \end{tikzpicture}
\end{equation}
\end{center}
is an equivalence of categories \cite[Chapter III, Proposition 5]{Gab_thesis}. 
In this case, $\kernel(F)$ is called a \textbf{localizing Serre subcategory}.
Dually: suppose we have a left adjoint $L: \BC \rightarrow \AC$ to $F$ such that the unit of the adjunction $\id_{\BC} \rightarrow FL$ is a natural isomorphism.
Then again, the unique functor induced by the universal property of Serre quotients in \eqref{equation:serre_via_adjunction} is an equivalence.
\end{para}

We end our introduction to Serre quotients with two remarks that we will use within the next subsections.

\begin{remark}\label{remark:can_quotient_is_epi}
 If we are given two exact functors $F,G: \frac{\AC}{\CC} \rightarrow \BC$ into an abelian category $\BC$ such that
\[
(\AC \rightarrow \frac{\AC}{\CC} \xrightarrow{F} \BC) \simeq (\AC \rightarrow \frac{\AC}{\CC} \xrightarrow{G} \BC),
\]
then $F \simeq G$. Indeed, this follows directly from the fact that
\[
\CC \subseteq \kernel(\AC \rightarrow \frac{\AC}{\CC} \xrightarrow{F} \BC)
\]
and the universal property of Serre quotients.
If we ignore set theoretic issues, then this remark shows that the canonical quotient functor is an epimorphism in the category of abelian categories with exact functors as its morphisms considered up to natural isomorphism.
\end{remark}

\begin{remark}[Lifting subobjects]\label{remark:lifting_subobjects}
Let $A \hookrightarrow A'$ be a monomorphism in a Serre quotient $\frac{\AC}{\CC}$ for $\CC := \kernel(F)$ and $F: \AC \rightarrow \BC$ an exact functor between abelian categories.
Then by \ref{para:serre_quots_via_adjunctions} this monomorphism is of the form
\[
(FA \xrightarrow{F(g)^{-1}} FA'' \xrightarrow{F(h)} FA')
\]
for two morphisms in $\AC$
\[
A \xleftarrow{g} A'' \xrightarrow{h} A'
\]
such that $F(g)$ is an isomorphism in $\BC$.
In particular, it follows that the image embedding
\[
\image( h ) \hookrightarrow A'
\]
in $\AC$ yields a subobject in $\frac{\AC}{\CC}$ which is equal in $\frac{\AC}{\CC}$ to the subobject defined by $A \hookrightarrow A'$.
In other words, we can lift subobjects in $\frac{\AC}{\CC}$ of an object $A'$ to subobjects in $\AC$ of the same object $A'$.
\end{remark}

\subsection{Construction as a Serre quotient}\label{subsection:construction_of_behM}

In this subsection, we define $\Beh{R}{\M}$ as a Serre quotient of $R\modl\modl$, introduce our notation for its objects, and explain how we view its hom sets as subsets of abelian group homomorphisms.

\begin{definition}[Abelian ambient category for behaviors]
Let $\M \in R\Modl$ and let 
\begin{align*}
\ev{\M}: R\modl\modl &\rightarrow \Ab \\
\mathcal{G} &\mapsto \mathcal{G}(\M)
\end{align*}
be the evaluation functor at $\M$ (see \ref{para:modules_as_functors}).
The \textbf{abelian ambient category for behaviors over $\M$} is defined as the following Serre quotient:
\[
\Beh{R}{\M} := \frac{R\modl\modl}{\kernel(\ev{\M})}.
\]
\end{definition}

\begin{remark}\label{remark:abstract_names}
We introduce the following notation for the canonical quotient functor on objects:
\begin{align*}
R\modl\modl &\rightarrow \Beh{R}{\M} \\
\mathcal{G} &\mapsto  \asBeh{\mathcal{G}}{\M} 
\end{align*}
In other words, whenever we regard a functor $\mathcal{G} \in R\modl\modl$ as an object in $\Beh{R}{\M}$, we write $\asBeh{\mathcal{G}}{\M}$ for this object.
We call an object in $\Beh{R}{\M}$ an \textbf{abstract behavior}.
Recall that we denote by $\ff$ the forgetful functor in $R\modl\modl$.
The object $\asBeh{\ff}{M}$ plays a crucial role within the category $\Beh{R}{\M}$. We call it the \textbf{abstract signal space} (recall that we call the module $M$ the signal space, see \ref{def:behavior}).
\end{remark}

\begin{remark}\label{remark:morphisms_of_beh}
In this remark, we spell out the meaning of \ref{para:explicit_serre_quotients} in the context of $\Beh{R}{\M}$:
we have a commutative diagram of functors:
\begin{center}
       \begin{tikzpicture}[label/.style={postaction={
        decorate,
        decoration={markings, mark=at position .5 with \node #1;}},
        mylabel/.style={thick, draw=none, align=center, minimum width=0.5cm, minimum height=0.5cm,fill=white}}]
        \coordinate (r) at (5,0);
        \coordinate (d) at (0,-1);
        \node (A) {$R\modl\modl$};
        \node (B) at ($(A)+(r)+(d)$) {$\Beh{R}{\M}$};
        \node (C) at ($(A) + 2*(r)$) {$\Ab$};
        \draw[->,thick] (A) to node[below,xshift=-2em]{$\mathcal{G} \mapsto \asBeh{\mathcal{G}}{\M}$} (B);
        \draw[->,thick] (A) to node[above]{$\mathcal{G} \mapsto \mathcal{G}(\M)$} (C);
        \draw[->,thick] (B) to node[below,xshift=2em]{$\asBeh{\mathcal{G}}{\M} \mapsto \mathcal{G}(\M)$}(C);
        \end{tikzpicture}
\end{center}
Within this diagram, the \emph{canoncical} functor 
\begin{align*}
\Beh{R}{\M} &\rightarrow \Ab \\
\asBeh{\mathcal{G}}{\M} &\mapsto \mathcal{G}(\M)
\end{align*}
is faithful and exact, and thus can be regarded as a functor that forgets extra structure\footnote{In the sense of nlab \cite{nlab}.}:
This means that we can think about $\asBeh{\mathcal{G}}{\M}$ as an object with an underlying abelian group $\mathcal{G}(\M)$ that carries some extra structure which turns it into an abstract behavior. And only those group homomorphisms that respect this extra structure are morphisms of abstract behaviors.

The homomorphism sets $\Hom_{\Beh{R}{\M}}(\asBeh{\mathcal{G}}{\M}, \asBeh{\mathcal{H}}{\M})$
are given by all abelian group homomorphisms $\gamma: \mathcal{G}(\M) \rightarrow \mathcal{H}(\M)$ which are of the form
\[
\gamma = \alpha_M \circ \beta_M^{-1},
\]
where $\beta, \alpha \in R\modl\modl$ are natural transformations of types
\begin{center}
\begin{tikzpicture}[baseline=(base)]
        \coordinate (r) at (1.5,0);
        \coordinate (d) at (0,-1);
        \node (A) {$\mathcal{G}$};
        \node (B) at ($(A)+(r)+(d)$) {$\mathcal{E}$};
        \node (C) at ($(A) + 2*(r)$) {$\mathcal{H}$};
        \node (base) at ($(A) + 0.5*(d)$) {};
        \draw[->,thick] (B) to node[below]{$\beta$} (A);
        \draw[->,thick] (B) to node[below]{$\alpha$} (C);
\end{tikzpicture}
\end{center}
for some $\mathcal{E} \in R\modl\modl$ and where $\beta_M \in \Ab$ is an isomorphism of abelian groups.
Composition and identites in $\Beh{R}{\M}$ are given by composition and identities in $\Ab$.
In particular, we can now write down  the canonical quotient functor on morphisms:
\begin{align*}
R\modl\modl &\rightarrow \Beh{R}{\M} \\
\big(\mathcal{G} \xrightarrow{\alpha} \mathcal{H} \big) &\mapsto \big( \asBeh{\mathcal{G}}{\M} \xrightarrow{\asBehmor{\alpha}{\M}} \asBeh{\mathcal{H}}{\M} \big).
\end{align*}
\end{remark}

\begin{remark}[All Serre quotients arise as an abelian ambient category for behaviors]\label{remark:all_Serre_quotients_arise_as_beh}
Let
\[
\CC \subseteq R\modl\modl
\]
be a Serre subcategory. Then there exists a module $M \in R\Modl$ such that $\CC = \kernel( \ev{M} )$, and
\[
\frac{R\modl\modl}{\CC} \simeq \Beh{R}{M}.
\]
Indeed, $R\modl\modl$ is a small abelian category, i.e., its isomorphism classes form a set. It follows that the Serre quotient $\frac{R\modl\modl}{\CC}$ is also a small abelian category.
In particular, we may apply the Freyd-Mitchell embedding theorem \cite{FreydBook} and obtain a faithful exact functor
\[
\frac{R\modl\modl}{\CC} \rightarrow \Ab.
\]
Composition with the canonical quotient functor yields an exact functor
\[
R\modl\modl \rightarrow \frac{R\modl\modl}{\CC} \rightarrow \Ab.
\]
whose kernel is given by $\CC$ since the Freyd-Mitchell embedding is faithful.
The universal property $R\modl\modl$ (Theorem \ref{theorem:up_Rmodmod}) implies that this functor in turn is given by $\ev{M}$ for some uniquely determined (up to isomorphism) module
\[
M: R \rightarrow \Ab.
\]
We remark that a more systematic approach to obtain such a module $M$ for a given Serre subcategory $\CC$ is provided by the theory of pure injective modules \cite{PrestPSL}.
\end{remark}

\subsection{Elementary features}

We fix a ring $R$ and a module $M \in R\Modl$ in this subsection.
We enlist several features of $\Beh{R}{M}$ that are desirable from the point of view of algebraic systems theory. Many claims in this subsection simply follow from general properties of faithful and exact functors and we collected their proofs in greater generality in the Appendix.

We begin with two lemmata that show that elementary properties of objects and morphisms in $\Beh{R}{M}$ can be tested within $\Ab$. 

\begin{lemma}\label{lemma:desirable_features_Beh_zero}
Let $\mathcal{G} \in R\modl\modl$.
We have
    \[
    \asBeh{\mathcal{G}}{\M} \cong 0 \text{\hspace{2em}$\Longleftrightarrow$ \hspace{2em}} \mathcal{G}(\M) \cong 0.
    \]
\end{lemma}
\begin{proof}
Follows from Appendix \ref{lemma:faithful_implies_zero_test}.
\end{proof}

\begin{lemma}\label{lemma:desirable_features_Beh_monoepiiso}
Let $\mathcal{G},\mathcal{H} \in R\modl\modl$.
Suppose given a morphism 
\[
\asBeh{\mathcal{G}}{\M} \xrightarrow{\gamma}
\asBeh{\mathcal{H}}{\M}
\]
in $\Beh{R}{M}$. Then $\gamma$ is a mono/epi/iso in $\Beh{R}{\M}$ if and only if
\[
\mathcal{G}(\M) \xrightarrow{\gamma}
\mathcal{H}(\M)
\]
is a mono/epi/iso of abelian groups.
\end{lemma}
\begin{proof}
Follows from Appendix \ref{lemma:faithful_monoepiiso_test}.
\end{proof}

The following theorem shows that subobjects of $\asBeh{\ff}{M}^{\oplus k}$ correspond to those abelian subgroups of $M^{\oplus k}$ that are cut out by a pp formula in $k \in \Nzero$ free variables.

\begin{theorem}\label{theorem:subobjects_of_ff}
Let $k \in \Nzero$.
Every subobject of $\asBeh{\ff}{M}^{\oplus k}$, i.e., of the $k$-fold direct sum of the abstract signal space, is of the form
\[
\asBeh{\mathcal{B}( \phi, - )}{M} \xrightarrow{\iota_{M}} \asBeh{\ff}{M}^{\oplus k}
\]
for $\phi$ a pp formula in $k$ free variables and $\iota$ the canonical inclusion ${\mathcal{B}( \phi, - )} \hookrightarrow \ff^{\oplus k}$.
\end{theorem}
\begin{proof}
By \ref{remark:lifting_subobjects} we may lift subobjects of $\Beh{R}{M}$ to subobjects in $R\modl\modl$.
By \ref{corollary:subobjects_are_beh_in_rmodmod} every such subobject is given by a functor of trajectories.
\end{proof}

The next two lemmata describe a useful way to obtain morphisms in $\Beh{R}{M}$.

\begin{lemma}\label{lemma:desirable_features_create_mor_via_monos}
Suppose given the following diagram in $\Beh{R}{M}$:
    \begin{center}
\begin{tikzpicture}[baseline=(base)]
        \coordinate (r) at (1.5,0);
        \coordinate (d) at (0,1);
        \node (A) {$\asBeh{\mathcal{G}}{M}$};
        \node (B) at ($(A)+(r)+(d)$) {$\asBeh{\mathcal{E}}{M}$};
        \node (C) at ($(A) + 2*(r)$) {$\asBeh{\mathcal{H}}{M}$};
        \node (base) at ($(A) + 0.5*(d)$) {};
        \draw[<-right hook,thick] (B) to node[above]{$\alpha$} (A);
        \draw[<-,thick] (B) to node[above,xshift=0.5em]{$\beta$} (C);
\end{tikzpicture}
\end{center}
with $\alpha$ a mono.
If there exists a morphism $\gamma$ of abelian groups such that we have a commutative triangle
\begin{center}
\begin{tikzpicture}[baseline=(base)]
        \coordinate (r) at (1.5,0);
        \coordinate (d) at (0,1);
        \node (A) {$\mathcal{G}({M})$};
        \node (B) at ($(A)+(r)+(d)$) {$\mathcal{E}({M})$};
        \node (C) at ($(A) + 2*(r)$) {$\mathcal{H}({M})$};
        \node (base) at ($(A) + 0.5*(d)$) {};
        \draw[<-right hook,thick] (B) to node[above]{$\alpha$} (A);
        \draw[<-,thick] (B) to node[above,xshift=0.5em]{$\beta$} (C);
        \draw[<-, thick] (A) to node[above]{$\gamma$} (C);
\end{tikzpicture}
\end{center}
then $\gamma$ is already a morphism of abstract behaviors
\[
\gamma: \asBeh{\mathcal{H}}{M} \rightarrow \asBeh{\mathcal{G}}{M}.
\]
\end{lemma}
\begin{proof}
Follows from Appendix \ref{lemma:left_exact_faithful_yields_triangles}.
\end{proof}

\begin{lemma}\label{lemma:desirable_features_create_mor_via_epis}
Suppose given the following diagram in $\Beh{R}{M}$:
    \begin{center}
\begin{tikzpicture}[baseline=(base)]
        \coordinate (r) at (1.5,0);
        \coordinate (d) at (0,-1);
        \node (A) {$\asBeh{\mathcal{G}}{M}$};
        \node (B) at ($(A)+(r)+(d)$) {$\asBeh{\mathcal{E}}{M}$};
        \node (C) at ($(A) + 2*(r)$) {$\asBeh{\mathcal{H}}{M}$};
        \node (base) at ($(A) + 0.5*(d)$) {};
        \draw[->>,thick] (B) to node[below]{$\beta$} (A);
        \draw[->,thick] (B) to node[below,xshift=0.5em]{$\alpha$} (C);
\end{tikzpicture}
\end{center}
with $\beta$ an epi.
If there exists a morphism $\gamma$ of abelian groups such that we have a commutative triangle
\begin{center}
\begin{tikzpicture}[baseline=(base)]
        \coordinate (r) at (1.5,0);
        \coordinate (d) at (0,-1);
        \node (A) {$\mathcal{G}({M})$};
        \node (B) at ($(A)+(r)+(d)$) {$\mathcal{E}({M})$};
        \node (C) at ($(A) + 2*(r)$) {$\mathcal{H}({M})$};
        \node (base) at ($(A) + 0.5*(d)$) {};
        \draw[->>,thick] (B) to node[below]{$\beta$} (A);
        \draw[->,thick] (B) to node[below,xshift=0.5em]{$\alpha$} (C);
        \draw[->, thick] (A) to node[above]{$\gamma$} (C);
\end{tikzpicture}
\end{center}
then $\gamma$ is already a morphism of abstract behaviors
\[
\gamma: \asBeh{\mathcal{G}}{M} \rightarrow \asBeh{\mathcal{H}}{M}.
\]
\end{lemma}
\begin{proof}
Dual to \ref{lemma:desirable_features_create_mor_via_monos}.
\end{proof}

We end this subsection with two corollaries that show that lifts of subgroups/quotient groups into the category $\Beh{R}{M}$ are necessarily unique, if they exist.

\begin{corollary}\label{corollary:desirable_features_at_most_one_lift_subobj}
Let $\asBeh{\mathcal{G}}{M} \in \Beh{R}{M}$. Suppose given an abelian subgroup
\[
U \hookrightarrow \mathcal{G}(M).
\]
Then there exists at most (up to equality of subojects) one subobject $\asBeh{\mathcal{U}}{M} \hookrightarrow\asBeh{\mathcal{G}}{M}$ that is mapped to $U \hookrightarrow \mathcal{G}(M)$ via the canonical functor from $\Beh{R}{M}$ to $\Ab$.
\end{corollary}
\begin{proof}
Follows from Appendix \ref{corollary:subobjects_fully_faithful}.
\end{proof}

\begin{corollary}\label{corollary:desirable_features_at_most_one_lift_quoobj}
Let $\asBeh{\mathcal{G}}{M} \in \Beh{R}{M}$. Suppose given an abelian quotient group
\[
\mathcal{G}(M) \twoheadrightarrow Q
\]
Then there exists at most (up to equality of quotient objects) one quotient object $\asBeh{\mathcal{G}}{M} \twoheadrightarrow \asBeh{\mathcal{Q}}{M}$ that is mapped to $\mathcal{G}(M) \twoheadrightarrow Q$ via the canonical functor from $\Beh{R}{M}$ to $\Ab$.
\end{corollary}
\begin{proof}
Dual to \ref{corollary:desirable_features_at_most_one_lift_subobj}.
\end{proof}

\subsection{The Serre quotient induced by the contravariant defect}\label{subsection:quo_ind_by_defect}

In this subsection, we show in \ref{corollary:beh_rmodop_equivalence} that for left coherent rings $R$ and fp-injective fp-cogenerators $M \in R\Modl$, we have an equivalence
\[
\Beh{R}{\M} \simeq (R\modl)^{\op}.
\]
This recovers the fundamental module-behaviour duality of algebraic systems theory.

\begin{para}\label{para:coherent_rings}
Recall that a ring $R$ is called \textbf{left coherent} if every finitely generated left submodule of ${_R}R$ is finitely presented.
In fact, this is equivalent to $R\modl$ being an abelian category.
In this case, the canonical embedding $R\modl \subseteq R\Modl$ is exact.
A constructive proof of these statements can for example be found in \cite[Theorem 4.6]{PosFreyd}.
Dually, $R$ is called \textbf{right coherent} if $R^{\op}$ is left coherent.
\end{para}

\begin{para}
Let $R$ be a left coherent ring.
Since $R\modl$ is abelian, the contravariant defect (see \ref{example:defect}) corestricts to a functor of type
\[
\defect: R\modl\modl \rightarrow (R\modl)^{\op}.
\]
\end{para}

\begin{theorem}\label{theorem:defect_adjunction}
Let $R$ be a left coherent ring.
Let
\begin{align*}
\yoneda: (R\modl)^{\op} &\rightarrow R\modl\modl \\
M &\mapsto \Hom(M,-)
\end{align*}
denote the contravariant Yoneda embedding. Then the contravariant defect is left adjoint to the contravariant Yoneda embedding:
\begin{equation}\label{eq:adj_setup}
\begin{tikzpicture}[mylabel/.style={fill=white}, baseline=(A)]
      \coordinate (r) at (6,0);
      \node (A) {$R\modl\modl$};
      \node (B) at ($(A) + (r)$) {$(R\modl)^{\op}$};
      \draw[->, out = 30, in = 180-30] (A) to node[mylabel]{$\defect$} (B);
      \draw[<-, out = -30, in = 180+30] (A) to node[mylabel]{$\yoneda$} (B);
      \node[rotate=90] (t) at ($(A) + 0.5*(r)$) {$\vdash$};
\end{tikzpicture}
\end{equation}
Moreover, the counit of this adjunction
\[
\defect(\yoneda(M)) \xrightarrow{\simeq} M
\]
is an isomorphism natural in $M \in R\modl$.
In particular, $\defect$ induces an equivalence of categories
\[
\frac{R\modl\modl}{\kernel( \defect )} \simeq (R\modl)^{\op}.
\]
\end{theorem}
\begin{proof}
The adjunction is a special case of \cite[Section 5, Lemma 8]{Russell16}.
The natural isomorphism
\[
\defect( \yoneda( M ) ) \cong \defect( \Hom( M, - ) ) \cong M
\]
for $M \in R\modl$ can be seen directly from the way we compute the defect in \ref{example:defect}.
The stated equivalence of categories follows from \ref{para:serre_quots_via_adjunctions}.
\end{proof}

Next, we identify the Serre quotient of \ref{theorem:defect_adjunction} with the category $\Beh{R}{\M}$ for a particular class of modules $M \in R\Modl$.

\begin{definition}
A module $M \in R\Modl$ is called \textbf{fp-injective} or \textbf{absolutely pure} if every short exact sequence in $R\Modl$ of the form
\begin{center}
       \begin{tikzpicture}[mylabel/.style={fill=white}]
        \coordinate (r) at (2,0);
        \node (A) {$0$};
        \node (B) at ($(A)+(r)$) {$N''$};
        \node (C) at ($(B) + (r)$) {$N'$};
        \node (D) at ($(C) + (r)$) {$N$};
        \node (E) at ($(D) + (r)$) {$0$};
        \draw[->,thick] (A) to (B);
        \draw[->,thick] (B) to (C);
        \draw[->,thick] (C) to (D);
        \draw[->,thick] (D) to (E);
        \end{tikzpicture}
\end{center}
with $N$ a finitely presented module is $\Hom(-,M)$-exact, i.e., the sequence
\begin{center}
       \begin{tikzpicture}[mylabel/.style={fill=white}]
        \coordinate (r) at (4,0);
        \node (A) {$0$};
        \node (B) at ($(A)+0.5*(r)$) {$\Hom(N'',M)$};
        \node (C) at ($(B) + 0.75*(r)$) {$\Hom(N',M)$};
        \node (D) at ($(C) + 0.75*(r)$) {$\Hom(N,M)$};
        \node (E) at ($(D) + 0.5*(r)$) {$0$};
        \draw[<-,thick] (A) to (B);
        \draw[<-,thick] (B) to (C);
        \draw[<-,thick] (C) to (D);
        \draw[<-,thick] (D) to (E);
        \end{tikzpicture}
\end{center}
is exact.
\end{definition}

\begin{remark}\label{remark:fp-injective_for_left_coherent}
It is easy to see that a module $M \in R\Modl$ is fp-injective if and only if $\Ext^{1}(N,M) \simeq 0$ for all finitely presented modules $N \in R\modl$ (see, e.g., \cite[Proposition 2.3.1]{PrestPSL}).

If $R$ is left coherent, then $M \in R\Modl$ is fp-injective if and only if every short exact sequence within the category of finitely presented modules
\begin{center}
       \begin{tikzpicture}[mylabel/.style={fill=white}]
        \coordinate (r) at (2,0);
        \node (A) {$0$};
        \node (B) at ($(A)+(r)$) {$N''$};
        \node (C) at ($(B) + (r)$) {$N'$};
        \node (D) at ($(C) + (r)$) {$N$};
        \node (E) at ($(D) + (r)$) {$0$};
        \draw[->,thick] (A) to (B);
        \draw[->,thick] (B) to (C);
        \draw[->,thick] (C) to (D);
        \draw[->,thick] (D) to (E);
        \end{tikzpicture}
\end{center}
is $\Hom(-,M)$-exact, i.e., mapped to an exact sequence by $\Hom(-,M)$.
This follows from the computation of $\Ext^{1}(N,M)$ for a finitely presented module $N$ that uses a projective resolution within $R\modl$.

Since $\Hom(-,M)$ is always left exact, we obtain the following useful characterization:
if $R$ is left coherent, then $M \in R\Modl$ is fp-injective if and only if for every monomorphism $N'' \rightarrow N' \in R\modl$ of finitely presented modules, the map
\[
\Hom(N', M) \rightarrow \Hom( N'', M)
\]
is surjective.
\end{remark}

\begin{definition}\label{definition:fpcogenerator}
A module $M \in R\Modl$ is called 
\begin{itemize}
    \item a \textbf{cogenerator} if
    \[
    \Hom(-,M): (R\Modl)^{\op} \rightarrow \Ab
    \]
    is a faithful functor. Here, the domain of the hom functor is given by all $R$-modules.
    \item an \textbf{fp-cogenerator}\footnote{This notion occurs for example in \cite{Garkusha}.} if
    \[
    \Hom(-,M): (R\modl)^{\op} \rightarrow \Ab
    \]
    is a faithful functor. Here, the domain of the hom functor is given only by finitely presented modules.
\end{itemize}
\end{definition}

\begin{lemma}\label{lemma:test_fpcogenerator}
Let $M \in R\Modl$.
If $M$ is an fp-cogenerator, then
\[
\Hom(N,M) \cong 0 \text{\hspace{1em} if and only if \hspace{1em}} N \cong 0
\]
for all $N \in R\modl$. The converse holds if $R$ is left coherent and $M$ is fp-injective. 
\end{lemma}
\begin{proof}
Follows from Appendix \ref{lemma:faithful_implies_zero_test}.
Note that if $R$ is left coherent, $R\modl$ is abelian (Paragraph \ref{para:coherent_rings}), and if $M$ is fp-injective, $\Hom(-,M): (R\modl)^{\op} \rightarrow \Ab$ is an exact functor by \ref{remark:fp-injective_for_left_coherent}.
\end{proof}

\begin{ex}\label{ex:von_neumann}
Let $R$ be a von Neumann regular ring, i.e., for every $r \in R$ there is an $s \in R$ such that $r = rsr$.
Von Neumann regular rings and their modules admit a rich structure theory \cite[Subsection 2.3.4]{PrestPSL}:
\begin{itemize}
    \item $R$ is left and right coherent.
    \item Every $M \in R\Modl$ is fp-injective.
    \item Every $M \in R\modl$ is projective.
\end{itemize}
It follows that ${_R}R$ is an fp-cogenerator (Lemma \ref{lemma:test_fpcogenerator}): since every $M \in R\modl$ is projective, it is a direct summand of ${{_R}R}^{n}$ for some $n \in \Nzero$ and thus
\[
\Hom(M, {_R}R) \cong 0 \hspace{1em}\text{if and only if}\hspace{1em} M \cong 0.
\]
We remark that $R$ is not necessarily a cogenerator. For example, let $K$ be a field, and let $V := \bigoplus_{i=1}^{\infty} K$ be a vector space of infinite dimension. Then $R := \End_{K}(V)$ is an example of a von Neumann regular ring that is injective as a module over itself \cite[Corollary 1.23]{Goodearl}.
Set
\[
M := R/\langle \psi \in R \mid \dim( \image(\psi) ) < \infty \rangle.
\]
Then
\begin{align*}
\Hom( M, R ) \cong \{ \chi \in R \mid \psi \circ \chi = 0 \text{~for all $\psi$ such that $\dim( \image(\psi) ) < \infty$}\} \cong 0
\end{align*}
but $M \not\cong 0$ since $V$ is infinite-dimensional.
\end{ex}

\begin{theorem}\label{theorem:kernel_of_defect}
Let $R$ be a left coherent ring and let $\M \in R\Modl$.
\begin{enumerate}
\item Then 
\[
\kernel( \defect ) \subseteq \kernel(\ev{\M})
\]
if and only if $\M$ is fp-injective.
\item Moreover, 
\[
\kernel( \defect ) = \kernel(\ev{\M})
\]
if and only if $\M$ is fp-injective and an fp-cogenerator.
\end{enumerate}
\end{theorem}
\begin{proof}
Let $\mathcal{G} \in R\modl\modl$ with a presentation given by $\alpha: N \rightarrow N' \in R\modl$.
Then by \ref{example:defect}, we have
\[
\defect( \mathcal{G} ) \cong \kernel( \alpha )
\]
which is zero if and only if $\alpha$ is a mono.
Thus, $\kernel( \defect )$ consists of those $\mathcal{G}$ that admit a presentation by a mono.
Now, the first statementfollows from our characterization of fp-injective modules in \ref{remark:fp-injective_for_left_coherent}.
In particular, whenever $M$ is fp-injective, we have a commutative diagram (up to natural isomorphism) of the following form, with the arrows provided by the canonical functors:
\begin{equation}\label{equation:defect_vs_eval}
       \begin{tikzpicture}[mylabel/.style={fill=white},baseline=($(A) + 0.5*(d)$)]
        \coordinate (r) at (2,0);
        \coordinate (d) at (0,-2);
        \node (A) {$R\modl\modl$};
        \node (B) at ($(A)-(r)+(d)$) {$\frac{R\modl\modl}{\kernel(\defect)}$};
        \node (C) at ($(B) + 2*(r)$) {$\frac{R\modl\modl}{\kernel(\ev{M})}$};
        \draw[->,thick] (A) to (B);
        \draw[->,thick] (A) to (C);
        \draw[->,thick] (B) to (C);
        \end{tikzpicture}
\end{equation}
By \ref{theorem:defect_adjunction}, $\frac{R\modl\modl}{\kernel(\defect)} \simeq (R\modl)^{\op}$, in particular, every object in $\frac{R\modl\modl}{\kernel(\defect)}$ is isomorphic to $\Hom(N,-)$ for some $N \in R\modl$. The functor given by the horizontal arrow in \eqref{equation:defect_vs_eval} maps $\Hom(N,-)$ to zero if and only if $\Hom(N,M)$ is zero (by \ref{lemma:desirable_features_Beh_zero}).
Thus, $\kernel( \defect ) = \kernel(\ev{\M})$ if and only if $M$ is an fp-cogenerator.
\end{proof}

\begin{corollary}[Module-behavior duality]\label{corollary:beh_rmodop_equivalence}
Let $R$ be a left coherent ring, and let $\M \in R\Modl$ be fp-injective and an fp-cogenerator. Then
\[
\Beh{R}{\M} \simeq (R\modl)^{\op}.
\]
\end{corollary}
\begin{proof}
This follows from \ref{theorem:kernel_of_defect} and \ref{theorem:defect_adjunction}.
\end{proof}

The following corollary corresponds to the second key problem described by Wood in \cite{WoodKey}.
It is a reformulation of the model theoretic fact that fp-injective left modules over left coherent rings have elimination of quantifiers - this is due to Eklof and Sabbagh \cite{ES70}.

\begin{corollary}[Elimination of latent variables]
Let $R$ be a left coherent ring, and let $\M \in R\Modl$ be fp-injective and an fp-cogenerator. Let $k \in \Nzero$.
Every subobject of $\asBeh{\ff}{M}^{\oplus k}$, i.e., of the $k$-fold direct sum of the abstract signal space, is of the form
\[
\asBeh{\mathcal{B}( \phi, - )}{M} \xrightarrow{\iota_{M}} \asBeh{\ff}{M}^{\oplus k}
\]
for $\phi$ a \emph{quantifier-free} pp formula and $\iota$ the canonical inclusion ${\mathcal{B}( \phi, - )} \hookrightarrow \ff^{\oplus k}$.
\end{corollary}
\begin{proof}
By \ref{corollary:beh_rmodop_equivalence}, subobjects of $\asBeh{\ff}{M}^{\oplus k}$ correspond to quotient objects $R^{1 \times k} \twoheadrightarrow N$, which in turn translate to subobjects of the form \[
\asBeh{\Hom( N, - )}{M} \hookrightarrow \asBeh{\ff}{M}^{\oplus k}.
\]
The claim follows.
\end{proof}

\begin{ex}
Let $R := \R[\partial_1, \dots, \partial_n]$ act on $M := \mathcal{C}^{\infty}(\R^n,\R)$, where $n \in \Nzero$ (see \ref{ex:pdes}).
Then $M$ is an injective cogenerator. This powerful result is due to Oberst \cite{Ob}.
In particular, we can use the module-behavior duality \ref{corollary:beh_rmodop_equivalence} as a tool for working with linear partial differential equations with constant coefficients when we want the solutions to lie in $M$.
\end{ex}

\begin{ex}
Let $R$ be a von Neumann regular ring.
Then $_{R}R$ is fp-injective and an fp-cogenerator by \ref{ex:von_neumann}.
Thus, we have
\[
\Beh{R}{_{R}R} \simeq (R\modl)^{\op}
\]
by \ref{corollary:beh_rmodop_equivalence}.
Note that by \ref{ex:von_neumann}, $_{R}R$ is not necessarily a cogenerator.
\end{ex}

\begin{remark}[Inclusion of behaviors]\label{remark:first_key_problem}
We close this subsection with a remark about the first key problem introduced by Wood in \cite{WoodKey}.
Let $R$ be an arbitrary ring and $M \in R\Modl$.
Note that $R\modl$ is not necessarily an abelian category.
The following holds:
\begin{enumerate}
    \item The functor $\Hom(-,M)$ respects kernels (it maps cokernels in $R\modl$ to kernels in $\Ab$).
    \item Each subobject $U \hookrightarrow N$ in $(R\modl)^{\op}$ arises as the kernel of a morphism. 
    This means that every quotient $N' \in R\modl$ of $N$ arises as the cokernel of a morphism in $R\modl$.
    And this always holds since any short exact sequence of $R$-modules
    \begin{center}
       \begin{tikzpicture}[mylabel/.style={fill=white}]
        \coordinate (r) at (2,0);
        \node (A) {$0$};
        \node (B) at ($(A)+(r)$) {$N''$};
        \node (C) at ($(B) + (r)$) {$N$};
        \node (D) at ($(C) + (r)$) {$N'$};
        \node (E) at ($(D) + (r)$) {$0$};
        \draw[->,thick] (A) to (B);
        \draw[->,thick] (B) to (C);
        \draw[->,thick] (C) to (D);
        \draw[->,thick] (D) to (E);
        \end{tikzpicture}
    \end{center}
    with $N'$ finitely presented and $N$ finitely generated implies that $N''$ is finitely generated \cite[{Tag 0517}]{stacks-project}.
\end{enumerate}

Moreover, if $M$ is an fp-cogenerator, we have:
\begin{enumerate}
    \setcounter{enumi}{2}
    \item The functor
    \[
    \Hom(-,M): (R\modl)^{\op} \rightarrow \Ab
    \]
    is faithful.
\end{enumerate}

By Appendix \ref{corollary:subobjects_fully_faithful}, these three categorical properties suffice to see that inclusion of behaviors can be decided by mere computations within $(R\modl)^{\op}$:
for $N \in R\modl$ and $\epsilon': N \twoheadrightarrow N'$, $\epsilon'': N \twoheadrightarrow N''$ two quotient objects, we have an inclusion
\[
\Hom( N', M ) \subseteq \Hom( N'',M )
\]
as subobjects of $\Hom( N, M )$ if and only if there exists an $\alpha: N'' \rightarrow N' \in R\modl$ such that
\[
\epsilon' = \alpha \circ \epsilon''.
\]
\end{remark}

\subsection{The Serre quotient induced by the covariant defect}

The discussion in this subsection is dual to the one in Subsection \ref{subsection:quo_ind_by_defect}.

\begin{para}
Let $R$ be a right coherent ring.
Since $\modr R$ is abelian, taking the covariant defect (see \ref{example:covariant_defect}) corestricts to a functor of type
\[
\codefect: R\modl\modl \rightarrow (\modr R)
\]
\end{para}

For our convenience, we state the dual version of \ref{theorem:defect_adjunction}.

\begin{theorem}\label{theorem:codefect_adjunction}
Let $R$ be a right coherent ring.
Let
\begin{align*}
\mathrm{Tensor}: (\modr R) &\rightarrow R\modl\modl \\
N &\mapsto (N \otimes -)
\end{align*}
denote the functor that maps a finitely presetend module to its tensor functor.
Then taking the covariant defect is right adjoint to $\mathrm{Tensor}$:
\begin{equation}\label{eq:codefect_adj_setup}
\begin{tikzpicture}[mylabel/.style={fill=white}, baseline=(A)]
      \coordinate (r) at (6,0);
      \node (A) {$R\modl\modl$};
      \node (B) at ($(A) + (r)$) {$(\modr R)$};
      \draw[->, out = 30, in = 180-30] (A) to node[mylabel]{$\codefect$} (B);
      \draw[<-, out = -30, in = 180+30] (A) to node[mylabel]{$\mathrm{Tensor}$} (B);
      \node[rotate=-90] (t) at ($(A) + 0.5*(r)$) {$\vdash$};
\end{tikzpicture}
\end{equation}
Moreover, the unit of this adjunction
\[
N \xrightarrow{\simeq} \codefect( N \otimes - )
\]
is an isomorphism natural in $N \in \modr R$.
In particular, the covariant defect $\codefect$ induces an equivalence of categories
\[
\frac{R\modl\modl}{\kernel( \codefect )} \simeq (\modr R).
\]
\end{theorem}
\begin{proof}
Apply \ref{theorem:defect_adjunction} to $R^{\op}$, use $\modr R \simeq R^{\op}\modl$, the Auslander-Gruson-Jensen duality $R\modl\modl \simeq (R^{\op}\modl\modl)^{\op}$, and the fact that this duality maps hom functors to tensor functors, see \ref{para:duality_hom_tensor}.
\end{proof}

Analogous to Subsection \ref{subsection:quo_ind_by_defect}, we recognize the Serre quotient of \ref{theorem:codefect_adjunction} as being of the form $\Beh{R}{\M}$ for a particular class of modules.

We recall the definition of flat modules.
\begin{definition}
A module $M \in R\Modl$ is called \textbf{flat} if
\[
(- \otimes M): \Modr R \rightarrow \Ab
\]
is an exact functor. Here, the domain of the tensor functor is given by all $R$-modules.
\end{definition}

\begin{remark}\label{remark:flat}
The functor $(- \otimes M)$ is always right exact. Thus, $M$ is flat if and only if $(- \otimes M)$ preserves monomorphisms. Moreover, it suffices to test this only for inclusions of finitely generated right ideals $I$ into $R$ \cite[Proposition 10.6]{Stenstroem}.

From this, we obtain the following useful characterization: if $R$ is right coherent, then $M \in R\Modl$ is flat if and only if for every monomorphism $N'' \rightarrow N' \in \modr R$ of finitely presented right modules, the map
\[
N'' \otimes M \rightarrow N' \otimes M
\]
is injective.
\end{remark}

\begin{definition}\label{definition:fp_faithfully_flat}
A module $M \in R\Modl$ is called \textbf{fp-faithfully flat} if it is flat and if 
\[
(- \otimes M): \modr R \rightarrow \Ab
\]
is a faithful functor. Here, the domain of the tensor functor is given only by finitely presented modules.
\end{definition}

\begin{lemma}\label{lemma:test_fpfaithfullyflat}
Let $R$ be a right coherent ring.
A flat module $M \in R\Modl$ is fp-faithfully flat if and only if the following holds:
\[
(N \otimes M) \cong 0 \hspace{1em}\text{if and only if}\hspace{1em} N \cong 0
\]
for all $N \in \modr{R}$.
\end{lemma}
\begin{proof}
Follows from Appendix \ref{lemma:faithful_implies_zero_test}.
Note that if $R$ is right coherent, $\modr R$ is abelian (Paragraph \ref{para:coherent_rings}), and if $M$ is flat, $(- \otimes M): \modr R \rightarrow \Ab$ is an exact functor.
\end{proof}

\begin{remark}
There is the common notion of a module $M \in R\Modl$ being \textbf{faithfully flat}, which means that $M$ is flat and that the functor $(- \otimes M)$ on the whole module category $R\Modl$ is faithful.
For example, ${_R}R$ is faithfully flat.
However, we will only need the weaker version given in \ref{definition:fp_faithfully_flat}, a version which the author could not find in the literature.
\end{remark}

\begin{ex}
We give an example of an fp-faithfully flat module which is not faithfully flat.
Let
\[
R := k[x_i \mid i \in \N]
\]
be the commutative polynomial ring over a field $k$ in infinitely many variables.
Set
\[
M := \bigoplus_{i \in \N} R[\frac{1}{x_i}]
\]
where $R[\frac{1}{x_i}]$ denotes the localization of $R$ at the element $x_i$ for $i \in \N$.
Then each $R[\frac{1}{x_i}]$ is a flat $R$-module (since it is a localization of $R$), and $M$ is flat as a filtered colimit of flat modules.
Moreover, $M$ is not faithfully flat: if we define the ideal
\[
\mathfrak{m} := \langle x_i \mid i \in \N \rangle \subseteq R,
\]
then
\[
({R}/{\mathfrak{m}}) \otimes M \cong \bigoplus_{i \in \N} \left( ({R}/{\mathfrak{m}}) \otimes R[\frac{1}{x_i}]\right) \cong 0
\]
since $x_i$ acts on $({R}/{\mathfrak{m}}) \otimes R[\frac{1}{x_i}]$ both as zero and as an automorphism for all $i \in \N$.

But $M$ is fp-faithfully flat. To see this,
let the matrix $A \in R^{m \times n}$ present the module $N \in \modr R$ for $m,n \in \Nzero$, i.e., 
\[
N \cong \cokernel( R^{n \times 1} \xrightarrow{A} R^{m \times 1}).
\]
Let $J \subseteq N$ be the subset of all $j \in \N$ such that $x_j$ occurs as a variable in one of the polynomial entries of $A$. Note that $J$ is finite.
We define the subring
\[
R_J := k[ x_j \mid j \in J] \subseteq R
\]
By definition of $J$, the matrix $A$ also gives rise to a map
\[
R_J^{n \times 1} \xrightarrow{A} R_J^{m \times 1}.
\]
Applying $(- \otimes_{R_J} R)$ to that map gives us back
\[
R^{n \times 1} \xrightarrow{A} R^{m \times 1}
\]
(up to isomorphism). It follows that we have
\[
N \cong \cokernel( (R_J^{n \times 1} \xrightarrow{A} R_J^{m \times 1}) \otimes_{R_J} R ) \cong \cokernel( R_J^{n \times 1} \xrightarrow{A} R_J^{m \times 1}) \otimes_{R_J} R
\]
since the tensor product is right exact.
We set
\[
N' := \cokernel( R_J^{n \times 1} \xrightarrow{A} R_J^{m \times 1})
\]
and note that, since $R$ is a free $R_J$-module,
$N'$ is zero if and only if $N$ is zero.
Last, we compute
\begin{align*}
N \otimes M &\cong \bigoplus_{i \in \N} N \otimes R[\frac{1}{x_i}] \\
&\cong \bigoplus_{i \in \N} \big(N' \otimes_{R_J} R \otimes_R R[\frac{1}{x_i}] \big) \\
&\cong \bigoplus_{i \in \N} \big(N' \otimes_{R_J} R[\frac{1}{x_i}]\big)
\end{align*}
But now, for all $i \not\in J$, $R[\frac{1}{x_i}]$ is a free $R_J$-module. For those $i$, the summand $N' \otimes_{R_J} R[\frac{1}{x_i}]$ is zero if and only if $N'$ is zero. Now, the claim follows from \ref{lemma:test_fpfaithfullyflat}.
\end{ex}

\begin{theorem}\label{theorem:kernel_of_codefect}
Let $R$ be a right coherent ring and let $\M \in R\Modl$.
\begin{enumerate}
\item Then 
\[
\kernel( \codefect ) \subseteq \kernel(\ev{\M})
\]
if and only if $\M$ is flat.
\item Moreover, 
\[
\kernel( \codefect ) = \kernel(\ev{\M})
\]
if and only if $\M$ is fp-faithfully flat.
\end{enumerate}
\end{theorem}
\begin{proof}
We remark that the proof is dual to the proof of \ref{theorem:kernel_of_defect}. For the convenience of the reader, we spell it out: we need work with copresentations instead of presentations, i.e., for $\mathcal{G} \in R\modl\modl$ we choose an $\alpha: N \rightarrow N' \in \modr R$ such that
\[
\mathcal{G} \simeq \kernel( (N \otimes -) \xrightarrow{(\alpha \otimes -)} (N' \otimes -) ).
\]
Then by \ref{example:covariant_defect}, we have
\[
\codefect( \mathcal{G} ) \cong \kernel( \alpha )
\]
which is zero if and only if $\alpha$ is a mono.
Thus, $\kernel( \codefect )$ consists of those $\mathcal{G}$ that admit a copresentation by a mono.
Now, the first statementfollows from our characterization of flat modules in \ref{remark:flat}.
In particular, whenever $M$ is flat, we have a commutative diagram (up to natural isomorphism) of the following form:
\begin{equation}\label{equation:codefect_vs_eval}
       \begin{tikzpicture}[mylabel/.style={fill=white},baseline=($(A) + 0.5*(d)$)]
        \coordinate (r) at (2,0);
        \coordinate (d) at (0,-2);
        \node (A) {$R\modl\modl$};
        \node (B) at ($(A)-(r)+(d)$) {$\frac{R\modl\modl}{\kernel(\codefect)}$};
        \node (C) at ($(B) + 2*(r)$) {$\frac{R\modl\modl}{\kernel(\ev{M})}$};
        \draw[->,thick] (A) to (B);
        \draw[->,thick] (A) to (C);
        \draw[->,thick] (B) to (C);
        \end{tikzpicture}
\end{equation}
By \ref{theorem:codefect_adjunction}, $\frac{R\modl\modl}{\kernel(\codefect)} \simeq (\modr R)$, in particular, every object $\frac{R\modl\modl}{\kernel(\codefect)}$ is isomorphic to $(N \otimes -)$ for some $N \in \modr R$. The functor given by the horizontal arrow in \eqref{equation:codefect_vs_eval} maps $(N \otimes -)$ to zero if and only if $(N \otimes M)$ is zero (by \ref{lemma:desirable_features_Beh_zero}).
Thus, $\kernel( \codefect ) = \kernel(\ev{\M})$ if and only if $M$ is fp-faithfully flat.
\end{proof}

\begin{corollary}\label{corollary:equivalence_beh_modR}
Let $R$ be a right coherent ring, and let $\M \in R\Modl$ be an fp-faithfully flat module. Then
\[
\Beh{R}{\M} \simeq \modr R.
\]
\end{corollary}
\begin{proof}
This follows from \ref{theorem:kernel_of_codefect} and \ref{theorem:codefect_adjunction}.
\end{proof}

\begin{remark}
If $R$ is not right coherent, $\modr R$ is not abelian, and hence cannot be equivalent to the abelian category $\Beh{R}{M}$, even when $M$ is fp-faithfully flat. In \cite[Lemma 6.3]{PrestRavi} , it is shown that in this case, $\Beh{R}{M}$ can be characterized as the smallest abelian subcategory (not necessarily full) of $\Modr R$ that contains $\modr R$.
\end{remark}

\subsection{Characterization by a universal property}

In the following lemma, we enlist data and properties which will turn out to characterize $\Beh{R}{\M}$.
\begin{lemma}\label{lemma:data_of_beh}
Let $\M \in R\Modl$.
\begin{enumerate}
    \item The category $\Beh{R}{\M}$ is abelian.
    \item The functor $\Beh{R}{\M} \rightarrow \Ab: \asBeh{\mathcal{G}}{\M} \mapsto \mathcal{G}(\M)$ is faithful and exact.
    \item The additive functor 
    \begin{align*}
    R &\rightarrow \Beh{R}{\M} \\
    \bullet &\mapsto \asBeh{\ff}{\M}
    \end{align*}
    i.e., the functor defined by the following composition of canonical functors
    \[
    R \rightarrow R\modl\modl \rightarrow \Beh{R}{\M}
    \]
    satisfies
    \[
    (R \rightarrow \Beh{R}{\M} \rightarrow \Ab) \simeq (R \xrightarrow{\M} \Ab).
    \]
\end{enumerate}
\end{lemma}
\begin{proof}
The category $\Beh{R}{\M}$ is abelian since it is constructed as a Serre quotient.
In \ref{remark:morphisms_of_beh}, we already saw exactness and faithfulness of $\Beh{R}{\M} \rightarrow \Ab$.
Last, the given equation is satisfied by the universal property of $R\modl\modl$ (see \ref{theorem:up_Rmodmod}).
\end{proof}

\begin{theorem}[Universal property of $\Beh{R}{\M}$]\label{theorem:up_of_beh}
Let $\M \in R\Modl$. Suppose given the following data (which we also call a \textbf{context for behaviors over $M$}):
\begin{enumerate}
    \item an abelian category $\AC$,
    \item a faithful and exact functor $\AC \xrightarrow{I} \Ab$,
    \item an additive functor $R \xrightarrow{F} \AC$ such that
    \[
    (R \xrightarrow{F} \AC \xrightarrow{I} \Ab) \simeq (R \xrightarrow{\M} \Ab).
    \]
\end{enumerate}
Then there exists an exact functor $\Beh{R}{\M} \rightarrow \AC$ (unique up to natural isomorphism) such that the following diagram commutes (up to natural isomorphism):
\begin{equation}\label{eq:up_of_beh}
       \begin{tikzpicture}[mylabel/.style={fill=white}, baseline=(R)]
        \coordinate (r) at (3,0);
        \coordinate (d) at (0,-0.75);
        \node (R) {$R$};
        \node (A) at ($(R)+(r)-(d)$) {$\Beh{R}{\M}$};
        \node (B) at ($(R)+(r)+(d)$) {$\AC$};
        \node (Ab) at ($(R)+2*(r)$) {$\Ab$};
        \draw[->,thick] (R) to (A);
        \draw[->,thick] (R) to (B);
        \draw[->,thick] (A) to (Ab);
        \draw[->,thick] (B) to (Ab);
        \draw[->,thick,dashed] (A) to (B);
        \end{tikzpicture}
\end{equation}
In other words, the data in \ref{lemma:data_of_beh} are an initial context for behaviors over $M$.
\end{theorem}
\begin{proof}
We make use of the following diagram, which also serves the purpose of temporarily giving names to all functors involved in the argument:
\begin{center}
       \begin{tikzpicture}[mylabel/.style={fill=white}]
        \coordinate (r) at (3.5,0);
        \coordinate (d) at (0,-1);
        \node (Ronly) {$R$};
        \node (R) at ($(Ronly)+(r)$){$R\modl\modl$};
        \node (A) at ($(R)+(r)-(d)$) {$\Beh{R}{\M}$};
        \node (B) at ($(R)+(r)+(d)$) {$\AC$};
        \node (Ab) at ($(R)+2*(r)$) {$\Ab$};
        \draw[->,thick] (R) to node[mylabel]{$Q$}(A);
        \draw[->,thick] (R) to node[mylabel]{$\ev{F}$}(B);
        \draw[->,thick] (A) to node[mylabel]{$S$}(Ab);
        \draw[->,thick] (B) to node[mylabel]{$I$}(Ab);
        \draw[->,thick] (Ronly) to node[above]{$J$}(R);
        \draw[->,thick,out=-25,in=-180] (Ronly) to node[mylabel]{$F$}(B);
        \draw[->,thick,out=45,in=180-45] (R) to node[mylabel]{$\ev{\M}$}(Ab);
        \draw[->,thick,out=-45,in=-180+45] (Ronly) to node[above]{$\M$}(Ab);
        \draw[->,thick,dashed] (A) to node[mylabel]{$G$}(B);
        \end{tikzpicture}
\end{center}
First, we compute that
\begin{align*}
\ev{\M} \circ J &\simeq M \\
&\simeq I \circ F \simeq I \circ \ev{F} \circ J
\end{align*}
from which we cancel $J$ on the right (by the universal property of $R\modl\modl$) and deduce
\[
\ev{\M} \simeq I \circ \ev{F}.
\]
Since $I$ is faithful, it follows that
\[
\kernel( \ev{M} ) \subseteq \kernel( \ev{F} ) 
\]
which gives us the functor $G$ in our diagram above (by the universal property of Serre quotients) which satisfies
\[
G \circ Q \simeq \ev{F}.
\]
Composition with $J$ yields
\[
G \circ Q \circ J \simeq \ev{F} \circ J \simeq F.
\]
This is the commutativity of the left triangle in \ref{eq:up_of_beh}.
For proving the commutativity of the right triangle in \ref{eq:up_of_beh}, we calculate
\begin{align*}
S \circ Q &\simeq \ev{\M} \\
&\simeq I \circ \ev{F} \simeq I \circ G \circ Q
\end{align*}
and cancel $Q$ from the right (by \ref{remark:can_quotient_is_epi}) to obtain
\[
S \simeq I \circ G
\]
and thus the claim.
\end{proof}

\begin{remark}
The induced functor in \ref{theorem:up_of_beh} is always faithful.
\end{remark}

\begin{ex}\label{example:comparision_with_oberst}
Let $M \in R\Modl$ and let $\End(M)$ be the ring of $R$-module endomorphisms of $M$. 
We remark that modules over $\End(M)$ also provide a context for behaviors over $M$:
\begin{enumerate}
 \item $\End(M)\Modl$ is an abelian category,
    \item the functor $\End(M)\Modl \rightarrow \Ab$ which forgets the module structure is faithful and exact, 
    \item we have an additive functor
        \begin{align*}
        R &\rightarrow \End(M)\Modl \\
        (\bullet \xrightarrow{r} \bullet) &\mapsto (M \xrightarrow{x \mapsto rx} M)
        \end{align*}
        such that
    \[
    (R \rightarrow \End(M)\Modl \rightarrow \Ab) \simeq (R \xrightarrow{\M} \Ab).
    \]   
\end{enumerate}
It follows from \ref{theorem:up_of_beh} that we have a faithful and exact functor
\[
\Beh{R}{M} \rightarrow \End(M)\Modl
\]
such that the diagram
\begin{center}
       \begin{tikzpicture}[mylabel/.style={fill=white}, baseline=(R)]
        \coordinate (r) at (3,0);
        \coordinate (d) at (0,-0.75);
        \node (R) {$R$};
        \node (A) at ($(R)+(r)-(d)$) {$\Beh{R}{\M}$};
        \node (B) at ($(R)+(r)+(d)$) {$\End(M)\Modl$};
        \node (Ab) at ($(R)+2*(r)$) {$\Ab$};
        \draw[->,thick] (R) to (A);
        \draw[->,thick] (R) to (B);
        \draw[->,thick] (A) to (Ab);
        \draw[->,thick] (B) to (Ab);
        \draw[->,thick] (A) to (B);        
        \end{tikzpicture}
\end{center}    
commutes (up to natural isomorphism).
This functor does not need to be full in general.
For example, let $R := \Z_{\langle p \rangle}$ be the ring of integers localized at the maximal ideal spanned by a prime $p \in \Z$.
Let $M := \Z[\frac{1}{p}]/\Z$ be the Prüfer group regarded as an $R$-module. Note that $M$ is an injective cogenerator and consequently $\Beh{R}{M} \simeq (R\modl)^{\op}$ by \ref{corollary:beh_rmodop_equivalence}.
In particular, $\End_{\Beh{R}{M}}( \asBeh{\ff}{M} ) \cong R$. More concretely, every endomorphism of $\asBeh{\ff}{M}$ is given by multiplication with an $r \in R$.

On the other hand, $\End(M)$ is given by the $p$-adic integers $\Z_p$.
Moreover, we have an injective map
\begin{align*}
\Z_p &\rightarrow \End_{\End(M)}( M )\\
r &\mapsto (x \mapsto rx).
\end{align*}
It follows that the image of the map
\[
\End_{\Beh{R}{M}}( \asBeh{\ff}{M} ) \rightarrow \End_{\End(M)}( M )
\]
which is induced by the functor $\Beh{R}{M} \rightarrow \End(M)\Modl$, only consists of those endomorphisms that are induced by multiplication with an $r \in R \subsetneq \Z_p$ and hence is not surjective.
\end{ex}

\subsection{Change of rings}

We investigate the change of rings as a tool to compare ambient categories for behaviors over modules which are defined over different rings.

\begin{construction}\label{construction:functor_change_of_rings}
Let $R \xrightarrow{\ringmap} S$ be a morphism of rings.
Let $M \in S\Modl$.
We denote the restriction of scalars by $M_{|\ringmap} \in R\Modl$, i.e., $M_{|\ringmap}$ is the module whose underlying abelian group is given by $M$, and $r \in R$ acts on $M$ via $\ringmap(r) \in S$.
In our functorial language (see \ref{para:modules_as_functors}), $M_{|\ringmap}$ is defined by the functor
\[
R \xrightarrow{\ringmap} S \xrightarrow{M} \Ab.
\]
In this situation, we have a commutative square of functors (up to natural isomorphism):
\begin{center}
    \begin{tikzpicture}[mylabel/.style={fill=white}, baseline=(R)]
    \coordinate (r) at (4,0);
    \coordinate (d) at (0,-1);
    \node (R) {$R$};
    \node (S) at ($(R)+0.5*(r)+0.5*(d)$) {$S$};
    \node (A) at ($(R)+(r)-(d)$) {$\Beh{R}{M_{|\ringmap}}$};
    \node (B) at ($(R)+(r)+(d)$) {$\Beh{S}{M}$};
    \node (Ab) at ($(R)+2*(r)$) {$\Ab$};
    \draw[->,thick] (R) to node[mylabel]{$\ringmap$} (S);
    \draw[->,thick] (R) to (A);
    \draw[->,thick] (S) to (B);
    \draw[->,thick] (A) to (Ab);
    \draw[->,thick] (B) to (Ab);
    \end{tikzpicture}
\end{center}
Thus, by the universal property of $\Beh{R}{M_{|\ringmap}}$ (see \ref{theorem:up_of_beh}), we get an exact and faithful functor $\ringmap_{\ast}$ which renders the following diagram commutative:
\begin{center}
    \begin{tikzpicture}[mylabel/.style={fill=white}, baseline=(R)]
    \coordinate (r) at (4,0);
    \coordinate (d) at (0,-1);
    \node (R) {$R$};
    \node (S) at ($(R)+0.5*(r)+0.5*(d)$) {$S$};
    \node (A) at ($(R)+(r)-(d)$) {$\Beh{R}{M_{|\ringmap}}$};
    \node (B) at ($(R)+(r)+(d)$) {$\Beh{S}{M}$};
    \node (Ab) at ($(R)+2*(r)$) {$\Ab$};
    \draw[->,thick] (R) to node[mylabel]{$\ringmap$} (S);
    \draw[->,thick] (R) to (A);
    \draw[->,thick] (S) to (B);
    \draw[->,thick] (A) to (Ab);
    \draw[->,thick] (B) to (Ab);
    \draw[->,thick,dashed] (A) to node[right]{$\ringmap_{\ast}$}(B);
    \end{tikzpicture}
\end{center}
\end{construction}

We will need the next lemma in our proof of \ref{theorem:epi_induces_equivalence}.

\begin{lemma}\label{lemma:prest_epis}
Let $\ringmap: R \rightarrow S$ be a morphism of rings.
We denote by $F$ the following composition of functors:
\[
R \xrightarrow{\ringmap} S \rightarrow S\modl\modl.
\]
Then $\ringmap$ is an epimorphism (in the category of rings) if and only if the functor 
\[\ev{F}: R\modl\modl \rightarrow S\modl\modl\]
(induced by the universal property of $R\modl\modl$, see \ref{theorem:up_Rmodmod}) induces an equivalence of categories
\[
\frac{R\modl\modl}{\kernel(\ev{F})} \simeq S\modl\modl.
\]
\end{lemma}
\begin{proof}
See \cite[Theorem 12.8.7]{PrestPSL}.
\end{proof}

\begin{theorem}\label{theorem:epi_induces_equivalence}
Let $\ringmap: R \rightarrow S$ be an epimorphism of rings, and let $\M \in S\Modl$.
Then
\[
\ringmap_{\ast}: \Beh{R}{\M_{|\ringmap}} \xrightarrow{\sim} \Beh{S}{M}
\]
is an equivalence of categories.
\end{theorem}
\begin{proof}
We consider the following diagram:
\begin{center}
   \begin{tikzpicture}[mylabel/.style={fill=white}]
    \coordinate (r) at (3.5,0);
    \coordinate (d) at (0,-3);
    \node (R) {$R$};
    \node (S) at ($(R)+2*(r)$) {$S$};
    \node (AbR) at ($(R)+(d)$) {$R\modl\modl$};
    \node (Z) at ($(AbR)+(r)$) {$\frac{R\modl\modl}{\kernel({\ev{F}})}$};
    \node (AbS) at ($(Z)+(r)$) {$S\modl\modl$};
    \node (BehR) at ($(AbR) + (d) + 0.5*(r)$) {$\Beh{R}{\M_{|\ringmap}}$};
    \node (BehS) at ($(AbS) + (d) - 0.5*(r)$) {$\Beh{S}{\M}$};
    \node (Ab) at ($(Z) + 2*(d)$) {$\Ab$};

    \draw[->,thick] (R) to node[above]{$\ringmap$}(S);
    \draw[->,thick] (AbR) to (Z);
    \draw[->,thick] (Z) to node[above]{$\simeq$}(AbS);
    \draw[->,thick,out=45,in=180-45] (AbR) to node[mylabel]{$\ev{F}$} (AbS);
    \draw[->,thick] (BehR) to node[above]{$\ringmap_{\ast}$}(BehS);
    \draw[->,thick] (BehR) to (Ab);
    \draw[->,thick] (BehS) to (Ab);
    
    \draw[->,thick] (R) to (AbR);
    \draw[->,thick] (S) to (AbS);
    \draw[->,thick] (AbR) to (BehR);
    \draw[->,thick] (AbS) to (BehS);
    
    \draw[->,thick,out=-180+45,in=180] (AbR) to node[mylabel]{$\ev{\M_{|\ringmap}}$}(Ab);
    \draw[->,thick,out=-45,in=0] (AbS) to node[mylabel]{$\ev{\M}$} (Ab);

    \draw[->,thick,dashed] (Z) to (BehR);
    
    \end{tikzpicture}
\end{center}
How to read this diagram:
$F$ denotes the functor given by the composition
\[
R \xrightarrow{\ringmap} S \rightarrow S\modl\modl,
\]
and $\ev{F}$ is the functor induced by the universal property of $R\modl\modl$.
The arrow marked as an equivalence is due to \ref{lemma:prest_epis}.
The unlabeled solid arrows are either canonical quotient functors or the canonical functors from ambient categories for behaviors to $\Ab$.
All inner shapes consisting of solid arrows are easily seen to be commutative (up to natural isomorphism). It follows that
\[
\ev{\M} \circ \ev{F} \simeq \ev{\M_{|\ringmap}}
\]
which implies
\[
\kernel( \ev{F} ) \subseteq \kernel( \ev{\M_{|\ringmap}} )
\]
which in turn implies the existence of the dashed arrow that renders the whole diagram commutative (up to natural isomorphism).
It follows that we can identify the composite
\[
S\modl\modl \xrightarrow{\simeq} \frac{R\modl\modl}{\kernel({\ev{F}})} \dashrightarrow \Beh{R}{\M_{|\ringmap}}
\]
with a canonical quotient functor of $S\modl\modl$, and since this canonical quotient functor is linked via a faithful functor $\ringmap_{\ast}$ to the canonical quotient functor
\[
S\modl\modl \rightarrow \Beh{S}{\M}
\]
the functor $\ringmap_{\ast}$ needs to be an equivalence.
\end{proof}

\begin{corollary}
Let $\ringmap: R \rightarrow R/I$ be the canonical ring epimorphism induced by a two-sided ideal $I \subseteq R$.
If $M \in R/I\Modl$, then 
\[
\ringmap_{\ast}: \Beh{R}{M_{|\ringmap}} \xrightarrow{\sim} \Beh{R/I}{M}
\]
is an equivalence of categories.
\end{corollary}
\begin{proof}
Since $\ringmap: R \rightarrow R/I$ is an epimorphism, we can apply \ref{theorem:epi_induces_equivalence}.
\end{proof}

\begin{para}
    For the next corollary, we need the localization of an arbitrary ring $R$ at an arbitrary subset $S \subseteq R$.
    Such a notion can be defined by means of a universal property.
    The resulting canonical map $R \rightarrow S^{-1}R$ is always an epimorphism in the category of rings, even though it is not necessarily a surjective map on the level of sets.
\end{para}

\begin{corollary}
Let $\ringmap: R \rightarrow S^{-1}R$ be the canonical ring morphism induced by a multiplicatively closed subset $S \subseteq R$.
If $M \in S^{-1}R\Modl$, then 
\[
\ringmap_{\ast}: \Beh{R}{M_{|\ringmap}} \xrightarrow{\sim} \Beh{S^{-1}R}{M}
\]
is an equivalence of categories.
\end{corollary}
\begin{proof}
Since $\ringmap: R \rightarrow S^{-1}R$ is an epimorphism, we can apply \ref{theorem:epi_induces_equivalence}.
\end{proof}

Again, as a preparation of the next \ref{theorem:compute_beh_via_ring_of_def_scalars}, we need the following lemma.

\begin{lemma}\label{lemma:prest_equivalence_ring_of_def_scalars}
Let 
\begin{itemize}
    \item $\CC \subseteq R\modl\modl$ be a Serre subcategory,
    \item $R_{\CC}$ denote the ring of endomorphisms of $\ff$ (the forgetful functor $R\Modl \rightarrow \Ab$) regarded as an object in the Serre quotient $\frac{R\modl\modl}{\CC}$,
    \item $G$ denote the functor given by the natural inclusion $R_{\CC} \rightarrow \frac{R\modl\modl}{\CC}$,
    \item $F$ denote the functor given by the composition of functors $R \rightarrow R_{\CC} \rightarrow R_{\CC}\modl\modl$.
\end{itemize}
Then we have the following commutative diagram (up to natural isomorphism) with the bottom horizontal arrow an equivalence of categories:
\begin{center}
   \begin{tikzpicture}[mylabel/.style={fill=white}]
    \coordinate (r) at (6,0);
    \coordinate (d) at (0,-2);
    \node (R) {$R$};
    \node (RC) at ($(R)+(r)$) {$R_{\CC}$};
    \node (AbR) at ($(R)+(d)$) {$R\modl\modl$};
    \node (AbRC) at ($(RC)+(d)$) {$R_{\CC}\modl\modl$};
    \node (AbRmodC) at ($(AbR)+(d)$) {$\frac{R\modl\modl}{\CC}$};
    \node (AbRCmodKer) at ($(AbRC)+(d)$) {$\frac{R_{\CC}\modl\modl}{\kernel( \ev{G} )}$};
    
    \draw[->,thick] (R) to (RC);
    \draw[->,thick] (R) to (AbR);
    \draw[->,thick] (RC) to (AbRC);
    \draw[->,thick] (AbR) to node[above]{$\ev{F}$} (AbRC);
    \draw[->,thick] (AbR) to (AbRmodC);
    \draw[->,thick] (AbRC) to (AbRCmodKer);
    \draw[->,thick] (AbRC) to node[mylabel]{$\ev{G}$}(AbRmodC); 
    \draw[->,thick] (AbRCmodKer) to node[mylabel]{$\simeq$} (AbRmodC);
    \end{tikzpicture}
\end{center}
Here, the lower two vertical arrows are the canonical quotient functors, and the equivalence is induced by the factorization of $\ev{G}$ described in \ref{para:explicit_serre_quotients}.
\end{lemma}
\begin{proof}
See \cite[Theorem 12.8.1]{PrestPSL}.
\end{proof}

\begin{definition}\label{definition:ring_of_def_scalars}
    Let $\M \in R\Modl$. Recall that we denote by $\ff$ the forgetful functor $R\Modl \rightarrow \Ab$.
    Then the endomorphism ring
    \[
    R_{\M} := \End_{\Beh{R}{M}}(\asBeh{\ff}{\M})
    \]
    is called the \textbf{ring of definable scalars} of $\M$.
    Note that $R_{\M}$ can actually be seen as an $R$-algebra, since we have a functor $R \rightarrow \Beh{R}{M}$ by \ref{lemma:data_of_beh} that directly corresponds to a ring map
    \[
    R \rightarrow R_{\M},
    \]
    which we call the \textbf{structure map} of $R_{\M}$ as an $R$-algebra.
\end{definition}

\begin{remark}
In \cite[Section 6]{PrestPSL}, the ring $R_{\M}$ is described in terms of pp-formulas. Furthermore, it is proven \cite[Section 12.8]{PrestPSL} that this coincides with the description that we used as our definition.
\end{remark}

\begin{construction}\label{construction:modules_over_ring_of_def_scalars}
Let $\M \in R\Modl$.
Then we can construct an $R_{\M}$-module, which we denote by $_{R_{\M}}\M$, that fits into the following commuative (up to natural isomorphism) diagram:
\begin{center}
    \begin{tikzpicture}[mylabel/.style={fill=white}, baseline=(R)]
    \coordinate (r) at (3,0);
    \coordinate (d) at (0,-1.5);
    \node (R) {$R$};
    \node (RM) at ($(R)+2*(d)$) {$R_{\M}$};
    \node (Beh) at ($(R)+(r)+(d)$) {$\Beh{R}{\M}$};
    \node (Ab) at ($(Beh) + (r)$) {$\Ab$};
    \draw[->,thick] (R) to (Beh);
    \draw[->,thick] (R) to (RM);
    \draw[->,thick] (RM) to (Beh);
    \draw[->,thick] (Beh) to (Ab);
    \draw[->,thick,out=0,in=180-45] (R) to node[below]{$\M$} (Ab);
    \draw[->,thick,out=0,in=180+45] (RM) to node[above]{$_{R_{\M}}\M$}(Ab);
    \end{tikzpicture}
\end{center}
How to read this diagram: from \ref{lemma:data_of_beh}, we get the commutative triangle at the top, from \ref{definition:ring_of_def_scalars}, we get the commutative triangle on the left. Then, we simply define $_{R_{\M}}\M$ as the composition of the functors
\[
R_{\M} \rightarrow \Beh{R}{\M} \rightarrow \Ab.
\]

\end{construction}

\begin{theorem}\label{theorem:compute_beh_via_ring_of_def_scalars}
Let $\M \in R\Modl$ and let $\ringmap: R \rightarrow R_{\M}$ be the structure map of the ring of definable scalars of $M$ (see \ref{definition:ring_of_def_scalars}). Let $_{R_{\M}}\M$ be the $R_{\M}$-module of \ref{construction:modules_over_ring_of_def_scalars}.
Then the restriction of $_{R_{\M}}\M$ along $\ringmap$ yields $\M$, and
\[
\ringmap_{\ast}: \Beh{R}{\M} \xrightarrow{\sim} \Beh{R_{\M}}{_{R_{\M}}\M}
\]
is an equivalence of categories.
\end{theorem}
\begin{proof}
The restriction of $_{R_{\M}}\M$ along $\ringmap$ yields $\M$ due to the commutativity of the diagram in \ref{construction:modules_over_ring_of_def_scalars}.

Next, we apply \ref{lemma:prest_equivalence_ring_of_def_scalars} in the case where $\CC := \kernel( \ev{\M} )$. In the notation of \ref{lemma:prest_equivalence_ring_of_def_scalars}, this yields $R_{\CC} = R_{\M}$, $\frac{R\modl\modl}{\CC} \simeq \Beh{R}{\M}$ and $\frac{R_{\CC}\modl\modl}{\kernel( \ev{G} )} \simeq \Beh{R_{\M}}{_{R_{\M}}\M}$.
Now, it follows from \ref{lemma:prest_equivalence_ring_of_def_scalars} that we have a commutative diagram (up to natural isomorphism)
\begin{equation}\label{equation:proof_equiv_ring_of_def_scalars}
    \begin{tikzpicture}[mylabel/.style={fill=white}, baseline=(R)]
    \coordinate (r) at (5,0);
    \coordinate (d) at (0,-1);
    \node (R) {$R$};
    \node (S) at ($(R)+0.5*(r)+0.5*(d)$) {$R_{\M}$};
    \node (A) at ($(R)+(r)-(d)$) {$\Beh{R}{\M}$};
    \node (B) at ($(R)+(r)+(d)$) {$\Beh{R_{\M}}{_{R_{\M}}M}$};
    \node (Ab) at ($(R)+2*(r)$) {$\Ab$};
    \draw[->,thick] (R) to node[mylabel]{$\ringmap$} (S);
    \draw[->,thick] (R) to (A);
    \draw[->,thick] (S) to (B);
    \draw[->,thick] (A) to (Ab);
    \draw[->,thick] (B) to (Ab);
    \draw[->,thick] (B) to node[mylabel]{$\simeq$}(A);
    \end{tikzpicture}
\end{equation}
Since $\ringmap_{\ast}$ is characterized as the unique functor (up to natural isomorphism) that renders the diagram
\begin{center}
    \begin{tikzpicture}[mylabel/.style={fill=white}, baseline=(R)]
    \coordinate (r) at (5,0);
    \coordinate (d) at (0,-1);
    \node (R) {$R$};
    \node (S) at ($(R)+0.5*(r)+0.5*(d)$) {$R_{\M}$};
    \node (A) at ($(R)+(r)-(d)$) {$\Beh{R}{M}$};
    \node (B) at ($(R)+(r)+(d)$) {$\Beh{R_{\M}}{_{R_{\M}}M}$};
    \node (Ab) at ($(R)+2*(r)$) {$\Ab$};
    \draw[->,thick] (R) to node[mylabel]{$\ringmap$} (S);
    \draw[->,thick] (R) to (A);
    \draw[->,thick] (S) to (B);
    \draw[->,thick] (A) to (Ab);
    \draw[->,thick] (B) to (Ab);
    \draw[->,thick,dashed] (A) to node[right]{$\ringmap_{\ast}$}(B);
    \end{tikzpicture}
\end{center}
commutative up to natural isomorphism (see \ref{construction:functor_change_of_rings}), $\ringmap_{\ast}$ has to be given by the inverse of the vertical arrow in \eqref{equation:proof_equiv_ring_of_def_scalars}. In particular, $\ringmap_{\ast}$ is an equivalence.
\end{proof}

%% file: example.tex
\begin{ex}\label{example:delay_differential_systems}
In this example, we study delay-differential systems with constant coefficients and a polynomial signal space.
Let $R := K[\sigma, \partial]$ be the polynomial ring in two indeterminates $\sigma$, $\partial$ over a field $K$ of characteristic $0$.
Then the set of polynomials with coefficients in $K$ in one indeterminate $M := K[t]$ becomes an $R$-module by letting $\partial$ act via differentiation
\[
\partial(p) := \frac{d p}{d t}
\]
and by letting $\sigma$ act via a unit shift
\[
(\sigma p)(t) := p(t + 1)
\]
for all $n \in \Nzero$ and $p \in M$. This gives a well-defined module structure since the two operators commute:
\begin{align*}
(\partial \sigma) p(t) &= \partial (p(t+1)) \\
&= \frac{d p}{d t}(t+1) \\
&=\sigma (\frac{d p}{d t}(t)) = (\sigma \partial) p(t)
\end{align*}
for all $p(t) \in M$.

It is our goal to describe the category $\Beh{R}{{_R}M}$.
Note that ${_R}M$ is not an fp-injective $R$-module and not an fp-cogenerator: for example, the following non-zero mono
\[
R/\langle \partial \rangle \xrightarrow{1 \mapsto (\sigma - 1) } R/\langle  \partial \rangle
\]
is mapped to the zero map (which is not an epi)
\[
K \xleftarrow{0} K
\]
via $\Hom(-,{_R}M)$, where we compute
\[
\Hom( R/\langle \partial \rangle, {_R}M ) \simeq \{ p \in {_R}M \mid \partial p = 0 \} = K.
\]
Since $\partial$ does not act via automorphisms on ${_R}M$, the same counter example remains valid if we would replace $R$ by the localization $S^{-1}R$ w.r.t. the set $S$ that consists of all $s \in R$ which act via automorphisms on ${_R}M$. 
However, we will see that $M$ becomes an injective module over its ring of definable scalars (see \ref{definition:ring_of_def_scalars}).

For this, we first note that $M$ is also a module over the ring of formal power series
$P := K[[\partial]]$ in one indeterminate $\partial$, where the action is given by
\[
(\sum_{i=0}^{\infty}a_i \partial^i) p:= \sum_{i=0}^{\infty}a_i \frac{d^i p}{d t^i}.
\]
for any given power series $(\sum_{i=0}^{\infty}a_i \partial^i) \in P$. Note that this expression is well-defined since $\frac{d^i p}{d t^i}$ eventually vanishes.
Next, we can reexpress the action of $\sigma$ by the element
\[
\exp( \partial ) := \sum_{i = 0}^{\infty}\frac{1}{i!} \partial^i \in P
\]
since
\begin{align*}
\exp(\partial)(t^n) &= \sum_{i = 0}^n\frac{1}{i!}\partial^i(t^n) \\
&= \sum_{i = 0}^n\frac{1}{i!}\frac{n!}{(n-i)!}t^{n-i} \\
&= \sum_{i = 0}^n \binom{n}{i}t^{n-i} = (t+1)^n = \sigma(t^n).
\end{align*}
We get a ring homomorphism
\begin{align*}
\ringmap: R &\rightarrow P \\
\partial &\mapsto \partial \\
\sigma &\mapsto \exp(\partial)
\end{align*}
and if we restrict $_{P}M$ along $\ringmap$, we get back $_{R}M$.
It follows from \ref{construction:functor_change_of_rings} that we get a faithful functor
\[
\ringmap_{\ast}: \Beh{R}{_R{M}} \rightarrow \Beh{\hat{R}}{_{P}M}.
\]
The ring $P$ is a local principal ideal domain with maximal ideal given by $\langle \partial \rangle$. Moreover,
the module ${_P}M$ is an injective cogenerator, since we have an isomorphism of $P$-modules
\begin{align*}
    {_P}M &\rightarrow \frac{\mathrm{Quot}( K[[\partial]] )}{K[[\partial]]} \\
    t^i &\mapsto i! \cdot \frac{1}{\partial^i}
\end{align*}
where $\mathrm{Quot}( K[[\partial]] )$ denotes the field of fractions of $K[[\partial]]$, and the module on the right hand side is known to be an injective cogenerator.
It follows from \ref{corollary:beh_rmodop_equivalence} that we have an equivalence of categories
\[
\Beh{\hat{R}}{_{P}M} \simeq (P\modl)^{\op}.
\]
In particular, 
\[
\End( \asBeh{\ff}{_{P}M} ) \cong \End_{(P\modl)^{\op}}( {_P}P ) \cong P.
\]
Since $\ringmap_{\ast}$ is faithful, it follows that we can compute the ring
\[
D := \End( \asBeh{\ff}{{_R}M} ) \subseteq \End( \asBeh{\ff}{_{P}M} ) \cong P
\]
of definable scalars of ${_R}M$ as a $K$-subalgebra of $P$.
We claim that $D$ satisfies the following three properties:
\begin{enumerate}
    \item The elements $\partial$ and $\exp({\partial})$ are in $D$.
    \item If $p \in D$ is invertible in $P$, then $p^{-1} \in D$.
    \item Whenever $p\partial \in D$ for $p \in P$, then $p \in D$.
\end{enumerate}
The first property is clear, since $\partial$ and $\exp(\partial)$ in $P$ correspond to $\partial$ and $\sigma$ in $R$.
The second property follows from \ref{lemma:desirable_features_Beh_monoepiiso}.
The third property follows from \ref{lemma:desirable_features_create_mor_via_epis} applied to the commutative triangle
\begin{center}
\begin{tikzpicture}[baseline=(base)]
        \coordinate (r) at (1.5,0);
        \coordinate (d) at (0,-1);
        \node (A) {$\ff({M})$};
        \node (B) at ($(A)+(r)+(d)$) {$\ff({M})$};
        \node (C) at ($(A) + 2*(r)$) {$\ff({M})$};
        \node (base) at ($(A) + 0.5*(d)$) {};
        \draw[->>,thick] (B) to node[below]{$\partial$} (A);
        \draw[->,thick] (B) to node[below,xshift=0.5em]{$p\partial$} (C);
        \draw[->, thick] (A) to node[above]{$p$} (C);
\end{tikzpicture}
\end{center}
Note that
\[
\partial: \ff(M) \twoheadrightarrow \ff(M)
\]
is surjective since every function in $M$ has an antiderivative.

It follows that $D$ is a local principal ideal domain with maximal ideal given by $\langle \partial \rangle$, since any element in $P$ can be written as $\partial^i \cdot q$ for $q \in P$ an invertible element, $i \in \Nzero$. The module ${_D}M$ is divisible since $\partial$ acts via a surjection on $M$. Since $D$ is a principal ideal domain, ${_D}M$ is an injective $D$-module. Since ${_D}M$ is clearly an fp-cogenerator, we end up with
\[
\Beh{R}{{_R}M} \simeq \Beh{D}{{_D}M} \simeq (D\modl)^{\op}
\]
where we combine \ref{theorem:compute_beh_via_ring_of_def_scalars} and \ref{corollary:beh_rmodop_equivalence}.
Thus, we managed to describe $\Beh{R}{{_R}M}$ as an opposite category of finitely presented modules over a subring of the ring of formal power series.
\end{ex}

\begin{ex}
Let $R$, $M$, and $D$ be as in \ref{example:delay_differential_systems}.
Our equivalence $\Beh{R}{{_R}M} \simeq (D\modl)^{\op}$ enables us to explain the equality of behaviors
\[
\{ p \in M \mid \partial p = 0 \} = K = \{ p \in M \mid (\sigma - 1) p = 0 \}
\]
by the fact that we have an equality of ideals
\[
\langle \partial \rangle = \langle \exp( \partial ) - 1 \rangle
\]
in $D$, since 
\[
\exp( \partial ) - 1 = \partial \cdot q
\]
for $q \in D$ an invertible element.
\end{ex}

\begin{remark}
Delay-differential systems were studied by Habets in \cite{Habets}. In the theoretical part of his paper, Habets assumes that $R$ is a commutative domain and $M \in R\Modl$ is a module such that each $q \neq 0$ in $R$ acts via a surjection on $\ff(M)$ \cite[Assumption 4.1]{Habets}. He introduces a subring of $\End( \ff(M) )$, the endomorphism ring of the underlying abelian group of $M$, that is given by all $r \in \End( \ff(M) )$ that fit into a commutative triangle
\begin{center}
\begin{tikzpicture}[baseline=(base)]
        \coordinate (r) at (1.5,0);
        \coordinate (d) at (0,-1);
        \node (A) {$\ff({M})$};
        \node (B) at ($(A)+(r)+(d)$) {$\ff({M})$};
        \node (C) at ($(A) + 2*(r)$) {$\ff({M})$};
        \node (base) at ($(A) + 0.5*(d)$) {};
        \draw[->,thick] (B) to node[left, xshift=-1em]{$(x \mapsto qx)$} (A);
        \draw[->,thick] (B) to node[right, xshift=1em]{$(x \mapsto px)$} (C);
        \draw[->, thick] (A) to node[above]{$r$} (C);
\end{tikzpicture}
\end{center}
with $q, p \in R$. 
It follows that all such $r \in \End( \ff(M) )$ are also endomorphisms of $\asBeh{\ff}{M}$ by \ref{lemma:desirable_features_create_mor_via_epis}.
In particular, the ring proposed by Habets is a subring of the ring of definable scalars of $M$.
Habets uses this ring to solve the problem of system equivalence in his setup, i.e., the problem whether two behaviors are mutually included in each other.
\end{remark}

%% file: controllability.tex
In this section, we define the controllability of an abstract behavior.
In \ref{example:controllable_and_ext1} we show that our definition yields the well-established notion of controllability in contexts studied in algebraic systems theory. For an overview of the development of this notion in algebraic systems theory, we refer to \cite[Remark 4.3]{Robertz15}.

Moreover, we also define the observability of an abstract behavior as an idea that suggests itself in our setup.
The observability of an abstract behavior given by a pp formula turns out to be a trivial notion (Example \ref{example:observability_trivial_for_beh}) which presumably is why it has not been studied as a mere property of a single abstract behavior in the context of algebraic systems theory, but rather as a property of a relation between behaviors.

\begin{definition}
Let $\asBeh{\mathcal{G}}{M} \in \Beh{R}{M}$. We call the elements of the underlying abelian group of $\asBeh{\mathcal{G}}{M}$, i.e., all elements $t \in \mathcal{G}(M)$, the \textbf{trajectories} of $\asBeh{\mathcal{G}}{M}$.
By abuse of notation, we will also address trajectories by writing an expression like
\[ t \in \asBeh{\mathcal{G}}{M}.\]
\end{definition}

\begin{para}\label{para:unconstrained_trajectories}
The abstract signal space $\asBeh{\ff}{\M}$ plays a prominent role.
We may regard its trajectories as being ``unconstrained'', i.e., the trajectories in $\asBeh{\ff}{\M}$ are exactly those trajectories that we want to specify as being ``physically possible in our universe in which we model our systems''.
\end{para}

\begin{para}\label{para:interpretation_controllable_observable}
If we think of a morphism within $\Beh{R}{M}$ as a valid way to transform the trajectories from one abstract behavior to another abstract behavior, then a morphism of type $\asBeh{\ff}{M} \rightarrow \asBeh{\mathcal{G}}{M}$ is a valid way to transform unconstrained trajectories (see \ref{para:unconstrained_trajectories}) to trajectories of $\asBeh{\mathcal{G}}{M}$. In this sense such morphisms are a means of ``controlling'' $\asBeh{\mathcal{G}}{M}$.

Dually, a morphism of type $\asBeh{\mathcal{G}}{M} \rightarrow \asBeh{\ff}{M}$ is a valid way to transform a trajectory of $\asBeh{\mathcal{G}}{M}$ to an unconstrained trajectory. In this sense such morphisms are a means of ``observing'' $\asBeh{\mathcal{G}}{M}$.
\end{para}

\begin{definition}
We call $t \in \asBeh{\mathcal{G}}{M}$ \textbf{controllable} if there exists an $s \in \asBeh{\ff}{M}$ and a morphism
$\alpha: \asBeh{\ff}{M} \rightarrow \asBeh{\mathcal{G}}{M}$ such that $\alpha(s) = t$.
\end{definition}

\begin{para}
Here, the intuitive understanding of a trajectory being controllable is that it can be obtained from some ``unconstrained input trajectory'' in the sense of \ref{para:interpretation_controllable_observable}.
\end{para}

\begin{definition}
We call $t \in \asBeh{\mathcal{G}}{M}$ \textbf{unobservable} if for all morphisms $\alpha:  \asBeh{\mathcal{G}}{M}\rightarrow \asBeh{\ff}{M}$ we have $\alpha(t) = 0$.
\end{definition}

\begin{para}
Here, the intuitive understanding of a trajectory being unobservable is that it can not be ``observed'' in the sense of \ref{para:interpretation_controllable_observable}.
\end{para}

\begin{remark}
Clearly, the set of controllable trajectories forms a subgroup.
Also, the set of unobservable trajectories forms a subgroup.
\end{remark}

\begin{definition}
We call the abelian subgroup 
\[
\{ t \in \mathcal{G}(M) \mid \text{$t$ is controllable} \}
\]
of $\mathcal{G}(M)$ the \textbf{controllable part} of $\asBeh{\mathcal{G}}{M}$.
Dually, we call the abelian quotient group 
\[
\mathcal{G}(M)/\{ t \in \mathcal{G}(M) \mid \text{$t$ is unobservable} \}
\]
of $\mathcal{G}(M)$ the \textbf{observable quotient} of $\asBeh{\mathcal{G}}{M}$.
\end{definition}

\begin{remark}
Note that by \ref{corollary:desirable_features_at_most_one_lift_subobj}, there is at most one way (up to equality of subobjects) to turn the controllable part of $\asBeh{\mathcal{G}}{M}$ from merely an abelian subgroup into a subobject within $\Beh{R}{M}$.
Dually by \ref{corollary:desirable_features_at_most_one_lift_quoobj}, there is at most one way (up to equality of quotient objects) to turn the observable quotient of $\asBeh{\mathcal{G}}{M}$ into a quotient object within $\Beh{R}{M}$.
It is not known to the author how to characterize those rings $R$ and modules $M$ such that the controllable part (or the observable quotient) can always be turned into an object in $\Beh{R}{M}$.
However, we will discuss cases in this section where this is indeed possible.
\end{remark}

\begin{definition}
An object $\asBeh{\mathcal{G}}{M} \in \Beh{R}{M}$ is called
\begin{enumerate}
    \item \textbf{controllable} if its controllable part is equal (as an abelian subgroup) to $\mathcal{G}(M)$,
    \item \textbf{finitely controllable} if there exists an $n \in \Nzero$ and an epimorphism
    \[
    (\asBeh{\ff}{M})^{\oplus n} \twoheadrightarrow \asBeh{\mathcal{G}}{M}
    \]
    in $\Beh{R}{M}$,
    \item \textbf{observable} if its observable factor is equal (as an abelian quotient group) to $\mathcal{G}(M)$,
    \item \textbf{finitely observable} if there exists an $n \in \Nzero$ and a monomorphism 
    \[
    \asBeh{\mathcal{G}}{M} \hookrightarrow (\asBeh{\ff}{M})^{\oplus n}
    \]
    in $\Beh{R}{M}$.
\end{enumerate}
\end{definition}

\begin{remark}
Clearly, if $\asBeh{\mathcal{G}}{M}$ is finitely controllable then it is controllable.
Also, if $\asBeh{\mathcal{G}}{M}$ is finitely observable then it is observable.
\end{remark}

\begin{ex}\label{example:observability_trivial_for_beh}
Let $\phi$ be a pp formula in $k \in \Nzero$ free variables over $R$ and let $M \in R\Modl$.
By \ref{para:beh_canonical_inclusion} we have a natural inclusion $\mathcal{B}( \phi, - ) \hookrightarrow \ff^{\oplus k}$.
It follows that we have a monomorphism
\[
\asBeh{\mathcal{B}( \phi, - )}{M} \hookrightarrow \asBeh{\ff^{\oplus k}}{M}
\]
which shows that abstract behaviors that are defined by a pp formula are always finitely observable.
\end{ex}

\begin{ex}
Let $R$ be a right coherent ring and let $M \in R\Modl$ be fp-faithfully flat. By \ref{corollary:equivalence_beh_modR}, we have
\[
\Beh{R}{M} \simeq \modr R.
\]
It follows that an abstract behavior is finitely observable in this setup if and only if its corresponding finitely presented module $N \in \modr R$ can be embedded into a finitely generated free $R$-module.
\end{ex}

\begin{remark}
Since there is a module $M \in R\Modl$ such that
\[
\Beh{R}{M} \simeq R\modl\modl
\]
(see \ref{remark:all_Serre_quotients_arise_as_beh}) it makes sense to apply the notions of this section to objects in $R\modl\modl$. In particular, we can ask for the controllable part and the observable quotient of an object $\mathcal{G} \in R\modl\modl$.
\end{remark}

\begin{theorem}\label{theorem:controllable_part_rmodmod}
Let $R$ be a right coherent ring.
The image of the counit of the adjunction described in \ref{theorem:codefect_adjunction}
\[
(\codefect( \mathcal{G} ) \otimes - ) \rightarrow \mathcal{G}
\]
is the controllable part of $\mathcal{G} \in R\modl\modl$.
In particular, the controllable part is again an object in $R\modl\modl$ and as such, it is finitely controllable.
\end{theorem}
\begin{proof}
Let $\alpha: \ff \rightarrow \mathcal{G}$ be a morphism.
Since $\ff \simeq (R \otimes -)$, the adjunction in \ref{theorem:codefect_adjunction} implies that there exists a morphism $\beta: R \rightarrow \codefect( \mathcal{G} )$ in $\modr R$ such that the following diagram commutes:
\begin{center}
       \begin{tikzpicture}[mylabel/.style={fill=white},baseline=($(A) + (d)$)]
        \coordinate (r) at (4,0);
        \coordinate (d) at (0,-1);
        \node (A) {$(\codefect( \mathcal{G} ) \otimes - )$};
        \node (B) at ($(A)+(r)+(d)$) {$\mathcal{G}$};
        \node (C) at ($(A)+2*(d)$) {$(R \otimes -)$};
        \draw[->,thick] (A) to (B);
        \draw[->,thick] (C) to node[below]{$\alpha$}(B);
        \draw[->,thick] (C) to node[left]{$\beta \otimes -$}(A);
        \end{tikzpicture}
\end{center}
It follows that the image of the counit contains the images of any given morphism of type $\ff \rightarrow \mathcal{G}$.

Conversely, since $\codefect(\mathcal{G})$ is finitely presented, there is an $n \in \Nzero$ and an epi in $\modr R$ of the form
\[
R^{n \times 1} \rightarrow \codefect(\mathcal{G}).
\]
The corresponding natural transformtion
\[
(R^{n \times 1} \otimes -)\rightarrow (\codefect(\mathcal{G}) \otimes -)
\]
is an epi in $R\modl\modl$ since taking the tensor product is right exact. It follows that the image of the counit only consists of images of morphisms of type $\ff \rightarrow \mathcal{G}$.
The claim follows.
\end{proof}

\begin{theorem}
Let $R$ be a left coherent ring.
The coimage of the unit of the adjunction described in \ref{theorem:defect_adjunction}
\[
\mathcal{G} \rightarrow \Hom(\defect( \mathcal{G} ), - )
\]
is the observable quotient of $\mathcal{G} \in R\modl\modl$.
In particular, the observable quotient is again an object in $R\modl\modl$ and as such, it is finitely observable.
\end{theorem}
\begin{proof}
Let $\alpha: \mathcal{G} \rightarrow \ff$ be a morphism.
Since $\ff \simeq \Hom(R, -)$, the adjunction in \ref{theorem:defect_adjunction} implies that there exists a morphism $\beta: R \rightarrow \defect( \mathcal{G} )$ in $R\modl$ such that the following diagram commutes:
\begin{center}
       \begin{tikzpicture}[mylabel/.style={fill=white},baseline=($(A) + (d)$)]
        \coordinate (r) at (4,0);
        \coordinate (d) at (0,-1);
        \node (A) {$\Hom(\defect( \mathcal{G} ), - )$};
        \node (B) at ($(A)+(r)+(d)$) {$\mathcal{G}$};
        \node (C) at ($(A)+2*(d)$) {$\Hom(R, -)$};
        \draw[<-,thick] (A) to (B);
        \draw[<-,thick] (C) to node[below]{$\alpha$}(B);
        \draw[<-,thick] (C) to node[left]{$\Hom(\beta, -)$}(A);
        \end{tikzpicture}
\end{center}
It follows that the kernel of the counit lies in the kernel of every possible $\alpha$ of type $\mathcal{G} \rightarrow \ff$, i.e., it can only contain unobservable trajectories.

Conversely, since $\defect(\mathcal{G})$ is finitely presented, there is an $n \in \Nzero$ and an epi in $R\modl$ of the form
\[
R^{1 \times n} \rightarrow \defect(\mathcal{G}).
\]
The corresponding natural transformation
\[
\Hom(\defect(\mathcal{G}), -) \rightarrow \Hom(R^{1 \times n}, - )
\]
is a mono in $R\modl\modl$ since hom functors are left exact. It follows that the kernel of the unit exactly consists of the unobservable trajectories.
The claim follows.
\end{proof}

\begin{remark}
Since the canonical quotient functor
\[
R\modl\modl \rightarrow \Beh{R}{M}
\]
is exact for all $M \in R\Modl$, the controllable part $\mathcal{H}$ of an object $\mathcal{G} \in R\modl\modl$ is mapped to a subobject
\[
\asBeh{\mathcal{H}}{M} \hookrightarrow \asBeh{\mathcal{G}}{M}
\]
in $\Beh{R}{M}$ whose underlying abelian group is a subgroup of the controllable part of $\asBeh{\mathcal{G}}{M}$.
Dually, the observable quotient $\mathcal{H}$ of an object $\mathcal{G} \in R\modl\modl$ is mapped to a quotient object
\[
\asBeh{\mathcal{H}}{M} \twoheadleftarrow \asBeh{\mathcal{G}}{M}
\]
in $\Beh{R}{M}$ whose underlying abelian group is a quotient of the observable quotient of $\asBeh{\mathcal{G}}{M}$.
\end{remark}

\begin{theorem}\label{theorem:computation_of_controllable_part}
Let $R$ be a right coherent ring.
Let $M \in R\Modl$, $\asBeh{\mathcal{G}}{M} \in \Beh{R}{M}$ and let
\[
R\modl\modl \rightarrow \Beh{R}{M}
\]
be the canonical quotient functor.
Assume that the canonical quotient functor admits a right adjoint 
\[
    F: \Beh{R}{M} \rightarrow R\modl\modl.
\]
If $\mathcal{H} \hookrightarrow F(\asBeh{\mathcal{G}}{M}) \in R\modl\modl$ is the controllable part of $F(\asBeh{\mathcal{G}}{M})$, then the image of the composite morphism
\begin{equation}\label{equation:mapped_controllable_part}
    \asBeh{\mathcal{H}}{M} \hookrightarrow \asBeh{F(\asBeh{\mathcal{G}}{M})}{M} \rightarrow \asBeh{\mathcal{G}}{M}
\end{equation}
is the controllable part of $\asBeh{\mathcal{G}}{M}$, where the second morphism in the composition is the counit of the adjunction.
\end{theorem}
\begin{proof}
Since $R$ is right coherent, we can apply \ref{theorem:controllable_part_rmodmod} and deduce that $\mathcal{H}$ is finitely controllable.
It follows that the image of the morphism in \eqref{equation:mapped_controllable_part} can only consist of controllable trajectories. Thus, it suffices to show that conversely, any morphism
\[
\alpha: \asBeh{\ff}{M} \rightarrow \asBeh{\mathcal{G}}{M}
\]
factors over the morphism in \eqref{equation:mapped_controllable_part}.
Let
\[
\alpha^{\sharp}: \ff \rightarrow F(\asBeh{\mathcal{G}}{M})
\]
be the morphism corresponding to $\alpha$ via the given adjunction. Then $\alpha^{\sharp}$ factors over the controllable part $\mathcal{H} \hookrightarrow F(\asBeh{\mathcal{G}}{M})$.
Next, we consider the following diagram:
\begin{center}
       \begin{tikzpicture}[mylabel/.style={fill=white},baseline=($(A) + (d)$)]
        \coordinate (r) at (4,0);
        \coordinate (d) at (0,-1);
        \node (A) {$\asBeh{\ff}{M}$};
        \node (B) at ($(A)+(r)$) {$\asBeh{F(\asBeh{\mathcal{G}}{M})}{M}$};
        \node (C) at ($(A)+2*(r)$) {$\asBeh{\mathcal{G}}{M}$};
        \node (D) at ($(A)+(d)+(r)$) {$\asBeh{\mathcal{H}}{M}$};
        \draw[->,thick] (A) to node[above]{$\alpha^{\sharp}_{M}$} (B);
        \draw[->,thick] (B) to (C);
        \draw[->,thick] (A) to (D);
        \draw[right hook ->,thick] (D) to (B);
        \draw[->,thick,out=25,in=180-25] (A) to node[below]{$\alpha$}(C);
        \end{tikzpicture}
\end{center}
The lower triangle commutes since it is 
the application of the canonical quotient functor to the factorization of $\alpha^{\sharp}$, and the upper part commutes by the given adjunction.
The claim follows.
\end{proof}

\begin{theorem}
Let $R$ be a left coherent ring.
Let $M \in R\Modl$, $\asBeh{\mathcal{G}}{M} \in \Beh{R}{M}$ and let
\[
R\modl\modl \rightarrow \Beh{R}{M}
\]
be the canonical quotient functor.
Assume that the canonical quotient functor admits a left adjoint 
\[
    F: \Beh{R}{M} \rightarrow R\modl\modl.
\]
If $\mathcal{H} \twoheadleftarrow F(\asBeh{\mathcal{G}}{M}) \in R\modl\modl$ is the observable quotient of $F(\asBeh{\mathcal{G}}{M})$, then the coimage (i.e., the domain modulo the kernel) of the composite morphism
\begin{equation}
    \asBeh{\mathcal{H}}{M} \twoheadleftarrow \asBeh{F(\asBeh{\mathcal{G}}{M})}{M} \leftarrow \asBeh{\mathcal{G}}{M}
\end{equation}
is the observable quotient of $\asBeh{\mathcal{G}}{M}$, where the first morphism in the composition is the unit of the adjunction.
\end{theorem}
\begin{proof}
Analogous to \ref{theorem:computation_of_controllable_part}.
\end{proof}

\begin{ex}\label{example:controllable_and_ext1}
Let $R$ be a left and right coherent ring, and let $M \in R\Modl$ be an fp-injective fp-cogenerator.
From \ref{corollary:beh_rmodop_equivalence} it follows that
\[
\Beh{R}{M} \simeq (R\modl)^{\op}.
\]
A module $N \in R\modl$ corresponds via this equivalence to the abstract behavior $\asBeh{\Hom(N,-)}{M} \in \Beh{R}{M}$.
We want to compute its controllable part.
Since we have the adjunction $(\yoneda \vdash \defect)$ by \ref{theorem:defect_adjunction} and since $R$ is right coherent, we can apply \ref{theorem:computation_of_controllable_part}, which means we can compute the controllable part of $\Hom(N,-) \in R\modl\modl$ and map it down to $(R\modl)^{\op}$ via $\defect$.
By \ref{theorem:controllable_part_rmodmod}, the controllable part of $\Hom(N,-)$ is given by the image of a natural transformation of type
\begin{equation}\label{equation:covdefect_to_hom}
\big(\codefect( {\Hom(N,-)} ) \otimes - \big) \rightarrow {\Hom(N,-)}.
\end{equation}
By \ref{example:covariant_defect}, we can compute
\[
\codefect( {\Hom(N,-)} ) \cong \Hom(N,R) \in \modr R.
\]
Now, let
\begin{center}
  \begin{tikzpicture}[mylabel/.style={fill=white}]
      \coordinate (r) at (3,0);
      
      \node (A) {$R^{1\times m}$};
      \node (B) at ($(A)+(r)$) {$R^{1 \times n}$};
      \node (C) at ($(B) + (r)$) {$N$};
      \node (D) at ($(C) + (r)$) {$0$};
      
      \draw[->,thick] (A) to node[above]{$A$} (B);
      \draw[->,thick] (B) to (C);
      \draw[->,thick] (C) to (D);
  \end{tikzpicture}
\end{center}
be a presentation of $N$.
From this, we get an exact sequence in $\modr R$
\begin{center}
  \begin{tikzpicture}[mylabel/.style={fill=white}]
      \coordinate (r) at (2,0);
      
      \node (A) {$0$};
      \node (B) at ($(A)+(r)$) {$N'$};
      \node (C) at ($(B) + (r)$) {$R^{m \times 1}$};
      \node (D) at ($(C) + (r)$) {$R^{n \times 1}$};
      \node (E) at ($(D) + (r)$) {$\Hom(N,R)$};
      \node (F) at ($(E) + (r)$) {$0$};

      \draw[->,thick] (F) to (E);
      \draw[->,thick] (E) to (D);
      \draw[->,thick] (D) to node[above]{$A$} (C);
      \draw[->,thick] (C) to (B);
      \draw[->,thick] (B) to (A);
  \end{tikzpicture}
\end{center}
where $N' \in \modr R$ denotes the so-called \textbf{Auslander transpose} of $N$.
Let $B \in R^{n \times p}$, $C \in R^{p \times q}$ be matrices for $p,q \in \Nzero$ such that
\begin{center}
  \begin{tikzpicture}[mylabel/.style={fill=white}]
      \coordinate (r) at (2,0);
      
      \node (A) {$0$};
      \node (B) at ($(A)+(r)$) {$N'$};
      \node (C) at ($(B) + (r)$) {$R^{m \times 1}$};
      \node (D) at ($(C) + (r)$) {$R^{n \times 1}$};
      \node (E) at ($(D) + (r)$) {$R^{p \times 1}$};
      \node (F) at ($(E) + (r)$) {$R^{q \times 1}$};

      \draw[->,thick] (F) to node[above]{$C$} (E);
      \draw[->,thick] (E) to node[above]{$B$} (D);
      \draw[->,thick] (D) to node[above]{$A$} (C);
      \draw[->,thick] (C) to (B);
      \draw[->,thick] (B) to (A);
  \end{tikzpicture}
\end{center}
is exact. It follows that $C$ yields a presentation of $\Hom(N,R)$ and consequently, we have a presentation
\begin{center}
  \begin{tikzpicture}[mylabel/.style={fill=white}]
      \coordinate (r) at (3,0);
      
      \node (A) {$\ff^{\oplus q}$};
      \node (B) at ($(A)+(r)$) {$\ff^{\oplus p}$};
      \node (C) at ($(B) + (r)$) {$(\Hom(N,R) \otimes -)$};
      \node (D) at ($(C) + (r)$) {$0$};
      
      \draw[->,thick] (A) to node[above]{$C$} (B);
      \draw[->,thick] (B) to (C);
      \draw[->,thick] (C) to (D);
  \end{tikzpicture}
\end{center}
in $R\modl\modl$.
By \ref{example:defect}, we can compute
\[
\defect( (\Hom(N,R) \otimes -) ) \cong \kernel( R^{1 \times p} \xrightarrow{C} R^{1 \times q} )
\]
and the right hand side readily identifies with
\[
\kernel( R^{1 \times p} \xrightarrow{C} R^{1 \times q} ) \cong  \Hom(\Hom(N,R),R).
\]
It follows that the application of $\defect$ to \eqref{equation:covdefect_to_hom} yields the natural map of type
\[
N \rightarrow \Hom(\Hom(N,R),R).
\]
It follows that the coimage of this natural map corresponds to the controllable part of $\asBeh{\Hom(N,-)}{M}$.
In particular, $\asBeh{\Hom(N,-)}{M}$ is controllable if this natural map is injective.
We note that due to Auslander \cite[Proposition 6.3]{A} we have an exact sequence
\begin{center}
  \begin{tikzpicture}[mylabel/.style={fill=white}]
      \coordinate (r) at (3,0);
      
      \node (A) {$0$};
      \node (B) at ($(A)+0.5*(r)$) {$\Ext^1(N',R)$};
      \node (C) at ($(B) + 0.75*(r)$) {$N$};
      \node (D) at ($(C) + (r)$) {$\Hom(\Hom(N,R),R)$};
      \node (E) at ($(D) + (r)$) {$\Ext^2(N',R)$};
      \node (F) at ($(E) + 0.5*(r)$) {$0$};
      
      \draw[->,thick] (A) to (B);
      \draw[->,thick] (B) to (C);
      \draw[->,thick] (C) to (D);
      \draw[->,thick] (D) to (E);
      \draw[->,thick] (E) to (F);
  \end{tikzpicture}
\end{center}
which shows that we have controllability if and only if
\[
\Ext^1(N',R) \cong 0.
\]
But this is exactly the test for controllability introduced by Pommaret and Quadrat in \cite{PommaretQuadrat} in the context of noetherian left Ore integral domains.
\end{ex}

%% file: appendix.tex
In this appendix, we recall some elementary aspects of subobjects, quotient objects, and their interaction with functors.

\subsection*{Subobjects and quotient objects}

\begin{definition}
    Let $C$ be an object in a category $\CC$.
    Then we can form its \textbf{category of subobjects} $\Sub( C )$ as follows:
    an object in $\Sub( C )$ is given by a monomorphism in $\CC$ whose range is given by $C$.
    A morphism from
    $\alpha: A \hookrightarrow C$
    to
    $\beta: B \hookrightarrow C$
    in $\Sub( C )$ is given by a morphism $\gamma: A \rightarrow B$ such that the following diagram commutes:
    \begin{equation}\label{equation:subobj_triangle}
        \begin{tikzpicture}[mylabel/.style={fill=white}, baseline=($(A) + 0.5*(u)$)]
            \coordinate (r) at (2,0);
            \coordinate (u) at (0,2);
            \node (A) {$A$};
            \node (B) at ($(A)+2*(r)$) {$B$};
            \node (C) at ($(A)+(r)+(u)$) {$C$};
            \draw[->,thick] (A) to node[above]{$\gamma$} (B);
            \draw[right hook->,thick] (A) to node[above, xshift = -0.5em]{$\alpha$} (C);
            \draw[left hook->,thick] (B) to node[above, xshift = 0.5em]{$\beta$} (C);
        \end{tikzpicture}
    \end{equation}
    If there exists such a morphism $\gamma$, we say that the object defined by $\alpha$ \textbf{is included} in the object defined by $\beta$.
    Moreover, if two monos $\alpha$ and $\beta$ considered as objects in $\Sub( C )$ are isomorphic,
    then we say that they are \textbf{equal as subobjects}.
    
    Dually, we can form the category of \textbf{quotient objects} (also called \textbf{factor objects}) $\Quo( C )$ whose objects are epimorphisms
    with $C$ as their source. The corresponding dual picture looks as follows:
    \begin{center}
        \begin{tikzpicture}[mylabel/.style={fill=white}]
            \coordinate (r) at (2,0);
            \coordinate (u) at (0,2);
            \node (A) {$A$};
            \node (B) at ($(A)+2*(r)$) {$B$};
            \node (C) at ($(A)+(r)+(u)$) {$C$};
            \draw[->,thick] (B) to node[above]{$\gamma$} (A);
            \draw[->>,thick] (C) to node[above, xshift = -0.5em]{$\alpha$} (A);
            \draw[->>,thick] (C) to node[above, xshift = 0.5em]{$\beta$} (B);
        \end{tikzpicture}
    \end{center}
    Again, if two epis $\alpha$ and $\beta$ considered as objects in $\Quo( C )$ are isomorphic,
    then we say that they are \textbf{equal as quotient objects} (or alternatively \textbf{equal as factor objects}).
\end{definition}

\begin{lemma}\label{lemma:sub_and_quo}
    Let $C$ be an object in a category $\CC$.
    \begin{enumerate}
        \item The morphism sets in both $\Sub( C )$ and $\Quo( C )$ are either empty or singletons. Thus, both categories may be viewed as classes equipped with a preorder, i.e., a reflexive and transitive relation.
        \item If $\gamma \in \Hom_{\Sub(C)}( \alpha, \beta )$, then $\gamma$ is a mono in $\CC$.
        \item If $\gamma \in \Hom_{\Quo(C)}( \beta, \alpha )$, then $\gamma$ is an epi in $\CC$.
    \end{enumerate}
\end{lemma}
\begin{proof}
Let $\alpha, \beta, \gamma$ be the morphisms of the commutative triangle in \eqref{equation:subobj_triangle}.
Suppose given another morphism $\delta: A \rightarrow B$ such that $\beta \circ \gamma = \alpha = \beta \circ \delta$. Since $\beta$ is a mono, it follows that $\delta = \gamma$.
We can proceed dually for quotient objects.
For the second statement: Let $\tau_1, \tau_2: T \rightarrow A$ be two morphisms such that $\gamma \circ \tau_1 = \gamma \circ \tau_2$.
Then $\beta \circ \gamma \circ \tau_1 = \beta \circ \gamma \circ \tau_2$ and hence $\alpha \circ \tau_1 = \alpha \circ \tau_2$. Since $\alpha$ is a mono, we have $\tau_1 = \tau_2$ and thus $\gamma$ is a mono.
The third statement is dual to the second.
\end{proof}

\begin{remark}
More classically, subobjects are introduced as the isomorphism classes of the objects in $\Sub( \CC )$.
However, whenever we work concretely with subobjects in category theory, we often make use of a concrete monomorphism which represents the subobject, which is why the introduction of subobjects by means of the category $\Sub( \CC )$ appears to be more suitable.
See also \cite[Definition 4.6.8]{RiehlContext} for a modern introduction to category theory where subobjects are also directly treated by their defining monomorphisms.
\end{remark}

\subsection*{Functors on subobjects}

\begin{para}\label{remark:functor_on_subobjects}
If a functor $F: \CC \rightarrow \DC$ respects monomorphisms, then it induces a functor on the level of subobjects for every $C \in \AC$:
\begin{align*}
\Sub( C ) &\rightarrow \Sub( F(C) ) \\
(A \hookrightarrow C) &\mapsto (F(A) \hookrightarrow F(C))
\end{align*}
In particular, this applies to left exact functors between abelian categories.
\end{para}

\begin{para}[Test for subobject inclusion]\label{remark:test_for_subobjects}
Let $\AC$ be an additive category such that every subobject $(\alpha: A \hookrightarrow B) \in \Sub(B)$ for $B \in \AC$ arises as the kernel of some morphism $\delta: B \rightarrow D$ in $\AC$.
This holds for example in abelian categories.
Such a morphism $\delta$ yields an elementary test whether another subobject $(\beta: C \hookrightarrow B)$
is included in $(\alpha: A \hookrightarrow B)$: this is the case if and only if
\begin{equation}\label{equation:test_subobject_inclusion}
\delta \circ \beta = 0.
\end{equation}
This follows directly from the universal property of kernels. More generally, if $\beta: C \rightarrow B$ is any morphism, then there exists a $\gamma: C \rightarrow A$ that renders the triangle

\begin{center}
\begin{tikzpicture}[baseline=(base)]
        \coordinate (r) at (1.5,0);
        \coordinate (d) at (0,-1);
        \node (A) {$A$};
        \node (B) at ($(A)+(r)-(d)$) {$B$};
        \node (C) at ($(A) + 2*(r)$) {$C$};
        \node (base) at ($(A) + 0.5*(d)$) {};
        \draw[<-right hook,thick] (B) to node[above]{$\alpha$} (A);
        \draw[<-,thick] (B) to node[above]{$\beta$} (C);
        \draw[->,thick] (C) to node[above]{$\gamma$} (A);
\end{tikzpicture}
\end{center}
commutative if and only if the Equation in \eqref{equation:test_subobject_inclusion} holds.
\end{para}

\begin{lemma}\label{lemma:left_exact_faithful_yields_triangles}
Let $\AC$ be a category as in \ref{remark:test_for_subobjects}.
Let $F: \AC \rightarrow \BC$ be a faithful functor that respects kernels.
Suppose given the following diagram in $\AC$ with $\alpha$ a mono:
\begin{center}
\begin{tikzpicture}[baseline=(base)]
        \coordinate (r) at (1.5,0);
        \coordinate (d) at (0,-1);
        \node (A) {$A$};
        \node (B) at ($(A)+(r)-(d)$) {$B$};
        \node (C) at ($(A) + 2*(r)$) {$C$};
        \node (base) at ($(A) + 0.5*(d)$) {};
        \draw[<-right hook,thick] (B) to node[above]{$\alpha$} (A);
        \draw[<-,thick] (B) to node[above]{$\beta$} (C);
\end{tikzpicture}
\end{center}
If there exists a morphism $\gamma'$ in $\BC$ such that we have a commutative triangle
\begin{center}
\begin{tikzpicture}[baseline=(base)]
        \coordinate (r) at (1.5,0);
        \coordinate (d) at (0,-1);
        \node (A) {$F(A)$};
        \node (B) at ($(A)+(r)-(d)$) {$F(B)$};
        \node (C) at ($(A) + 2*(r)$) {$F(C)$};
        \node (base) at ($(A) + 0.5*(d)$) {};
        \draw[<-right hook,thick] (B) to node[above,xshift=-1em]{$F(\alpha)$} (A);
        \draw[<-,thick] (B) to node[above,xshift=1em]{$F(\beta)$} (C);
        \draw[<-, thick] (A) to node[above]{$\gamma'$} (C);
\end{tikzpicture}
\end{center}
then there exists a unique morphism $\gamma: C \rightarrow A$ in $\AC$ such that we have a commutative triangle
\begin{center}
\begin{tikzpicture}[baseline=(base)]
        \coordinate (r) at (1.5,0);
        \coordinate (d) at (0,-1);
        \node (A) {$A$};
        \node (B) at ($(A)+(r)-(d)$) {$B$};
        \node (C) at ($(A) + 2*(r)$) {$C$};
        \node (base) at ($(A) + 0.5*(d)$) {};
        \draw[<-right hook,thick] (B) to node[above]{$\alpha$} (A);
        \draw[<-,thick] (B) to node[above]{$\beta$} (C);
        \draw[->,thick] (C) to node[above]{$\gamma$} (A);
\end{tikzpicture}
\end{center}
and $F(\gamma ) = \gamma'$.
\end{lemma}
\begin{proof}
Let $\delta: B \rightarrow D$ be a morphism in $\AC$ whose kernel embedding is given by $\alpha$ (see \ref{remark:test_for_subobjects}).
Since $F$ respects kernels, $F(\delta): F(B) \rightarrow F(D)$ is a morphism whose kernel embedding is given by $F(\alpha)$.
By \ref{remark:test_for_subobjects}, we conclude that the existence of $\gamma'$ is equivalent to $F( \delta \circ \beta) = F(\delta) \circ F(\beta) = 0$. This in turn is equivalent to $\delta \circ \beta = 0$ since $F$ is faithful. Last, this is equivalent to the existence of $\gamma$ again by \ref{remark:test_for_subobjects}.
The uniqueness of $\gamma$ follows since $\alpha$ is a monomorphism.
\end{proof}

\begin{corollary}\label{corollary:subobjects_fully_faithful}
Let $\AC$ be a category as in \ref{remark:test_for_subobjects}.
Let $F: \AC \rightarrow \BC$ be a faithful functor that respects kernels and let $B \in \AC$.
Then the induced functor on subobjects (see \ref{remark:functor_on_subobjects}) 
\[
\Sub(B) \rightarrow \Sub( F(B) )
\]
is fully faithful, i.e., a subobject $A \hookrightarrow B$ is included in a subobject $C \hookrightarrow B$ if and only if $FA \hookrightarrow FB$ is included in $FC \hookrightarrow FB$.
\end{corollary}
\begin{proof}
This is a direct consequence of \ref{lemma:left_exact_faithful_yields_triangles}.
\end{proof}

\begin{remark}
All statements of this subsection have corresponding dual statements for quotient objects.
\end{remark}

\subsection*{Faithful functors}

\begin{remark}
We can test if an object $A$ in an additive category is isomorphic to the zero object by checking $\id_A = 0$.
\end{remark}

\begin{lemma}\label{lemma:faithful_implies_zero_test}
Let $F: \AC \rightarrow \BC$ be a functor between additive categories. 
If $F$ is faithful, then
\[
F(A) \cong 0 \text{\hspace{1em} if and only if \hspace{1em}} A \cong 0
\]
for all $A \in \AC$. The converse holds if $F$ is an exact functor between abelian categories.
\end{lemma}
\begin{proof}
We have $\id_{F(A)} = F(\id_A) = 0$ if and only if $\id_A = 0$.
Moreover, if $F$ is an exact functor between abelian categories, we can test whether a morphism $\alpha$ is zero by testing whether its image object $\image( \alpha )$ is zero.
\end{proof}

\begin{lemma}\label{lemma:faithful_monoepiiso_test}
Let $F: \AC \rightarrow \BC$ be a faithful and exact functor between abelian categories.
Suppose given a morphism $\alpha$ in $\AC$.
Then $\alpha$ is a mono/epi/iso in $\AC$ if and only if $F(\alpha)$ is a mono/epi/iso in $\BC$.
\end{lemma}
\begin{proof}
Being a mono/epi in an abelian category can be decided by testing whether the kernel/cokernel of a morphism is zero. Moreover, a morphism in an abelian category is an iso if and only if it is an epi and a mono.
Now, the claim follows from the exactness of $F$ and from \ref{lemma:faithful_implies_zero_test}.
\end{proof}